\documentclass[11pt,letter]{amsart}
\usepackage[
top=2.3cm, bottom=1.2cm, inner=2.4cm, outer=2.3cm,
marginparwidth=2cm %
]{geometry}

\usepackage{amssymb,latexsym, amsmath, amsxtra, mathrsfs, bm}

\usepackage[dvips]{graphics}
\usepackage{xypic,xcolor}
\usepackage{subcaption}
\usepackage{verbatim}

\usepackage[abs]{overpic}
\usepackage{hyperref}
\usepackage{mathtools, tikz}
\usetikzlibrary{positioning, shapes.geometric}

\DeclareMathAlphabet{\mathbbold}{U}{bbold}{m}{n}

\allowdisplaybreaks[4]

\makeatletter
\def\th@plain{%
	\thm@notefont{}%
	\itshape %
}
\def\th@definition{%
	\thm@notefont{}%
	\normalfont %
}
\makeatother

\theoremstyle{plain}
\newtheorem{theorem}{Theorem}[section]
\newtheorem*{theorem*}{Theorem}
\newtheorem*{conj*}{Conjecture}
\newtheorem{lemma}[theorem]{Lemma}
\newtheorem{prop}[theorem]{Proposition}
\newtheorem{cor}[theorem]{Corollary}
\newtheorem{thmx}{Theorem}

\theoremstyle{definition}
\newtheorem{definition}[theorem]{Definition}
\newtheorem{rem}[theorem]{Remark}

\theoremstyle{remark}
\newtheorem*{remark}{Remark}

\numberwithin{equation}{section}
\numberwithin{theorem}{section}
\numberwithin{table}{section}
\numberwithin{figure}{section}

\providecommand{\defn}[1]{\emph{#1}}

\renewcommand{\leq}{\leqslant}

\renewcommand{\geq}{\geqslant}

\newcommand{\diam}  {\operatorname{diam}}

\newcommand{\R}{\mathbb{R}}

\newcommand{\C}{\mathbb{C}}

\newcommand{\N}{\mathbb{N}}
\newcommand{\Z}{\mathbb{Z}}

\newcommand{\CDach}{\widehat{\mathbb{C}}}

\newcommand{\D}{\mathbb{D}}

\providecommand{\abs}[1]{\lvert#1\rvert}
\providecommand{\Absbig}[1]{\bigl|#1\bigr|}
\providecommand{\Absbigg}[1]{\biggl|#1\biggr|}

\providecommand{\norm}[1]{\|#1\|}
\providecommand{\Normbig}[1]{\bigl\|#1\bigr\|}

\renewcommand{\:}{\colon}

\newcommand{\mesh}{\operatorname{mesh}}

\newcommand{\CCC}{C}

\newcommand{\BBB}{B}

\newcommand{\PPP}{\mathcal{P}}

\newcommand{\MMM}{\mathcal{M}}

\newcommand{\SAA}{\mathscr{M}}

\newcommand{\SBB}{\mathscr{B}}

\newcommand{\vertiii}[1]{{\left\vert\kern-0.25ex\left\vert\kern-0.25ex\left\vert #1
		\right\vert\kern-0.25ex\right\vert\kern-0.25ex\right\vert}}

\newcommand{\wt}[1]{\widetilde{#1}}

\renewcommand{\=}{\coloneqq}

\renewcommand{\O}{\mathcal{O}}

\newcommand{\wh}{\widehat}

\newcommand{\cA}{\mathcal{A}}

\newcommand{\cC}{\mathcal{C}}
\renewcommand{\cD}{\mathcal{D}}

\newcommand{\cF}{\mathcal{F}}
\newcommand{\cG}{\mathcal{G}}

\newcommand{\cK}{\mathcal{K}}
\renewcommand{\cL}{\mathcal{L}}

\newcommand{\cO}{\mathcal{O}}
\newcommand{\cP}{\mathcal{P}}
\newcommand{\cQ}{\mathcal{Q}}
\renewcommand{\cR}{\mathcal{R}}
\newcommand{\cS}{\mathcal{S}}

\newcommand{\cU}{\mathcal{U}}
\newcommand{\cV}{\mathcal{V}}

\newcommand{\fA}{\mathfrak{A}}

\newcommand{\fC}{\mathfrak{C}}

\newcommand{\fK}{\mathfrak{K}}

\newcommand{\fT}{\mathfrak{T}}

\newcommand{\sB}{\mathscr{B}}

\newcommand{\hF}{\widehat{F}}

\newcommand{\hN}{\widehat{N}}

\newcommand{\hP}{\widehat{P}}

\newcommand{\hf}{\widehat{f}}
\newcommand{\hg}{\widehat{g}}

\newcommand{\hmu}{\widehat{\mu}}
\newcommand{\hnu}{\widehat{\nu}}

\newcommand{\hvarphi}{\widehat{\varphi}}

\newcommand{\ox}{\overline{x}}

\newcommand{\oz}{\overline{z}}

\newcommand{\ophi}{\overline{\phi}}

\newcommand{\udU}{\underline{U}}

\newcommand{\udx}{\underline{x}}
\newcommand{\udy}{\underline{y}}
\newcommand{\udz}{\underline{z}}

\newcommand{\tE}{\widetilde{E}}
\newcommand{\tF}{\widetilde{F}}

\newcommand{\tu}{\widetilde{u}}
\newcommand{\tv}{\widetilde{v}}
\newcommand{\tw}{\widetilde{w}}

\newcommand{\tepsilon}{\tilde{\epsilon}}

\newcommand{\tpi}{\widetilde{\pi}}

\newcommand{\tphi}{\widetilde{\phi}}
\newcommand{\tvarphi}{\widetilde{\varphi}}

\newcommand{\diff}{\,\mathrm{d}}

\providecommand{\zeroto}[1]{[\![#1]\!]}  	%
\providecommand{\oneto}[1]{]\!]#1]\!]}  	%

\providecommand{\vect}[3]{#1|_{#2}^{#3}}

\renewcommand{\limsup}{\varlimsup}

\newcommand{\tpk}{\fK}

\newif\ifFirstInitial
\FirstInitialfalse  %

\begin{document}
	\title[Thermodynamic formalism for correspondences]{Thermodynamic formalism for correspondences}
	\author{Xiaoran~Li \and Zhiqiang~Li \and Yiwei~Zhang}
	\thanks{Z.~Li and X.~Li were partially supported by NSFC Nos.~12471083, 12101017, 12090010, 12090015, and BJNSF No.~1214021. X.~Li was also supported by Peking University Funding Nos.~7101303303 and~6201001891. Y.~Zhang was partially supported by NSFC Nos.~12161141002 and 12271432, and USTC--AUST Math Basic Discipline Research Center.}
	\address{Xiaoran~Li, School of Mathematical Sciences, Peking University, Beijing 100871, China}
	\email{2000010758@alumni.pku.edu.cn}
	\address{Zhiqiang~Li, School of Mathematical Sciences \& Beijing International Center for Mathematical Research, Peking University, Beijing 100871, China}
	\email{zli@math.pku.edu.cn}
	\address{Yiwei~Zhang, Department of Mathematics \& SUSTech International Center for Mathematics, Southern University of Science and Technology, Shenzhen, Guangdong 518055, China}
	\email{zhangyw@sustech.edu.cn}

	

	\begin{abstract}
		In this article, we investigate the Variational Principle and develop thermodynamic formalism for correspondences. We define the measure-theoretic entropy for transition probability kernels and topological pressure for correspondences. Based on these two notions, we establish the following results: 
		
		The Variational Principle holds and equilibrium states exist for continuous potential functions, provided that the correspondence satisfies some expansion property called forward expansiveness. If, in addition, the correspondence satisfies the specification property and the potential function is Bowen summable, then the equilibrium state is unique. On the other hand, for a distance-expanding, open, strongly transitive correspondence and a H\"{o}lder continuous potential function, there exists a unique equilibrium state, and the backward orbits are equidistributed. Furthermore, we investigate the Variational Principle for general correspondences.
		
		In complex dynamics, we establish the Variational Principle for the Lee--Lyubich--Markorov--Mukherjee anti-holomorphic correspondences, which are matings of some anti-ho\-lo\-mor\-phic rational maps with anti-Hecke groups and are not forward expansive. We also show a Ruelle--Perron--Frobenius theorem for a family of hyperbolic holomorphic correspondences of the form $\boldsymbol{f}_c (z)= z^{q/p}+c$.
	\end{abstract}
	
	\maketitle
	
	\setcounter{tocdepth}{1}
	\tableofcontents
	
	\section{Introduction}     \label{sct_Introduction}
		
	A \emph{correspondence} $T$ on a compact metric space $X$ is a map from $X$ to the set consisting of all non-empty closed subsets of $X$ with the property that the set $\bigl \{ (x_1 ,  x_2)\in X^2  :  x_2 \in T(x_1) \bigr\}$ is closed in $X^2$. As a natural generalization of (single-valued) continuous maps, correspondences appear abundantly in control theory \cite{Po21}, differential games \cite{Pe93}, mathematical economics and game theory \cite{CPMP08}, qualitative physics and viability \cite{Fo88}, and continuous selections \cite{Mic561, Mic562, Mic57}. To quote from the monograph of {\ifFirstInitial J.~P.~\fi}Aubin and {\ifFirstInitial H.~\fi}Frankowska \cite{AF90}: ``Who needs set-valued analysis? Everyone, we are tempted to say.'' Indeed, they listed eight famous examples of correspondences studied by {\ifFirstInitial J.\fi}~Hadamard, {\ifFirstInitial J.~\fi}von~Neumann, {\ifFirstInitial K.~\fi}Kuratowski, {\ifFirstInitial E.~\fi}Michael, {\ifFirstInitial T.~\fi}Wa\.{z}ewski, {\ifFirstInitial V.~V.~\fi}Filippov, and many other mathematicians, ranging from eight different mathematical subjects mentioned in \cite[Introduction]{AF90}.
	
	Studies on correspondences\footnote{There are notions in the literature related to correspondences, such as upper semi-continuous set-valued functions in \cite{KT17}, set-valued maps in \cite{RT18}, and closed relations in \cite{MA99}. Our notion of correspondence coincides with the first one but differs slightly from the other two.} date back to investigations on the ill-posed problems (for partial differential equations) in the sense of {\ifFirstInitial J.~\fi}Hadamard \cite{Ha02}. Here, by ill-posed problems, we mean that the existence of a solution or the uniqueness of the solution fails for some choice of data. This was indeed noticed during the first three decades of the 20th century by founders of ``Functional Calculus'', such as {\ifFirstInitial P.~\fi}Painlev\'{e}, {\ifFirstInitial F.~\fi}Hausdorff, {\ifFirstInitial G.~\fi}Bouligand, and {\ifFirstInitial K.~\fi}Kuratowski to quote only a few. In his remarkable book \emph{Topologie} \cite{Ku58}, {\ifFirstInitial K.~\fi}Kuratowski gave set-valued maps their proper status. Since then, the study of correspondences (known as set-valued analysis) has been increasing rapidly. Many fundamental concepts of single-valued analysis, such as limits, differentiation, integral, and fixed point theorems have been adapted to the set-valued realm; see the monograph \cite{AF90} and references therein.
	
	Other than general correspondences, the study of holomorphic and anti-holomorphic correspondences attracts its independent interests in complex dynamics. %

	The study of (anti-)holomorphic correspondences dates back at least to {\ifFirstInitial P.~\fi}Fatou \cite{Fa29}. Indeed, {\ifFirstInitial P.~\fi}Fatou observed similarities between limit sets of Kleinian groups and Julia sets of rational maps in the 1920s, and proposed the following question \cite{Fa29}:
	
	``L'analogie remarqu\'{e}e entre les ensembles de points limites des groupes Klein\'{e}ens-et ceux qui sont constitu\'{e}s par les fronti\`{e}res des r\'{e}gions de convergence des it\'{e}r\'{e}es d'une fonction rationnelle ne para\^{i}t d'ailleurs pas fortuite et il serait probablement possible d'en faire la synt\`{e}se dans une th\'{e}orie g\'{e}n\'{e}rale des groupes discontinus des substitutions algr\'{e}briques.''
	
	About the analogy between Kleinian groups and rational maps, {\ifFirstInitial D.~P.~\fi}Sullivan discovered deep connections between the iteration theory of rational maps and the theory of Kleinian groups (see \cite{Su85}), which became known as \defn{Sullivan's dictionary}\index{Sullivan's dictionary}. Since then, there have been considerable efforts to draw direct connections between these two branches of holomorphic dynamics. See e.g., the works of {\ifFirstInitial S.~\fi}Bullett and {\ifFirstInitial C.~\fi}Penrose \cite{BP94}, {\ifFirstInitial C.~T.~\fi}McMullen \cite{Mc95, Mc96}, {\ifFirstInitial M.~Yu.~\fi}Lyubich and {\ifFirstInitial Y.~\fi}Minsky \cite{LM97}, {\ifFirstInitial P.~\fi}Ha\"{\i}ssinsky and {\ifFirstInitial K.~M.~\fi}Pilgrim \cite{HP09}, {\ifFirstInitial M.~\fi}Bonk and {\ifFirstInitial D.~\fi}Meyer \cite{BM10, BM17}, {\ifFirstInitial M.~\fi}Mj and {\ifFirstInitial S.~\fi}Mukherjee \cite{MM23}, and references therein.
	
	Our current article is partially motivated by the interest of the community, including {\ifFirstInitial M.~\fi}Bonk, {\ifFirstInitial D.~\fi}Meyer, {\ifFirstInitial S.~\fi}Rohde, etc., to extend Sullivan's dictionary to some fractals arising from probability theory, hoping to transplant key analytic tools and techniques to such settings, and partially motivated by the recent works of {\ifFirstInitial S.~\fi}Lee, {\ifFirstInitial M.~Yu.~\fi}Lyubich, {\ifFirstInitial N.~G.~\fi}Makarov, {\ifFirstInitial S.~\fi}Mukherjee, etc., on certain anti-holomorphic correspondences which will be discussed below.
	
	Apart from the analogy from Sullivan's dictionary, to answer Fatou's question, we need to ``naturally'' combine the dynamics of a rational map with that of a Kleinian group. Matings between Kleinian groups and rational maps developed by {\ifFirstInitial S.~\fi}Bullett, {\ifFirstInitial C.~\fi}Penrose, {\ifFirstInitial L.~\fi}Lomonaco, {\ifFirstInitial P.~\fi}Ha\"{i}ssinsky, and {\ifFirstInitial M.~\fi}Freiberger could combine some Kleinian groups and some rational maps in the category of holomorphic correspondences, see \cite{Bu00, BF05, BH07, BL20, BL22, BL24, BP94}. %
	
	Recently, motivated by the study of the dynamics of Schwarz reflection maps associated with quadrature domains, {\ifFirstInitial S.~\fi}Lee, {\ifFirstInitial M.~Yu.~\fi}Lyubich, {\ifFirstInitial N.~G.~\fi}Makarov, and {\ifFirstInitial S.~\fi}Mukherjee investigated matings between such reflection maps and a discrete group abstractly isomorphic to the modular group, see \cite{LLMM21}. Such matings are anti-holomorphic correspondences. Later, {\ifFirstInitial M.~Yu.~\fi}Lyubich, {\ifFirstInitial J.~\fi}Mazor, and {\ifFirstInitial S.~\fi}Mukherjee constructed a family of anti-holomorphic correspondences for Schwarz reflection maps associated with quadrature domains and gave two criteria that ensure that such anti-holomorphic correspondences are matings between Schwarz reflections and anti-Hecke groups, see \cite{LMM24}.
	
	%
	
	%

	In order to study the dynamics of holomorphic correspondences, as well as to investigate the density of hyperbolicity and structural stability in the category of holomorphic correspondences, {\ifFirstInitial C.~\fi}Siqueira and {\ifFirstInitial D.~\fi}Smania studied a family of holomorphic correspondences $\boldsymbol{f}_c (z) =z^{q/p} +c$ in \cite{Siq15, SS17, Siq22, Siq23}. In these papers, {\ifFirstInitial C.~\fi}Siqueira and {\ifFirstInitial D.~\fi}Smania generalized the notion of Julia sets\footnote{{\ifFirstInitial C.~\fi}Siqueira and {\ifFirstInitial D.~\fi}Smania's version of Julia sets for holomorphic correspondences is different from {\ifFirstInitial S.~\fi}Bullett and {\ifFirstInitial C.~\fi}Penrose's version in \cite[Section~3.2]{BP01}.} for $\boldsymbol{f}_c$, discussed the hyperbolicity for such holomorphic correspondences, established some geometric rigidity results for the Julia sets, and gave an upper bound of the Hausdorff dimension of the Julia sets.
	
    In the works cited above about holomorphic and anti-holomorphic correspondences, direct considerations on ergodic theory for correspondences have yet to be extensively carried out.
	
	On the other hand, ergodic theory for general correspondences has also attracted interest recently. %
	For example, Poincar\'{e}'s recurrence theorem was investigated by {\ifFirstInitial J.~P.~\fi}Aubin, {\ifFirstInitial H.~\fi}Frankowska, and {\ifFirstInitial A.~\fi}Lasota in \cite{AFL91}; various notions of topological entropy, and their upper and lower bounds were established in \cite{KT17}; some version of expansiveness was discussed in \cite{Wi70, PV17}; several characterizations of invariant measures are systematically investigated in \cite{MA99}; Perron--Frobenius operators and approximations of invariant measures are studied in \cite{Mil95}. Moreover, in the setting of holomorphic correspondences, {\ifFirstInitial T.~C.~\fi}Dinh, {\ifFirstInitial L.~\fi}Kaufmann, and {\ifFirstInitial H.~\fi}Wu \cite{Wu20, DKW20} studied some canonical probability measures under the dynamics of some holomorphic correspondences on Riemann surfaces. {\ifFirstInitial V.~\fi}Matus~de~la~Parra \cite{Ma23a, Ma23b} studied the measures towards which the backward or forward orbits of the matings discussed in \cite{BP94} equidistributes and proved that a version of entropy (see \cite{VS22}) of the measures coincides with the topological entropy (see \cite{KT17}) of the matings. However, systematic studies on invariant measures are still under development, which motivates us to study thermodynamic formalism for correspondences.

	\subsubsection*{Thermodynamic formalism for (single-valued) maps}

	Thermodynamic formalism, inspired by statistical mechanics and created by {\ifFirstInitial Ya.~G.~\fi}Sinai, {\ifFirstInitial R.~\fi}Bowen, {\ifFirstInitial D.~\fi}Ruelle, and others around the early 1970s \cite{Do68, Sin72, Bow75, Ru78}, is a mechanism to produce invariant measures with nice properties and prescribed Jacobian functions. 	
	To be more precise, for a continuous (single-valued) map $f\: X\to X$ on a compact metric space $(X,d)$, and a continuous function $\varphi \: X \to \mathbb{R}$ (called a potential), we can consider the associated topological pressure $P(f,\varphi)$ as a weighted version of the topological entropy $h_{\operatorname{top}}(f)$. The \emph{Variational Principle} identifies $P(f,\varphi)$ with the supremum of its measure-theoretic counterpart, the measure-theoretic pressure $P_\mu(f,\varphi) \= h_{\mu} (f)+\int_{X} \! \varphi   \diff \mu$, %
	(where $h_{\mu} (f)$ is the measure-theoretic entropy), over all invariant Borel probability measures $\mu$ \cite{Bow75,Wa76}. A measure that attains the supremum is called an \defn{equilibrium state}\index{equilibrium state} for the given map and potential. In particular, when the potential $\varphi$ is (cohomologous to) a constant function, the equilibrium state is called a \defn{measure of maximal entropy}\index{measure of maximal entropy}. The studies on the existence and uniqueness of equilibrium states (or measures of maximal entropy), as well as their ergodic and statistical properties such as supporting sets and equidistributions, have been the main motivation for much research in ergodic theory.
	
	The theory of thermodynamic formalism for $f$ with strong forms of hyperbolicity has been systematically studied. For example, it is well-known that if $f$ is forward expansive, then an equilibrium state exists. Moreover, we have the \emph{Ruelle--Perron--Frobenius theorems},
	which describe the equilibrium states for more regular potentials; see e.g., \cite[Theorem~2.1]{RT18} and \cite[Chapter~5]{PU10}. 

	%
	
	%
	%
	%
	%
	%
	
	%
	%
	%
	
	%
	%
	%
	%
	
	One active direction for investigation in thermodynamic formalism nowadays is to extend the Ruelle--Perron--Frobenius theorem beyond the scope of uniform hyperbolicity. Our Theorem~\ref{theo_matings} on the Lee--Lyubich--Markorov--Mukherjee anti-holomorphic correspondences can be seen as such an attempt in complex dynamics. %

\subsubsection*{Thermodynamic formalism for correspondences}

	In the present article, we systematically develop thermodynamic formalism for correspondences. We will address the following aims:
	\begin{enumerate}
		\smallskip
		\item[(i)] Formulate definitions of \emph{measure-theoretic entropy of transition probability kernels} and \emph{topological pressure} for correspondences;
		\smallskip
		\item[(ii)] Establish a Variational Principle for correspondences with some expansion property;
		\smallskip
		\item[(iii)] Establish the existence of equilibrium states and obtain a \emph{Ruelle--Perron--Frobenius theorem} for such correspondences and potentials with certain regularity.
	\end{enumerate}

\subsection*{Statement of main results}

	Our main results consist of four parts: a Variational Principle and existence of equilibrium states, thermodynamic formalism for equilibrium states, a lower bound for the topological pressure, and applications to holomorphic and anti-holomorphic correspondences.

	\subsubsection*{Variational Principle}
	In this article, we will introduce the measure-theoretic entropy for transition probability kernels and the topological pressure for correspondences.
	
	Roughly speaking, a transition probability kernel $\cQ$ on a compact metric space $(X,  d)$ assigns each $x\in X$ a Borel probability measure $\cQ_x$ on $X$. Denote by $\MMM(X,\cQ)$ the set of $\cQ$-invariant Borel probability measures\footnote{Note that $\cQ$-invariant measures are also known as $\cQ$-stationary measures in probability.} (see Definition~\ref{invariant measure}).  The measure-theoretic entropy $h_\mu (\cQ)$ for a transition probability kernel $\cQ$ and a measure $\mu \in \MMM(X, \cQ)$ will be introduced (see Definition~\ref{measure-theoretic entropy of transition probability kernels}). The potentials are defined on the set $\cO_2 (T) \= \bigl\{ (x_1 ,  x_2)\in X^2  :  x_2 \in T(x_1) \bigr\}$ equipped with the metric $d_2$ given by $d_2 ((x_1 ,  x_2) ,  (y_1 ,  y_2)) \= \max\{ d(x_1 ,  y_1) ,\, d(x_2 ,  y_2) \}$ for all $(x_1 ,  x_2) ,\, (y_1 ,  y_2) \in \cO_2 (T)$. %

	We say that a transition probability kernel $\cQ$ on $X$ is \defn{supported by} a correspondence $T$ if the measure $\cQ_x$ is supported on the closed set $T(x)$ for every $x\in X$. Denote by $\tpk(X;T)$ the set of  transition probability kernels on $X$ supported by $T$.

	We conjecture the following Variational Principle to hold. 
	
	\begin{conj*}[Variational Principle for correspondences]
		Let $T$ be in a suitable class of correspondences on a compact metric space $(X,  d)$ and $\phi\: \cO_2 (T) \to \R$ be a sufficiently regular function, then
		\begin{equation}\label{P(T,phi)=sup_Q,mu(h_mu(Q)+int_Xphidmu)}
			P(T,  \phi)= \sup_{\cQ\in \tpk(X;T), \, \mu\in \MMM(X,\cQ) } \bigg\{ h_\mu (\cQ)+ \int_X \! \int_{T(x_1)} \! \phi (x_1 ,  x_2)   \diff \cQ_{x_1} (x_2)   \diff \mu (x_1) \bigg\} .
		\end{equation}
	\end{conj*}
	
	This conjecture naturally generalizes the classical Variational Principle for (single-valued) maps. Specifically, the topological pressure $P (T,  \phi)$ for the correspondence $T$ generalizes the classical topological pressure for continuous maps, the measure-theoretic entropy $h_\mu (\cQ)$ of the transition probability kernel $\cQ$ for the $\cQ$-invariant measure $\mu$ generalizes the classical measure-theoretic entropy of a measure-preserving endomorphism, and the integral in (\ref{P(T,phi)=sup_Q,mu(h_mu(Q)+int_Xphidmu)}) corresponds to the potential energy in the classical Variational Principle.
	
	To the best of our knowledge, no version of Variational Principle for correspondences has been established. In this article, we establish a version of Variational Principle for correspondences with some expansion properties (see Theorem~\ref{t_VP_forward_exp}). We have not found any counterexample to our conjecture.

	If a transition probability kernel $\cQ \in \tpk(X;T)$ and a measure $\mu \in \MMM(X, \cQ)$ attain the supremum in (\ref{P(T,phi)=sup_Q,mu(h_mu(Q)+int_Xphidmu)}), then we call the pair $(\mu ,  \cQ)$ an \defn{equilibrium state}\index{equilibrium state} for the correspondence $T$ and potential function $\phi$. Moreover, if $\phi \equiv 0$, we call $(\mu ,  \cQ)$ a \defn{measure of maximal entropy}\index{measure of maximal entropy} for $T$.
	
	\subsubsection*{Variational Principle and the existence of equilibrium states}

	A correspondence is \emph{forward expansive}, if, roughly speaking, every pair of distinct forward orbits $(x_1,  x_2,  \dots),  (y_1,  y_2,  \dots)$ consists of a pair of corresponding entries $x_k$ and $y_k$ with at least a specific distance apart (see Definition~\ref{forward expansiveness}).

	\begin{thmx}\label{t_VP_forward_exp}
		Let $(X,  d)$ be a compact metric space, $T$ be a forward expansive correspondence on $X$, and $\phi \: \cO_2 (T) \to \R$ be a continuous function. Then the following statements are true:
		\begin{enumerate}
			\smallskip
			\item[(i)] The Variational Principle holds:
			\begin{equation*}
				P(T,  \phi)= \sup_{\cQ\in \tpk(X;T), \, \mu\in \MMM(X,\cQ) } \biggl\{ h_\mu (\cQ)+ \int_X \! \int_{T(x_1)} \! \phi (x_1 ,  x_2)   \diff \cQ_{x_1} (x_2)   \diff \mu (x_1) \biggr\} \in \R.
			\end{equation*}
			\smallskip
			\item[(ii)]%
			There exists an equilibrium state $(\mu,  \cQ)$ for the correspondence $T$ and potential $\phi$.
		\end{enumerate}
	\end{thmx}
	
	The proof occupies Section~\ref{sct_Variational_principle_for_positively_RW-expansive_correspondences} and is the most technical part of this article.

	\subsubsection*{Thermodynamic formalism and equidistribution}

	We introduce various properties for correspondences and potential functions and then give two versions of thermodynamic formalism.

	Let $T$ be a correspondence on a compact metric space $X$. Its orbit space $\cO_\omega (T)$ equipped with metric $d_\omega$ is given in (\ref{O_omega(T)=}) and~(\ref{d_omega=}). If $\cQ\in \tpk(X;T)$ and $\mu\in \MMM(X,\cQ)$, then we denote by $\mu \cQ^\omega |_T$ a probability measure on $\cO_\omega (T)$ given in Remark~\ref{muQ|_T}.
	
	In the first version, we assume that $T$ has the specification property (Definition~\ref{specification correspondence}) and $\phi$ is Bowen summable (Definition~\ref{Bowen summable potential}), then consequently the Variational Principle holds, the equilibrium state exists and is unique in an appropriate sense, and the unique equilibrium state can be obtained by investigating the (classical) Ruelle operator $\cL_{\tphi}$ and $\cL_{\tphi}^*$ (see (\ref{aovewyacbwco8ey0fawe}) and~(\ref{int_YPhidL_psi*(nu)=int_YL_psi(Phi)dnu})).%
	
	\begin{thmx}\label{equilibrium state 1}
		Let $(X,  d)$ be a compact metric space, $T$ be a forward expansive correspondence with the specification property, and $\phi\: \cO_2 (T) \to \R$ be a Bowen summable continuous function. Then the Variational Principle (\ref{P(T,phi)=sup_Q,mu(h_mu(Q)+int_Xphidmu)}) holds and there exists an equilibrium state $(\mu_\phi ,  \cQ)$ for the correspondence $T$ and potential $\phi$, i.e., there exist $\cQ\in \tpk(X;T)$ and $\mu\in \MMM(X,\cQ)$ such that
		\begin{equation}\label{P(T,phi)=h_mu(Q)+int_Xphidmu intro}
			P(T,  \phi)= h_{\mu_\phi} (\cQ) +\int_X \! \int_{T(x_1)} \! \phi (x_1 ,  x_2)   \diff \cQ_{x_1} (x_2)   \diff \mu_\phi (x_1).
		\end{equation}
		Moreover, the equilibrium state $(\mu_\phi ,  \cQ)$ is unique in the sense that the measure $\mu_\phi$ is unique and that if there are two equilibrium states $(\mu_\phi ,  \cQ)$ and $(\mu_\phi ,  \cQ ')$, then $\cQ_x (A)=\cQ_x '(A)$ for $\mu_\phi$-almost every $x\in X$ and all Borel measurable $A\subseteq X$.
		
		Furthermore, the equilibrium state $(\mu_\phi ,  \cQ)$ can be obtained in the following way:		
		\begin{enumerate}
			\smallskip
			\item[(i)] There is a Borel probability measure $m_\phi$ on $X$ and a transition probability kernel $\cQ$ on $X$ supported by $T$ such that $m_\phi \cQ^\omega |_T$ is an eigenvector of $\cL_{\tphi}^*$.
			\smallskip
			\item[(ii)] There is a Borel measurable function $u_\phi \in L^1 (m_\phi)$ such that $\cL_{\tphi} (\tu_\phi)= \lambda \tu_\phi$, where $\lambda =\exp \bigl(  P\bigl( \sigma ,  \tphi \bigr)  \bigr) =\exp ( P(T ,  \phi) )$ and $\tu_\phi \: \cO_\omega (T)\to \R$ is the bounded Borel measurable function induced by $u_\phi$ given by
			$\tu_\phi (x_1 ,  x_2 ,  \dots)\= u_\phi (x_1)$.
			\smallskip
			\item[(iii)] Set $\mu_\phi \= u_\phi m_\phi$, then $(\mu_\phi ,  \cQ)$ is the equilibrium state for $T$ and $\phi$.
		\end{enumerate}
	\end{thmx}
	
	In the second version, we assume that $T$ is distance-expanding (Definition~\ref{distance-expanding correspondence}), open (Definition~\ref{open correspondence}), and strongly transitive (Definition~\ref{strongly transitive correspondence}) and $\phi$ is H\"{o}lder continuous, then the Variational Principle holds, the equilibrium state exists and is unique in an appropriate sense, and the unique equilibrium state can be obtained from the Ruelle operator and has the equidistribution property.
	
	\begin{thmx}\label{phwfq9}
		Let $T$ be a open, strongly transitive, distance-expanding correspondence on $X$ and $\phi\: \cO_2 (T) \to \R$ be a H\"{o}lder continuous function. Then the following statements are true:
		\begin{enumerate}
			\smallskip
			\item[(1)]	The Variational Principle (\ref{P(T,phi)=sup_Q,mu(h_mu(Q)+int_Xphidmu)}) holds and there exists an equilibrium state $(\mu_\phi ,  \cQ)$ for the correspondence $T$ and potential $\phi$.	Moreover, the equilibrium state $(\mu_\phi ,  \cQ)$ is unique in the sense that the measure $\mu_\phi$ is unique and if there are two equilibrium states $(\mu_\phi ,  \cQ)$ and $(\mu_\phi ,  \cQ ')$, then $\cQ_x (A)=\cQ_x '(A)$ for $\mu_\phi$-almost every $x\in X$ and all Borel measurable $A\subseteq X$. 
			
			\smallskip
			\item[(2)] The equilibrium state $(\mu_\phi ,  \cQ)$ can be obtained in the same way as (i)--(iii) in Theorem~\ref{equilibrium state 1}. If, moreover, $T$ is continuous (Definition~\ref{continuity}), then $u_\phi$ is continuous.
			
			\smallskip
			\item[(3)]	The backward orbits under $T$ are equidistributed with respect to the measure $\mu_\phi$. More precisely, if we denote 
		\begin{align*}
			\cO_{-n} (x)&\= \bigl\{ (y_0 ,  y_1 ,  \dots ,  y_n)\in X^{n+1}  :  y_n =x ,\, y_k \in T(y_{k-1}) \text{ for each } k\in \{ 1 ,\, \dots ,\, n \} \bigr\}, \\
			Z_n(x) &\= \sum\limits_{(y_0 ,  \dots ,  y_n) \in \cO_{-n} (x)} \exp \biggl( \sum\limits_{i=0}^{n-1}   \phi (y_i ,  y_{i+1})  \biggr) ,
		\end{align*}
		then the following statements are true:
		\begin{enumerate}
			\smallskip
			\item[(a)] For each $x\in X$, the following sequence of Borel probability measures on $X$
			\begin{equation*}
				\frac{1}{ Z_n(x) } \sum_{(y_0 ,  \dots ,  y_n) \in \cO_{-n} (x)} \frac{\sum_{j=0}^n    \delta_{y_j}  }{n+1}  \exp  \biggl( \sum_{i=0}^{n-1}  \phi (y_i ,  y_{i+1}) \biggr)   ,\, n\in \N,
			\end{equation*}
			converges to $\mu_\phi$ in the weak* topology as $n \to +\infty$.
			\smallskip
			\item[(b)] If, moreover, $T$ is topologically exact (Definition~\ref{topologically exact}), then for each $x\in X$, the following sequence of Borel probability measures on $X$
			\begin{equation*}
				\frac{1}{ Z_n(x)  }\sum_{(y_0 ,  \dots ,  y_n) \in \cO_{-n} (x)} \delta_{y_0} \exp \biggl( \sum_{i=0}^{n-1}  \phi (y_i ,  y_{i+1}) \biggr)  ,\, n\in \N,
			\end{equation*}
			converges to $m_\phi$ in the weak* topology as $n \to +\infty$.
		\end{enumerate}
		\end{enumerate}
	\end{thmx}

	For general correspondences and continuous potential functions, we establish the following result.
	
	\begin{thmx}\label{t_HVP}
		Let $T$ be a correspondence on a compact metric space $(X,  d)$ and $\phi\: \cO_2 (T) \to \R$ be a continuous function. Then the following statements are true:
		\begin{enumerate}
			\smallskip
			\item[(i)] There exists $\cQ\in \tpk(X;T)$ and $\mu\in \MMM(X,\cQ)$.
			\smallskip
			\item[(ii)] $P(T,  \phi)\geq \sup\limits_{\cQ\in \tpk(X;T), \, \mu\in \MMM(X,\cQ) } \big\{ h_\mu (\cQ)+ \int_X \! \int_{T(x_1)} \! \phi (x_1 ,  x_2)   \diff \cQ_{x_1} (x_2)   \diff \mu (x_1) \big\}$.
		\end{enumerate}
	\end{thmx}

\subsection*{Applications to holomorphic and anti-holomorphic correspondences}

	Our primary motivation for the investigation in this article comes from the following two classes of correspondences in complex dynamics. For more details, see Section~\ref{sct_Applications to holomorphic or anti-holomorphic correspondences}.
	
	\subsubsection*{Lee--Lyubich--Markorov--Mukherjee anti-holomorphic correspondences.}

	Let $d \in \N$ and $f\: \CDach \to \CDach$ be a rational map of degree $d+1$ that is univalent on the open unit disk $\D$. Set $\eta (z) \= 1 / \oz$, the reflection map on the unit circle.
	
	\defn{The Lee--Lyubich--Markorov--Mukherjee anti-holomorphic correspondence}\index{Lee--Lyubich--Markorov--Mukherjee anti-holomorphic correspondence} $\fC^*$ is defined as follows:
	\begin{equation}\label{equ:Lee--Lyubich--Markorov--Mukherjee anti-holomorphic correspondence}
		\fC^* (z)\= \biggl\{ w\in \CDach  :  \frac{f(w)-f(\eta (z))}{w-\eta (z)} =0 \biggr\}
	\end{equation}
	for all $z\in \CDach$. See \cite[Section~2]{LMM24} for more details. We note that $\fC^*$ is not forward expansive. %
	
	%
	%
	%
	
	%
	
	%
	%
	%
	%
	%
	
	%

	%
	
	\begin{thmx}\label{theo_matings}
		Let $\fC^*$ be the correspondence given above and $\phi \: \cO_2 (\fC^*) \to \R$ be a continuous function. 
		Then the Variational Principle (\ref{P(T,phi)=sup_Q,mu(h_mu(Q)+int_Xphidmu)}) holds for $T\=\fC^*$ and $X\=\CDach$.
	\end{thmx}

	\subsubsection*{A family of hyperbolic holomorphic correspondences}

	Fix $p ,\, q\in \N$ satisfying $p<q$. Let $c\in \C$.
	Denote by $\boldsymbol{f}_c (z)=z^{q/p} +c$ the correspondence\footnote{The superscript $q/p$ in $z^{q/p} +c$ is merely a notation, not a fraction. The two correspondences $z^{2/1} +c$ and $z^{4/2} +c$ are different by this definition.} on $\CDach$ given by
	\begin{equation*}
		\boldsymbol{f}_c (z) \= \bigl\{ w\in \CDach  :  (w-c)^p =z^q \bigr\}  \qquad \text{for all } z\in \CDach.
	\end{equation*}

	A version of \emph{Julia set} $J(\boldsymbol{f}_c)$ is defined as the closure of the union of all repelling periodic orbits of $\boldsymbol{f}_c$, see e.g., \cite[Definition~6.31]{Siq15} or \cite[Section~2.1]{SS17}. Denote by $\boldsymbol{f}_c |_J$ a map given by $\boldsymbol{f}_c |_J (z) \= J(\boldsymbol{f}_c) \cap \boldsymbol{f}_c (z)$ for all $z\in J(\boldsymbol{f}_c)$.
	
	Set $P_c \= \overline{\bigcup_{n\in \N} \boldsymbol{f}_c^n (0)}$ and
	\begin{equation*}
		M_{q/p ,  0} \= \{ c\in \C  :  \exists (x_1 ,  x_2 ,  \dots) \in \cO_\omega (\boldsymbol{f}_c)
		\text{ such that } x_1 =0 \text{ and } \{ x_n \}_{n\in \N} \text{ is bounded} \}.
	\end{equation*}
	
	A number $c\in \C$ is called a \defn{simple center}\index{simple center} if $c \neq 0$ and there is only one bounded orbit $(x_1,  x_2,  \dots) \in \cO_\omega (\boldsymbol{f}_c)$ with $x_1 =0$ and such a bounded orbit is a cycle, see e.g., \cite[Section~2.2]{Siq22}.
	
	\begin{thmx}\label{theo_hyperboliccorrespondence}
		There is an open set $H_{q/p}$ containing both $\C \smallsetminus M_{q/p ,  0}$ and every simple center such that for every $c \in H_{q/p}$, the following statements are true:
		\begin{enumerate}
			\smallskip
			\item[(i)] The set $\C \smallsetminus P_c$ is a hyperbolic Riemann surface.
			\smallskip
			\item[(ii)] Statements~(1)--(3) in Theorem~\ref{phwfq9} hold for the correspondence $\boldsymbol{f}_c |_J$ on the compact metric space $(J (\boldsymbol{f}_c),  d_c)$, where $d_c$ refers to the hyperbolic metric on $\C \smallsetminus P_c$.
		\end{enumerate}
	\end{thmx}

	Note that 0 is not a simple center, so Theorem~\ref{theo_hyperboliccorrespondence} does not work when $c$ is close to 0. For $c$ in a neighborhood of 0, we have the following result.
	
	\begin{thmx}\label{theo_hyperboliccorrespondence2}
		There is an open neighborhood $U_{q/p}$ of $0$ with the property that for every $c\in U_{q/p}$, statements (1)--(3) in Theorem~\ref{phwfq9} hold for the correspondence $\boldsymbol{f}_c |_J$ on the compact space $J (\boldsymbol{f}_c)$ equipped with the Euclidian metric on $\C$.
	\end{thmx}

\subsection*{Strategies of this article}

	We discuss below the difficulties and novelties in this article.
	
	Correspondences assign each point a set, which means they are multi-valued, while the continuous maps are single-valued. Despite this difference, we define the topological pressure for correspondences in a natural way using $(n,  \epsilon)$-separated sets and $(n,  \epsilon)$-spanning sets (see Section~\ref{subsct_Definition of topological pressure for correspondences}).

	In contrast, measure-theoretic entropy is more delicate to define for correspondences, because, given a subset of a space, there are no canonical distributions on it. We overcome this difficulty by assigning a distribution on the image of each point under a correspondence, i.e., we consider a transition probability kernel. Recall a \emph{transition probability kernel} on a space $X$ assigns each point a probability measure on $X$. We introduce the measure-theoretic entropy of a transition probability kernel using the entropy of partitions (see Section~\ref{subsct_Definition of measure-theoretic entropy for transition probability kernels}).
	
	To establish our Variational Principle, we need a relation between correspondences and transition probability kernels. The most natural relation is \emph{support}. Recall that a correspondence $T$ on $X$ assigns each point $x\in X$ a closed subset $T(x)$ of $X$, and a transition probability kernel on $X$ assigns each point a probability measure on $X$. Recall that a transition probability kernel is \emph{supported by} a correspondence $T$ if, for each point $x\in X$, the probability measure is supported on the closed subset $T(x)$. We conjecture a version of the Variational Principle in terms of transition probability kernels and their invariant measures, see (\ref{P(T,phi)=sup_Q,mu(h_mu(Q)+int_Xphidmu)}).

	Now, we sketch the main results and their proofs.
	
	We first establish the characterizations of both the topological pressure of correspondences and the measure-theoretic entropy of transition probability kernels in terms of the shift map on the orbit space. We briefly explain the dynamics of such a shift map now.
	
	Let $T$ be a correspondence on a compact metric space $X$, $\cQ$ be a transition probability kernel on $X$, and $\mu$ be a $\cQ$-invariant probability measure on $X$, i.e., the pushforward of $\mu$ under $\cQ$ is still $\mu$ (Definition~\ref{invariant measure}). Let $\sigma$ be the shift map on $X^\omega \= \{ (x_1 ,  x_2 ,  \dots )  :  x_k \in X$ for all $k\in \N \}$. It turns out that $\cO_\omega (T) \= \{ (x_1 ,  x_2 ,  \dots)\in X^\omega  :  x_{k+1} \in T(x_k) \}$ is an invariant set of $\sigma$. Moreover, for the potential function $\phi \: \cO_2 (T) \to \R$, we also lift it to the orbit space $\cO_\omega (T)$: denote by $\tphi$ the function on $\cO_\omega (T)$ given by $\tphi (x_1 ,  x_2 ,  \dots) \= \phi (x_1)$. Now, we state the following characterizations:
	
	First, the (classical) topological pressure of $\sigma$ restricted on $\cO_\omega (T)$ with respect to the potential $\tphi$ is equal to the topological pressure of $T$ with respect to the potential $\phi$ (see Theorem~\ref{topological pressure coincide with the lift}).
	
	Second, we consider the Markov process with the initial distribution $\mu$ and the transition probability kernel $\cQ$. Denote by $\mu \cQ^\omega$ the distribution of forward infinite orbits of the Markov process. It is a probability measure on $X^\omega$. The assumption that $\mu$ is $\cQ$-invariant implies that $\mu \cQ^\omega$ is $\sigma$-invariant (see Subsection~\ref{subsct_Measure-theoretic entropy of transition probability kernels and the shift map are equal}). We prove that the (classical) measure-theoretic entropy of $\sigma$ for the $\sigma$-invariant measure $\mu \cQ^\omega$ is equal to the measure-theoretic entropy of $\cQ$ for the $\cQ$-invariant measure $\mu$. %
	
	Thereby, the following conjecture about $\sigma$ is equivalent to our (conjectured) Variational Principle:
	\begin{equation}\label{q938f-qy239yq3ueq3}
		P\bigl( \sigma |_{\cO_\omega (T)} ,  \tphi \bigr) =\sup_{\nu} \biggl\{ h_\nu (\sigma) +\int \! \tphi   \diff \nu \biggr\},
	\end{equation}
	where $\nu$ ranges over all Borel probability measures on $\cO_\omega (T)$ induced by Markov processes and invariant under $\sigma$ (see Subsection~\ref{subsct_Thermodynamic formalism for forward expansive correspondences with the specification property} for details).
	
	Note that if the supremum in (\ref{q938f-qy239yq3ueq3}) is obtained by letting $\nu$ range over all Borel probability measures on $\cO_\omega (T)$ invariant under $\sigma$, then (\ref{q938f-qy239yq3ueq3}) holds, ensured by the (classical) Variational Principle. From this perspective, we establish Theorem~\ref{t_HVP}. But some $\sigma$-invariant measures on $\cO_\omega (T)$ are not induced by Markov processes, which is the difficulty of establishing the Variational Principle for correspondences. For the Variational Principle and the corresponding thermodynamic formalism, we use two kinds of methods in Sections~\ref{sct_Variational_principle_for_positively_RW-expansive_correspondences} and~\ref{sct_Thermodynamic_formalism_of_expansive_correspondences}, respectively.%
	
	In Section~\ref{sct_Variational_principle_for_positively_RW-expansive_correspondences}, we introduce the forward expansiveness for correspondences and establish the Variational Principle for forward expansive correspondences (Theorem~\ref{t_VP_forward_exp}). %
	We overcome the difficulty that a probability measure invariant under the shift map may not be induced by a Markov process in this section. Specifically, for an arbitrary Borel probability measure $\nu$ on the orbit space $\cO_\omega (T)$ which is invariant under the shift map, the projection of $\nu$ onto the first two coordinates can induce a measure $\mu$ and a conditional transition probability kernel $\cQ$. It turns out that $\cQ$ is supported by $T$ and that $\mu$ is $\cQ$-invariant. Moreover, if $T$ is forward expansive, by formulating a Rokhlin formula for measure-theoretic entropy of transition probability kernels, we show that the measure-theoretic entropy of $\cQ$ for $\mu$ is at least that of $\sigma$ for $\nu$, from which our Variational Principle follows.
	
	In Section~\ref{sct_Thermodynamic_formalism_of_expansive_correspondences}, we first introduce various properties for correspondences or potential functions, including specification property, Bowen summability, distance-expanding property, openness, and strong transitivity, and recall the topologically exact property for correspondences. Then we give two versions of thermodynamic formalism for a forward expansive correspondence $T$ on a compact metric space $X$ with a continuous potential $\phi \: \cO_2 (T) \to \R$ satisfying some of the properties above. Since the Ruelle--Perron--Frobenius theorem describes the equilibrium state explicitly, we can verify that the equilibrium state is indeed induced by a Markov process.

	\subsection*{Structures}

	Section~\ref{sct_notation} collects some notations used throughout the article.
	In Section~\ref{sct_Topological_pressure_of_correspondences}, we introduce the topological pressure of a correspondence $T$ with respect to a continuous potential function. Then we give a characterization of our topological pressure.
	In Section~\ref{sct_Measure-theoretic_entropy_of_transition_probability_kernels}, we introduce the measure-theoretic entropy $h_\mu (\cQ)$ of a transition probability kernel $\cQ$ for a $\cQ$-invariant probability measure $\mu$. Then we we give a characterization of our measure-theoretic entropy of a transition probability kernel.
	In Section~\ref{sct_Variational_principle_for_positively_RW-expansive_correspondences}, we introduce the forward expansiveness for correspondences and establish the Variational Principle for forward expansive correspondences (Theorem~\ref{t_VP_forward_exp}). We also investigate general correspondences and establish Theorem~\ref{t_HVP}.
	In Section~\ref{sct_Thermodynamic_formalism_of_expansive_correspondences}, we first introduce various properties for correspondences and potential functions. Then we give two versions of thermodynamic formalism for a forward expansive correspondence and some continuous potential.
	Apart from the above, in Section~\ref{sct_Applications to holomorphic or anti-holomorphic correspondences}, we apply our theory to (i) two examples: the Lee--Lyubich--Markorov--Mukherjee anti-holomorphic correspondences and a family of hyperbolic holomorphic correspondences of the form $z^{q/p} +c$; and (ii) two degenerate cases: transition matrices and (single-valued) maps.
	
\subsection*{Acknowledgments} 

	The authors want to thank Wenyuan Yang for interesting discussions on random walks. X.~Li wants to thank Xianghui Shi for some useful comments. %

\section{Notation}\label{sct_notation}

We follow the convention $\N \= \{1,\, 2,\, 3,\, \dots\}$, $\N_0 \= \N\cup \{ 0\}$, and $\wh{\N}\= \N \cup \{\omega \}$. Here $\omega$ is the least infinite ordinal. For each $n\in \N_0$, write $\zeroto{n}\= \{ 0 ,\, 1 ,\, \dots ,\, n\}$ and $\oneto{n} \= \zeroto{n} \smallsetminus \{0\}$. Let $\Z$ be the set of integers, $\C$ be the set of all complex numbers, $\CDach \= \C \cup \{\infty\}$, $\D \= \{ z\in \C  :   \abs{z}<1 \}$, and $\D^* \= \{ z\in \CDach  :   \abs{z}>1 \}$.

When we use the notation $A \cap B \times C$ or $B \times C \cap A$ for sets $A$, $B$, and $C$, it should be interpreted as first performing the multiplication operation between sets $B$ and $C$ and subsequently finding the intersection of the result with $A$.

Let $Y$ be a set and $n\in \N$. Define the reversal $\gamma_n \: Y^n \to Y^n$ by
\begin{equation*}
	\gamma_n (x_1 ,  x_2 ,  \dots ,  x_n) \= (x_n ,  x_{n-1} ,  \dots ,  x_1) \quad\text{ for all } (x_1 ,  \dots ,  x_n) \in Y^n.
\end{equation*}
For a function $\phi \: Y\to \R$, write $ \norm{\phi}_\infty \= \sup \{  \abs{\phi (x)}  :  x\in Y \}$. Let $\SAA(Y)$ be a $\sigma$-algebra on $Y$. Denote by $\BBB(Y,\R)$ the set of real-valued bounded measurable functions on $Y$, by $\PPP(Y)$ the set of probability measures on $Y$, and by $\tpk(Y,X)$ the set of transition probability kernels from $Y$ to $X$.

For a subset $A\subseteq Y$, denote by $\mathbbold{1}_A \: Y\to \R$ the function that assigns $1$ to each point $y\in Y$ and $0$ elsewhere.

Let $\psi\: M \to \C$ be a function defined on a subset $M$ of $Y^2$. For each $n\in \N$, we write
\begin{equation}\label{e:S_nphi=}
	S_n \psi (\udx) = S_n \psi (x_1 ,  \dots ,  x_{n+1})\= \sum_{i=1}^n   \psi (x_i ,  x_{i+1})
\end{equation}
for $\udx=(x_1 ,  \dots ,  x_{n+1}) \in Y^{n+1}$, whenever the right-hand side of (\ref{e:S_nphi=}) makes sense.

We write, for $i,\, j\in \Z$ with $i\leq j$,
\begin{equation}    \label{e:vect}
	\vect{x}{i}{j} \= (x_i, x_{i+1}, \dots, x_j) \qquad \text{and} \qquad
	\vect{x}{i}{\infty} \= (x_i, x_{i+1}, \dots).
\end{equation}

Let $(X,d)$ be a compact metric space. Denote by $\SBB (X)$ the Borel $\sigma$-algebra on $X$, by $\cF (X)$ the set of all non-empty closed subsets of $X$, by $\CCC(X,\R)$ the set of real-valued continuous functions on $X$, and by $\MMM(X,g)$ the set of $g$-invariant Borel probability measures for a continuous map $g\: X \to X$. In this setting, $\PPP(X)$ denotes the set of Borel probability measures on $X$.

For a map $f\: X\to X$, denote by $\cC_f \: X\to \cF (X)$ the map given by
\begin{equation}  \label{e:Def_C_f}
	\cC_f (x) \= \{ f(x) \}  \qquad \text{for } x \in X.
\end{equation}

Let $T \: X \to \cF (X)$ be a map. %
For each subset $A \subseteq X$, set $T(A) \= \bigcup_{x\in A} T(x) \subseteq X$. For $n \in \N$, define $T^n (A) \subseteq X$ recursively on $n$ with $T^1 (A) \= T(A)$ and $T^{n+1} (A) \= T (T^n (A)) \text{ for all } n\in \N$. Define $T^{-1} (A) \= \{ x \in X  :  A \cap T (x) \neq \emptyset \} \subseteq X$. For $n\in \N$, define $T^{-n} (A) \subseteq X$ recursively on $n$ with $T^{-(n+1)} (A) \= T^{-1} (T^{-n} (A))$ for all $n\in \N$. Write $T^n (x) \= T^n (\{ x \})$ for $n\in \Z \smallsetminus \{ 0 \}$ and $x\in X$.

For a subset $A \subseteq X$ and each $x\in X$, define $T|_A (x) \= T(x) \cap A$.

For each $n\in \N$, equip $X^n$ with the metric $d_n$ given by
\begin{equation}\label{d_n=}
	d_n ( \vect{x}{1}{n} ,  \vect{y}{1}{n})\= \max \{ d(x_i ,  y_i)  :  i \in \oneto{n}\}   \qquad \text{for } \vect{x}{1}{n},  \vect{y}{1}{n}\in X^n. 
\end{equation}
Similarly, equip $X^\omega \= \{ \vect{x}{1}{\infty}  :  x_k \in X$ for all $k\in \N \}$ with the metric $d_\omega$ given by
\begin{equation}\label{d_omega=}
    d_\omega (\vect{x}{1}{\infty},  \vect{y}{1}{\infty})\= \sum_{k=1}^{+\infty}   \frac{d(x_k ,  y_k)}{2^k (1+d(x_k ,  y_k))}    \qquad \text{for } \vect{x}{1}{\infty},  \vect{y}{1}{\infty}\in X^\omega. 
\end{equation}

For each $n\in \N$, write
\begin{equation*}
	\cO_n (T)\= \{\vect{x}{1}{n} \in X^n  :  x_{k+1}\in T(x_k) \text{ for each } k\in \oneto{n-1} \} \subseteq X^n.
\end{equation*}
Then the \defn{orbit space}\index{orbit space} $\cO_\omega (T)$ induced by $T$ is given by
\begin{equation}\label{O_omega(T)=}
    \cO_\omega (T)\= \{ \vect{x}{1}{\infty} \in X^\omega  :  x_{k+1}\in T(x_k) \text{ for each } k\in \N\} \subseteq X^\omega.
\end{equation}
For each $n\in \wh{\N}$, We call an element in $\cO_n (T)$ an \defn{orbit}\index{orbit}.

Suppose $\varphi \in \CCC( X, \R)$ and $\phi \in \CCC(\cO_2 (T), \R)$. Denote by $\tvarphi ,\, \tphi \: \cO_\omega (T) \to \R$ and $\hvarphi \: \cO_2 (T) \to \R$ the functions given by
\begin{equation}\label{lifts of potential}
    \tvarphi ( \vect{x}{1}{\infty} ) \= \varphi (x_1), \quad \tphi ( \vect{x}{1}{\infty} ) \= \phi (x_1 ,  x_2), \quad \text{and } \hvarphi (x_1 ,  x_2) \= \varphi (x_1).
\end{equation}

Denote by $\tpi_1 ,\, \tpi_2 \: \bigcup_{n\in \wh{\N} \smallsetminus \{1\}} X^n \to X$ and $\tpi_{12} \: \bigcup_{n\in \wh{\N} \smallsetminus \{1\}} X^n \to X^2$ the projection maps given by
\begin{equation}\label{tpi=}
	\tpi_{12} \: \vect{x}{1}{n} \mapsto (x_1 ,  x_2), \quad \tpi_1 \: \vect{x}{1}{n}   \mapsto x_1 ,\quad \tpi_2 \: \vect{x}{1}{n}  \mapsto x_2.
\end{equation}
If $\mu$ is a Borel probability measure on a subset $A$ of $X^n$ for some $n\in \wh{\N} \smallsetminus \{1\}$, then $\mu \circ \tpi_{12}^{-1}$ refers to a Borel probability measure on $X^2$ given by $\bigl( \mu \circ \tpi_{12}^{-1} \bigr) (B) \=\mu \bigl( A \cap \tpi_{12}^{-1} (B) \bigr)$ for all $B\in \SBB \bigl(X^2\bigr)$, and $\mu \circ \tpi_i^{-1}$ refers to a Borel probability measure on $X$ given by $\bigl( \mu \circ \tpi_i^{-1} \bigr) (B) \=\mu \bigl( A \cap \tpi_i^{-1} (B) \bigr)$ for all $B\in \SBB (X)$, where $i=1$ or $i=2$.

\section{Holomorphic and anti-holomorphic correspondences}\label{sct_Applications to holomorphic or anti-holomorphic correspondences}

In this section, we focus on our primary motivation for the investigation in this article: two examples in complex dynamics. We also discuss the degenerate cases of transition matrices and single-valued maps. Since we will apply our theory developed throughout this article to them, the reader can skip this section in the first read.%

\subsection{Lee--Lyubich--Markorov--Mukherjee correspondences}\label{subsct_Lee--Lyubich--Markorov--Mukherjee anti-holomorphic correspondences}

We aim to prove Theorem~\ref{theo_matings}.

Recall $d \in \N$ and that $f\: \CDach \to \CDach$ is a rational map of degree $d+1$ that is univalent on the open unit disk $\D$. Recall $\eta (z)=  1 / \oz$ and the correspondence $\fC^*(z)= \bigl\{ w\in \CDach : \frac{f(w)-f(\eta (z))}{w-\eta (z)} =0 \bigr\}$ for all $z \in \CDach$. See \cite{LMM24} for a more detailed study of such correspondences.

Denote $\Omega \= f(\D)$. An anti-holomorphic map $\tau \: \Omega \to \CDach$ is given by
\begin{equation}\label{equ:anti-holomorphic map}
	\tau \= f\circ \eta \circ (f|_{\D})^{-1}.
\end{equation}
Let $T(\tau) \= \CDach \smallsetminus \Omega$ and $S(\tau)$ be the singular set (consisting of all cusps and double points) of $\partial T(\tau) = \partial \Omega$. 
Denote $\overline{\Omega}^c  \= \CDach \smallsetminus \overline{\Omega}$. 
Set $S'(\tau) \= \partial \Omega \smallsetminus S(\tau)$, $T^0 (\tau) \= T(\tau) \smallsetminus S(\tau)$, and $T^\infty (\tau) \= \bigcup_{n=0}^{+\infty} \tau^{-n} \bigl( T^0 (\tau) \bigr)$. Write $K(\tau) \= \CDach \smallsetminus T^\infty (\tau)$, which is called the \emph{non-escaping set of $\tau$}. Note that $K(\tau)$ is a closed subset of $\CDach$ by \cite[Proposition~2.2]{LMM24}. Write
\begin{equation*}\label{equ:non-escaping set}
	\widetilde{K(\tau)} \= f^{-1} (K(\tau)) \quad \text{and}\quad
	 \widetilde{T^\infty (\tau)} \= f^{-1} (T^\infty (\tau)) =\CDach \smallsetminus \widetilde{K(\tau)}.
\end{equation*}
The subset $\widetilde{K(\tau)}$ is closed in $\CDach$, and the subset $\widetilde{T^\infty (\tau)}$ is open.

\begin{prop}[{\cite[Proposition~2.4]{LMM24}}]\label{wiodpeac}
	For all $z ,\, w\in \CDach$, if $w \in \fC^* (z)$, then $z\in \widetilde{K(\tau)}$ if and only if $w\in \widetilde{K(\tau)}$, and $z\in \widetilde{T^\infty (\tau)}$ if and only if $w\in \widetilde{T^\infty (\tau)}$.
\end{prop}

Set $V_0 \= f^{-1} \bigl(\overline{\Omega}^c\bigr)$, $U_0 \=f^{-1} (S'(\tau)) \cap \partial \D$, $V_n \= f^{-1} \bigr(\tau^{-n} \bigl( \overline{\Omega}^c \bigr)\bigr) \cap \D^*$, $U_n \=f^{-1} (\tau^{-(n-1)} (S' (\tau))) \cap \D^*$, $V_{-n} \= f^{-1} \bigr(\tau^{-n} \bigl( \overline{\Omega}^c \bigr)\bigr) \cap \D$, and $U_{-n} \=f^{-1} (\tau^{-n} (S' (\tau))) \cap \D$ for all $n\in \N$. %

\begin{lemma}\label{V_n and U_n are partition}
	The collection $\{V_n\}_{n\in \Z} \cup \{U_n\}_{n\in \Z}$ is a partition of $\widetilde{T^\infty (\tau)}$.
\end{lemma}

\begin{proof}
	Since $\tau^{-1} \bigl(T^0 (\tau)\bigr) \subseteq \Omega$ and $T^0 (\tau) \subseteq \CDach \smallsetminus \Omega$, we can see that $\bigl\{\tau^{-n} \bigl(T^0 (\tau)\bigr)\bigr\}_{n\in \N_0}$ is a partition of $T^\infty (\tau)$. Note that $T^0 (\tau) =\bigl( \CDach \smallsetminus \Omega\bigr) \smallsetminus S (\tau)$ is the disjoint union of $\overline{\Omega}^c$ and $S'(\tau)$, so $\bigr\{f^{-1} \bigl(\tau^{-n} \bigl( \overline{\Omega}^c \bigr)\bigr)\bigr\}_{n\in \N_0} \cup \bigl\{f^{-1}(\tau^{-n} ( S'(\tau))) \bigr\}_{n\in \N_0}$ is a partition of $\widetilde{T^\infty (\tau)} =f^{-1} (T^\infty (\tau))$.
	
	Recall that $\Omega$ is the univalent image of $\D$ under $f$. For each $n\in \N$, since $\tau^{-n} \bigl( \overline{\Omega}^c\bigr) \subseteq \Omega$ and $\tau^{-n} (S'(\tau)) \subseteq \Omega$, we have $f^{-1} \bigl(\tau^{-n} \bigl(\overline{\Omega}^c\bigr)\bigr) \subseteq \CDach \smallsetminus \partial \D =\D \cup \D^*$ and $f^{-1}(\tau^{-n} ( S'(\tau) )) \subseteq \D \cup \D^*$. Moreover, since $f(\D) =\Omega$, whose intersection with $S'(\tau)$ is empty, we have $f^{-1} ( S'(\tau) ) \subseteq \partial \D \cup \D^*$. Therefore, %
	the collection $\{V_n\}_{n\in \Z} \cup \{U_n\}_{n\in \Z}$ is a partition of $\widetilde{T^\infty (\tau)}$.
\end{proof}

By (\ref{equ:anti-holomorphic map}), we have $(\tau \circ f) (z) =\bigl( f \circ \eta \circ (f|_{\D})^{-1} \circ f \bigr) (z) =(f \circ \eta) (z)$ for all $z\in \D$, and $(\tau \circ f \circ \eta) (z) =f(z)$ for all $z \in \D^*$. This implies that for each $A \subseteq \CDach$,
\begin{equation}\label{8374fgcqb8o37gd3yibced}
	f^{-1} \bigl( \tau^{-1} (A) \bigr) \cap \D =\eta \bigl(f^{-1} (A) \bigr) \cap \D \quad \text{ and } \quad
	 \eta\bigl(f^{-1} \bigl(\tau^{-1} (A)\bigr)\bigr) \cap \D^* =f^{-1} (A) \cap \D^*.
\end{equation}

\begin{prop}\label{C^*^-1 on V_n}
	Let $n\in \N_0$ be arbitrary, then 
	\begin{enumerate}
		\smallskip
		\item[(i)] $(\fC^*)^{-1} (V_n) \subseteq V_{n-1} \cup V_{-n-1}$, 
		\smallskip
		\item[(ii)] $(\fC^*)^{-1} (V_{-n}) \subseteq V_{-n-1}$, 
		\smallskip
		\item[(iii)] $(\fC^*)^{-1} (U_n) \subseteq U_{n-1} \cup U_{-n}$, and 
		\smallskip
		\item[(iv)] $(\fC^*)^{-1} (U_{-n}) \subseteq U_{-n-1}$.
	\end{enumerate}
\end{prop}

\begin{proof}
	For $z ,\, w \in \CDach$, recall $w \in \fC^*(z)$ if and only if $\frac{f(w)-f(\eta (z))}{w-\eta (z)} =0$. It follows that $(\fC^*)^{-1} (w) \subseteq \eta \bigl( f^{-1} (f(w) ) \bigr)$. If $z \in (\fC^*)^{-1} (w)$, then $f' (w)=0$ or $\eta (z) \neq w$, and $f(\eta (z))= f(w)$. Recall $\Omega$ is the univalent image of $\D$ under $f$, and thus $f\bigl(\overline{\D}\bigr) =\overline{\Omega}$ and $f(\partial \D) =\partial \Omega$. If $w \in \D$ and $z \in (\fC^*)^{-1} (w)$, since $f$ is injective on $\D$ and thus $f' (w) \neq 0$, we have $\eta (z) \notin \D$. It follows that $\eta (z) \in \D^*$, i.e., $z \in \D$ because $f(\partial \D) =\partial \Omega$ and $f(\D) =\Omega$. As a result, $(\fC^*)^{-1} (\D) \subseteq \D$.
	
	First, $(\fC^*)^{-1} (V_0) \subseteq \eta \bigl(f^{-1} \bigl(f \bigl(f^{-1} \bigl(\overline{\Omega}^c\bigr)\bigr)\bigr)\bigr) =\eta \bigl(f^{-1} \bigl(\overline{\Omega}^c\bigr)\bigr)$. Recall $f(\overline{\D}) =\overline{\Omega}$, so $f^{-1} \bigl(\overline{\Omega}^c\bigr) \subseteq \D^*$. By (\ref{8374fgcqb8o37gd3yibced}), 
	$\eta \bigl(f^{-1} \bigl(\overline{\Omega}^c\bigr)\bigr) 
		=\eta \bigl(f^{-1} \bigl(\overline{\Omega}^c\bigr)\bigr) \cap \D 
		=f^{-1} \bigl(\tau^{-1} \bigl(\overline{\Omega}^c\bigr)\bigr) \cap \D =V_{-1}$.
	Hence, $(\fC^*)^{-1} (V_0) \subseteq V_{-1}$. %
	
	For each $w \in U_0 =f^{-1} ( S'(\tau) ) \cap \partial \D$, if $z \in (\fC^*)^{-1} (w)$, i.e., $\frac{f(w)- f(\eta (z))}{(w- \eta (z))} =0$, then $f(\eta (z))= f(w) \in S'(\tau) $, and thus $\eta (z) \in \CDach \smallsetminus \D =\D^* \cup \partial \D$, i.e., $z\in \D \cup \partial \D$. Moreover, we have $\eta (z) \neq w$ or $f' (w) =0$. We argue by contradiction and assume $z\in \partial \D$, then $\eta (z) \neq w$ indicates that $f(w) =f(\eta(z))$ is a double point on $\partial \Omega$, and $f' (w)=0$ indicates that $f(w)$ is a cusp on $\partial \Omega$. This contradicts $f(w) \in  S'(\tau)$ and we conclude $z\in \D$, so $(\fC^*)^{-1} (U_0) \subseteq \D$. Consequently, by (\ref{8374fgcqb8o37gd3yibced}), 
	\begin{equation*}
		(\fC^*)^{-1} (U_0) 
		\subseteq \eta \bigl(f^{-1} ( S'(\tau) ) \bigr) \cap \D 
		=f^{-1} \bigl(\tau^{-1} ( S'(\tau) ) \bigr) \cap \D 
		=U_{-1}.
	\end{equation*} %
	
	Fix an arbitrary $n\in \N$.
	
	Recall $(\fC^*)^{-1} (\D) \subseteq \D$ and $V_{-n} =f^{-1} \bigl(\tau^{-n} \bigl(\overline{\Omega}^c\bigr)\bigr) \cap \D$. By (\ref{8374fgcqb8o37gd3yibced}), 
	\begin{equation*}
		(\fC^*)^{-1} (V_{-n}) 
		\subseteq \eta \bigl(f^{-1} \bigl(\tau^{-n} \bigl(\overline{\Omega}^c\bigr)\bigr)\bigr) \cap \D
		= f^{-1} \bigl(\tau^{-(n+1)} \bigl(\overline{\Omega}^c\bigr)\bigr) \cap \D 
		=V_{-n-1}.
	\end{equation*}
	
	Recall $f(\partial \D) =\partial \Omega$, so $f^{-1} \bigl(\tau^{-n} \bigl( \overline{\Omega}^c \bigr)\bigr) \subseteq f^{-1} (\Omega) \subseteq \D \cup \D^*$. %
	By (\ref{8374fgcqb8o37gd3yibced}) we have 
	\begin{equation*}
		\begin{aligned}
			 (\fC^*)^{-1} (V_n) 
			&\subseteq \eta \bigl(f^{-1} \bigl(\tau^{-n} \bigl(\overline{\Omega}^c\bigr)\bigr)\bigr) 
			=\bigl(\eta \bigl(f^{-1} \bigl(\tau^{-n} \bigl(\overline{\Omega}^c\bigr)\bigr)\bigr) \cap \D\bigr) \cup \bigl(\eta \bigl(f^{-1} \bigl(\tau^{-n} \bigl(\overline{\Omega}^c\bigr)\bigr)\bigr) \cap \D^*\bigr) \\
			&=\bigl(f^{-1} \bigl( \tau^{-n-1} \bigl( \overline{\Omega}^c\bigr)\bigr) \cap \D\bigr) \cup \bigl(f^{-1} \bigl( \tau^{-n+1} \bigl( \overline{\Omega}^c\bigr)\bigr) \cap \D^*\bigr) 
			=V_{-n-1} \cup V_{n-1}.
		\end{aligned}
	\end{equation*}
	
	Recall $(\fC^*)^{-1} (\D) \subseteq \D$ and $U_{-n} =f^{-1} (\tau^{-n} ( S'(\tau) )) \cap \D$. By (\ref{8374fgcqb8o37gd3yibced}), 
	\begin{equation*}
		(\fC^*)^{-1} (U_{-n}) 
		\subseteq \eta \bigl(f^{-1} (\tau^{-n} ( S'(\tau) )) \bigr) \cap \D 
		=f^{-1} \bigl(\tau^{-(n+1)} ( S'(\tau) ) \bigr) \cap \D 
		=U_{-n-1}.
	\end{equation*}
	
	Now we show $(\fC^*)^{-1} (U_n) \subseteq U_{n-1} \cup U_{-n}$. Recall $U_n =f^{-1} \bigl(\tau^{-(n-1)} ( S'(\tau) ) \bigr) \cap \D^*$.
	
	If $n \geq 2$, then $f^{-1} \bigl(\tau^{-n+1} ( S'(\tau) ) \bigr) \subseteq f^{-1} (\Omega) \subseteq \D \cup \D^*$. By (\ref{8374fgcqb8o37gd3yibced}) we have 
	\begin{equation*}
		\begin{aligned}
			 (\fC^*)^{-1} (U_n) 
			&\subseteq \eta \bigl(f^{-1} \bigl(\tau^{-(n-1)} ( S'(\tau) )\bigr)\bigr) 
			=\bigl(\eta \bigl(f^{-1} \bigl(\tau^{-(n-1)} ( S'(\tau) ) \bigr)\bigr) \cap \D \bigr) \cup \bigl( \eta \bigl(f^{-1} \bigl(\tau^{-(n-1)} ( S'(\tau) )\bigr)\bigr) \cap \D^*\bigr) \\
			&=\bigl(f^{-1} (\tau^{-n} ( S'(\tau) )) \cap \D \bigr) \cup \bigl(f^{-1} \bigl(\tau^{-(n-2)} ( S'(\tau) )\bigr) \cap \D^*\bigr) 
			=U_{-n} \cup U_{n-1}.
		\end{aligned}
	\end{equation*}

	Assume $n=1$. if $w \in U_1$ and $z \in (\fC^*)^{-1} (w)$, then $f(\eta (z)) =f(w) \in  S'(\tau) $ implies $z\in \D \cup \partial \D$. If $z \in \partial \D$, then $f(z) =f(\eta (z)) \in  S'(\tau) $, and thus $z \in f^{-1} ( S'(\tau) ) \cap \partial \D =U_0$. %
	If $z \in \D$, then $\tau (f(z)) =f (\eta (z)) \in  S'(\tau)$, and thereby, $z \in \tau^{-1} \bigl(f^{-1} ( S'(\tau) )\bigr) \cap \D =U_{-1}$. Therefore, we have $(\fC^*)^{-1} (U_1) \subseteq U_0 \cup U_{-1}$.
\end{proof}

\begin{prop}\label{ug97o8vyutic7fuyghbnjmk}
	If $\cQ \in \tpk \bigl(\CDach,\CDach ; \fC^* \bigr)$ and $\mu \in \MMM \bigl(\CDach, \cQ\bigr)$, then $\mu \bigl(\widetilde{K(\tau)}\bigr) =1$.
\end{prop}

\begin{proof}
	By \cite[Theorem~3.1]{MA99}, if $A_1 ,\, A_2 \in \SBB \bigl(\CDach\bigr)$ satisfy $(\fC^*)^{-1} (A_1) \subseteq A_2$, then $\mu (A_2) \geq \mu (A_1)$. 
	
	As a result, the statements in Proposition~\ref{C^*^-1 on V_n} indicate $\mu (V_{-n-1}) \geq \mu (V_{-n})$ and $\mu (U_{-n-1}) \geq \mu (U_{-n})$ for all $n\in \N_0$. For all $n\in \N_0$ and $k\in \N$, we have 
	$k\mu (V_{-n}) \leq \sum_{j=0}^{k-1} \mu (V_{-n-j}) \leq \mu \bigl(\widetilde{T^\infty (\tau)}\bigr)$
	by Lemma~\ref{V_n and U_n are partition}. Hence, $\mu (V_{-n}) \leq \frac{1}{k}$, and thus $\mu (V_{-n}) =0$. Similarly, $\mu (U_{-n}) =0$ for each $n\in \N_0$.
	
	Again, by Proposition~\ref{C^*^-1 on V_n}, for each $n\in \N$ we have $\mu (V_n) \leq \mu (V_{n-1}) +\mu (V_{-n-1}) =\mu (V_{n-1})$ and $\mu (U_n) \leq \mu (U_{n-1}) +\mu (U_{-n}) =\mu (U_{n-1})$. This implies $\mu (V_n) \leq \mu (V_{n-1}) \leq \cdots \leq \mu (V_0) =0$, and thus $\mu (V_n) =0$ for every $n\in \N$. Similarly, $\mu (U_n) =0$ for every $n\in \N$.
	
	By Lemma~\ref{V_n and U_n are partition}, we conclude $\mu \bigl(\widetilde{T^\infty (\tau)}\bigr) =0$, and therefore, $\mu \bigl(\widetilde{K(\tau)}\bigr)=1$.
\end{proof}

Propositions~\ref{wiodpeac},~\ref{ug97o8vyutic7fuyghbnjmk}, and~\ref{aiubcw092d9ua3wwc} imply the following corollary.

\begin{cor}\label{hhhhh09huh}
	If $\phi \in C(\cO_2 (\fC^*) , \R)$, then $\fC^*|_{\widetilde{K(\tau)}}$ is a correspondence on $\widetilde{K(\tau)}$ and $P(\fC^* , \phi) =P\bigl(\fC^*|_{\widetilde{K(\tau)}} , \phi\bigr)$.
\end{cor}

\begin{prop}[{\cite[Proposition~2.5]{LMM24}}]\label{dynamics of LMM correspondence}
	 The following statements are true:
	\begin{enumerate}
		\smallskip
		\item[(i)] For all $z\in \widetilde{K(\tau)}$, we have $\# \fC^*|_{\widetilde{K(\tau)}} (z)= \# \bigl(\fC^*|_{\widetilde{K(\tau)}}\bigr)^{-1} (z)= d$.
		\smallskip
		\item[(ii)] If $w \in \fC^*|_{\widetilde{K(\tau)}} (z)$, then $z\in \widetilde{K(\tau)} \cap \overline{\D^*}$ implies that $w\in \widetilde{K(\tau)} \cap \overline{\D^*}$, moreover, $w\in \widetilde{K(\tau)} \cap \overline{\D}$ implies that $z\in \widetilde{K(\tau)} \cap \overline{\D}$.
		\smallskip
		\item[(iii)] The correspondence $\fC^*|_{\widetilde{K(\tau)}}$ has one forward branch carrying $\widetilde{K(\tau)}\cap \overline{\D}$ onto itself with degree $d$, which is topologically conjugate to $\tau \: K(\tau) \to K(\tau)$, and the remaining forward branches carry $\widetilde{K(\tau)} \cap \overline{\D}$ onto $\widetilde{K(\tau)} \cap \overline{\D^*}$.
		\smallskip
		\item[(iv)] The correspondence $\fC^*|_{\widetilde{K(\tau)}}$ has a backward branch carrying $\widetilde{K(\tau)} \cap \overline{\D^*}$ onto itself with degree $d$, which is topologically conjugate to $\tau \: K(\tau) \to K(\tau)$, and the remaining backward branches carry $\widetilde{K(\tau)} \cap \overline{\D^*}$ onto $\widetilde{K(\tau)} \cap \overline{\D}$.
	\end{enumerate}
\end{prop}

Write $G\= \fC^*|_{\widetilde{K(\tau)}}$. Proposition~\ref{dynamics of LMM correspondence}~(iii) indicates that $\fC^*|_{\widetilde{K(\tau)} \cap \overline{\D}}$ is induced by a single-valued continuous map, so we suppose it is induced by $g \: \widetilde{K(\tau)} \cap \overline{\D} \to \widetilde{K(\tau)} \cap \overline{\D}$, i.e., $\fC^*|_{\widetilde{K(\tau)} \cap \overline{\D}} =\cC_g$ (recall (\ref{e:Def_C_f})). Similarly, Proposition~\ref{dynamics of LMM correspondence}~(iv) indicates that $\bigl(\fC^*|_{\widetilde{K(\tau)} \cap \overline{\D^*}} \bigr)^{-1} =\cC_{g^*}$, where $g^* \: \widetilde{K(\tau)} \cap \overline{\D^*} \to \widetilde{K(\tau)} \cap \overline{\D^*}$ is a single-valued continuous map.

Let $\phi \in C(\cO_2 (\fC^*) , \R)$. Note that each $(x,y) \in \cO_2 (\cC_g)$ is of the form $(x, g(x))$, so functions on $\cO_2 \bigl( \fC^*|_{\widetilde{K(\tau)} \cap \overline{\D}} \bigr) =\cO_2 (\cC_g)$ actually only depend on the first coordinate. As a result, we can choose $\varphi \in C\bigl(\widetilde{K(\tau)} \cap \overline{\D} , \R\bigr)$ such that $\varphi (x) =\phi (x,y)$ for all $(x,y) \in \cO_2 \bigl( \fC^*|_{\widetilde{K(\tau)} \cap \overline{\D}} \bigr)$. Similarly, we can choose $\varphi^* \in C\bigl(\widetilde{K(\tau)} \cap \overline{\D^*} , \R\bigr)$ such that $\varphi^* (y) =\phi (x,y)$ for all $(x,y) \in \cO_2 \bigl( \fC^*|_{\widetilde{K(\tau)} \cap \overline{\D^*}} \bigr)$.

With the dynamics of $G=\fC^*|_{\widetilde{K(\tau)}}$ given by Proposition~\ref{dynamics of LMM correspondence}, we estimate the topological pressure $P(G, \phi)$ (see Subsection~\ref{subsct_Definition of topological pressure for correspondences}).%

\begin{prop}\label{23egw78cdvhy827dgiuywcvjhds}
	Let $\phi \in C(\cO_2 (\fC^*) , \R)$ and $G$, $g$, $g^*$, $\varphi$, and $\varphi^*$ be given above. We have $P(G, \phi) =\max \{P(g, \varphi),\, P(g^* ,\varphi^*)\}$.
\end{prop}

\begin{proof}
	We recall the definition of topological pressure for correspondences from Definition~\ref{topological pressure}.

	Let $d$ be the spherical metric on $\CDach$ and $d_n$ be the metrics given by (\ref{d_n=}) and~(\ref{d_omega=}). %

	For all $n\in \N$ and $\epsilon >0$, write  
	\begin{equation}\label{alpha(n,epsilon)}
		\alpha (n, \epsilon) \= \sup_{E_n (\epsilon)} \sum_{\udx \in E_n (\epsilon)} \exp (S_n \phi (\udx)) \quad \text{and}\quad
         \beta (n, \epsilon) \= \sup_{F_n (\epsilon)} \sum_{\udx \in F_n (\epsilon)} \exp (S_n \phi (\udx)),
	\end{equation}
	where $E_n (\epsilon)$ (resp.\ $F_n (\epsilon)$) ranges over all $\epsilon$-separated subsets of $(\cO_{n+1} (\cC_g) ,d_{n+1})$ (resp.\ $\bigl(\cO_{n+1} \bigl(\cC_{g^*}^{-1}\bigr), d_{n+1}\bigr)$).
	
	Recall $\fC^*|_{\widetilde{K(\tau)} \cap \overline{\D}} =\cC_g$, $\fC^*|_{\widetilde{K(\tau)} \cap \overline{\D^*}} =\cC_{g^*}^{-1}$, $\varphi (x) =\phi (x,y)$ for all $(x,y) \in \cO_2 (\cC_g)$, and $\varphi^* (y) =\phi (x,y)$ for all $(x,y) \in \cO_2 \bigl(\cC_{g^*}^{-1}\bigr)$. By Propositions~\ref{P(f,phi)=P(C_f,phi)} and~\ref{q3f33333333333333333f}, we have $P(g , \varphi) =P\bigl(\fC^*|_{\widetilde{K(\tau)} \cap \overline{\D}} , \phi \bigr)$ and $P(g^* , \varphi^*) =P\bigl(\fC^*|_{\widetilde{K(\tau)} \cap \overline{\D^*}} , \phi \bigr)$. By (\ref{alpha(n,epsilon)}) and Definition~\ref{topological pressure}, we have
	\begin{equation}\label{39hf4qf0qddddddddddddddddddddddq3h9qw}
		\begin{aligned}
			P(g , \varphi) =\lim_{\epsilon \to 0^+} \limsup_{n\to +\infty} \frac{1}{n} \log (\alpha (n, \epsilon))   
			\quad \text{and} \quad
			P(g^* , \varphi^*) =\lim_{\epsilon \to 0^+} \limsup_{n\to +\infty} \frac{1}{n} \log (\beta (n, \epsilon)).			
		\end{aligned}		
	\end{equation}
	
	Fix arbitrary $\epsilon >0$, $n\in \N$, and $\epsilon$-separated subset $W_n (\epsilon)$ of $\cO_{n+1}  (G  )$. By Proposition~\ref{dynamics of LMM correspondence}~(ii), we can write $W_n (\epsilon) = \bigcup_{k=-1}^n W_{n ,  k} (\epsilon)$, where 
	\begin{equation*}
		W_{n ,  k} (\epsilon) \=\{ \vect{x}{0}{n}  \in W_n (\epsilon)  :  x_i \in \D \text{ for } i \leq k \text{ and } x_i \in \D^* \text{ for } i>k \}.
	\end{equation*}
	Fix an arbitrary maximal $\frac{\epsilon}{2}$-separated subset $E_k (\epsilon /2)$ of $\cO_{k+1} (\cC_g)$. For each $k\in \zeroto{n}$ and each $\vect{x}{0}{k}  \in E_k ( \epsilon / 2)$, set 
	\begin{equation*}
		W_{n ,  k ,  \vect{x}{0}{k} } (\epsilon) \=\{ \vect{y}{0}{n}  \in W_{n ,  k} (\epsilon)  :  d(x_i ,  y_i)<  \epsilon / 2  \text{ for all } i \in \zeroto{k}  \},
	\end{equation*}
	Then for each $\vect{y}{0}{n}  \in W_{n ,  k} (\epsilon)$, the maximality ensures that there must be some $\vect{x}{0}{k}\in E_k (\epsilon / 2)$ such that $d(x_i ,  y_i)< \epsilon / 2$ for all $i \in \zeroto{k}$, so
	\begin{equation}\label{w4837gcbx04b37qynqfrfe}
		\bigcup_{\vect{x}{0}{k} \in E_k (\frac{\epsilon}{2})} W_{n ,  k ,  \vect{x}{0}{k}} (\epsilon) =W_{n ,  k} (\epsilon).
	\end{equation}
	
	Fix arbitrary $k\in \zeroto{n-1}$ and $\vect{x}{0}{k} \in E_k (\epsilon / 2)$. For each $\udy =\vect{y}{0}{n} \in W_{n,k,\vect{x}{0}{k}} (\epsilon)$, $d(x_i ,y_i) <\epsilon /2$ for all $i\in \zeroto{k}$ implies that 
	\begin{equation*}
		\sum_{j=0}^{k-1} \phi (y_j , y_{j+1}) \leq \sum_{j=0}^{k-1} \phi (x_j , x_{j+1}) +k \Delta \Bigl(\phi , \frac{\epsilon}{2}\Bigr),
	\end{equation*}
	where $\Delta (\phi , \delta) \= \sup \{\abs{\phi (x_1 ,x_2) -\phi (y_1 ,y_2)} : d(x_1 ,y_1) < \delta$ and $d(x_2 ,y_2)<\delta \}$ for all $\delta >0$. So
	\begin{equation}\label{w9e7guidsbh}
		S_n \phi (\udy) \leq \sum_{j=0}^{k-1} \phi (x_j , x_{j+1}) +k \Delta \Bigl(\phi , \frac{\epsilon}{2} \Bigr) +\norm{\phi}_\infty +\sum_{j=k+1}^{n-1} \phi (y_j , y_{j+1}).
	\end{equation}
	
	Because $W_{n ,  k ,  \vect{x}{0}{k}} (\epsilon)$ is contained in $W_n (\epsilon)$, an $\epsilon$-separated subset of $\cO_{n+1} (G)$, for each pair of distinct orbits $\vect{y}{0}{n} , \,\vect{z}{0}{n} \in W_{n ,  k ,  \vect{x}{0}{k}} (\epsilon)$, there exists $l \in  \zeroto{n}$ such that $d(y_l ,  z_l)\geq \epsilon$. Such an integer $l$ must be greater than $k$, because for each $j\in \zeroto{k}$, we have $d(y_j ,  z_j)\leq d(x_j ,  y_j)+ d(x_j ,  z_j)< \frac{\epsilon}{2} +\frac{\epsilon}{2} =\epsilon$. So $\bigl\{ \vect{y}{k+1}{n}   :  \vect{y}{1}{n}   \in W_{n ,  k ,  \vect{x}{0}{k}} (\epsilon) \bigr\}$ is an $\epsilon$-separated subset of $\cO_{n-k} \bigl( \cC_{g^*}^{-1} \bigr)$. Thus, by (\ref{w9e7guidsbh}) and~(\ref{alpha(n,epsilon)}),
	\begin{equation*}
		\sum_{\udy \in W_{n,k,\vect{x}{0}{k}} (\epsilon)} e^{S_n \phi (\udy)} \leq \beta (n-k-1 , \epsilon) \exp \biggl(\sum_{j=0}^{k-1} \phi (x_j , x_{j+1}) +k \Delta \Bigl(\phi , \frac{\epsilon}{2} \Bigr) +\norm{\phi}_\infty\biggr).
	\end{equation*}
	
	Consequently, by (\ref{w4837gcbx04b37qynqfrfe}) and~(\ref{alpha(n,epsilon)}) we have
	\begin{equation}\label{wev0h38uced}
		\sum_{\udy \in W_{n,k} (\epsilon)} e^{S_n \phi (\udy)} \leq \alpha \Bigl(k, \frac{\epsilon}{2}\Bigr) \beta (n-k-1 ,\epsilon)\exp \Bigl(n \Delta \Bigl(\phi , \frac{\epsilon}{2}\Bigr) +\norm{\phi}_\infty\Bigr).
	\end{equation}
	
	Note that (\ref{wev0h38uced}) holds for $k\in \zeroto{n-1}$, so we need to consider $k=-1$ and $k=n$.
	Indeed, $W_{n ,  -1} (\epsilon) =\{ \vect{x}{0}{n}  \in W_n (\epsilon)  : x_i \in \D^* \text{ for all } i\in \zeroto{n} \}$ is $\epsilon$-separated in $\cO_{n+1} \bigl(\cC_{g^*}^{-1} \bigr)$ and $W_{n ,  n} (\epsilon) =\{ \vect{x}{0}{n}  \in W_n (\epsilon)  : x_i \in \D \text{ for all } i\in \zeroto{n} \}$ is $\epsilon$-separated in $\cO_{n+1} (\cC_g)$. By (\ref{alpha(n,epsilon)}), we have $\sum_{\udy \in W_{n ,-1} (\epsilon)} \exp (S_n \phi (\udy)) \leq \beta (n ,\epsilon)$ and $\sum_{\udy \in W_{n ,n} (\epsilon)} \exp (S_n \phi (\udy)) \leq \alpha (n ,\epsilon)$. By (\ref{wev0h38uced}) and since $W_n (\epsilon) = \bigcup_{k=-1}^n W_{n ,  k} (\epsilon)$, we have
	\begin{equation*}
		\begin{aligned}
			\sum_{\udy \in W_n (\epsilon)} e^{S_n \phi (\udy)} &\leq \alpha (n ,\epsilon) +\beta (n ,\epsilon)+ e^{n \Delta (\phi , \frac{\epsilon}{2}) +\norm{\phi}_\infty} \sum_{k=1}^{n-1} \alpha \Bigl(k, \frac{\epsilon}{2}\Bigr) \beta (n-k-1 ,\epsilon).
		\end{aligned}
	\end{equation*}

	Since $\lim_{\epsilon \to 0^+} \Delta \bigl(\phi , \frac{\epsilon}{2}\bigr) =0$ due to the uniform continuity of $\phi$, applying (\ref{39hf4qf0qddddddddddddddddddddddq3h9qw}) we conclude
	\begin{equation}\label{o2ewhin98oi}
		\begin{aligned}
				P (G ,\phi) =\lim_{\epsilon \to 0^+} \limsup_{n \to +\infty} \frac{1}{n} \log \Bigl( \sup_{W_n (\epsilon)} \sum_{\udy \in W_n (\epsilon)} e^{S_n \phi (\udy)} \Bigr) 
				\leq \max \{P(g, \varphi),\, P(g^* ,\varphi^*)\}.
		\end{aligned}
	\end{equation}
	
	Additionally, for arbitrary $n\in \N$ and $\epsilon >0$, every $\epsilon$-separated subset of $\cO_{n+1} (\cC_g)$ is also an $\epsilon$-separated subset of $\cO_{n+1} (G )$. Thus, by Definition~\ref{topological pressure} we have $P (G , \phi) \geq P (\cC_g , \phi  )$. Similarly, we have $P (G , \phi) \geq P\bigl(\cC_{g^*}^{-1} , \phi \bigr)$. Recall $P (\cC_g , \phi  ) =P(g, \varphi)$ and $P\bigl(\cC_{g^*}^{-1} , \phi \bigr) =P(g^* ,\varphi^*)$. Hence, $P (G ,\phi) \geq \max \{P(g, \varphi),\, P(g^* ,\varphi^*)\}$. Therefore, by (\ref{o2ewhin98oi}) we conclude $P (G ,\phi) =\max \{P(g, \varphi),\, P(g^* ,\varphi^*)\}$.
\end{proof}

Now, we prove Theorem~\ref{theo_matings}.

\begin{proof}[Proof of Theorem~\ref{theo_matings}]
	In this proof, if $\mu$ is a Borel probability measure on some Borel subset $K$ of $\CDach$, then $\hmu$ will refer to the Borel probability measure on $\CDach$ given by $\hmu (A) \= \mu (K \cap A)$ for all $A \in \SBB \bigl( \CDach \bigr)$. Corollary~\ref{hhhhh09huh} and Proposition~\ref{23egw78cdvhy827dgiuywcvjhds} imply that $P(\fC^* ,\phi) =\max \{P(g, \varphi),\, P(g^* ,\varphi^*)\}$. We establish Theorem~\ref{theo_matings} by discussing the following two cases:
	
	\smallskip
	
	\emph{Case~1.} $P(\fC^* ,\phi) =P(g, \varphi)$.
	
	\smallskip
	
	By the classical Variational Principle, we have
	\begin{equation}\label{VP for g}
		P(g ,\varphi) = \sup_{\mu \in \MMM(\widetilde{K(\tau)} \cap \overline{\D} , g)} \bigg\{h_\mu (g) +\int_{\widetilde{K(\tau)} \cap \overline{\D}} \! \varphi \, \mathrm{d} \mu\bigg\} .
	\end{equation}

	Fix an arbitrary $\mu \in \MMM \bigl( \widetilde{K(\tau)} \cap \overline{\D}, g \bigr)$. By Lemma~\ref{2ex invariant measure (1)} and Proposition~\ref{h_mu(F)=h_mu(hF)}, $\mu$ is $\hg$-invariant and $h_\mu (g) =h_\mu (\hg)$, where $\hg$ is the transition probability kernel on $\widetilde{K(\tau)} \cap \overline{\D}$ induced by $g$ given in Definition~\ref{2ex transition probability kernel (1)}. We choose a Borel measurable branch $a_0$ of $\fC^*$, where the existence of $a_0$ is ensured by \cite[Lemma~1.1]{MA99}. Let $\cS \in \tpk \bigl(  \CDach, \CDach \bigr)$ be given by 
	\begin{equation*}
		\cS (z, A) \= \begin{cases}
		\mathbbold{1}_A (g(z)) &\text{if } z\in \widetilde{K(\tau)} \cap \overline{\D};\\
		\mathbbold{1}_A (a_0(z)) &\text{if } z\notin \widetilde{K(\tau)} \cap \overline{\D}
		\end{cases}
	\end{equation*} 
	for all $z\in \CDach$ and $A \in \SBB \bigl(\CDach\bigr)$. It follows that $\cS_z =\delta_{g(z)}$ for all $z\in \widetilde{K(\tau)} \cap \overline{\D}$ and that $\cS$ is supported by $\fC^*$. By Corollary~\ref{restrict measure-theoretic entropy}, $\hmu$ is $\cS$-invariant and $h_{\hmu} (\cS) =h_\mu (\hg) =h_\mu (g)$. Recall $\varphi (z) =\phi (z, g(z))$ for all $z \in \widetilde{K(\tau)} \cap \overline{\D}$. We have
	\begin{equation}\label{72w827d723}
	\begin{aligned}
		\int_{\CDach} \! \int_{\fC^* (z)} \! \phi (z,w) \diff \cS_z (w) \diff \hmu (z) 
		&=\int_{\widetilde{K(\tau)} \cap \overline{\D}} \! \int_{\fC^* (z)} \! \phi (z,w) \diff \delta_{g(z)} (w) \diff \mu (z)\\
		& =\int_{\widetilde{K(\tau)} \cap \overline{\D}} \! \phi (z ,g(z)) \diff \mu (z)  %
		=\int_{\widetilde{K(\tau)} \cap \overline{\D}} \! \varphi \diff \mu.
	\end{aligned}
	\end{equation}
	Recall that $\cS$ is supported on $\fC^*$. By (\ref{VP for g}), (\ref{72w827d723}), $h_{\hmu} (\cS) =h_\mu (g)$, and $P(\fC^* ,\phi) =P(g, \varphi)$, we have
	\begin{equation}\label{P(T,phi)=sup_Q,mu(h_a)}
		P(\fC^* ,\phi)\leq \sup_{\cQ \in \tpk( \CDach,\CDach;\fC^* ), \, \mu \in \MMM( \CDach, \cQ)} \biggl\{ h_\mu (\cQ) +\int_{\CDach} \! \int_{\fC^* (x_1)} \! \phi (x_1 ,  x_2)   \diff \cQ_{x_1} (x_2)   \diff \mu (x_1) \biggr\}.
	\end{equation}
	Therefore, (\ref{P(T,phi)=sup_Q,mu(h_mu(Q)+int_Xphidmu)}) in this setting follows by Theorem~\ref{t_HVP}.

	\smallskip
	
	\emph{Case~2.} $P(\fC^* ,\phi) =P(g^* ,\varphi^*)$.
	
	\smallskip
	
	By the classical Variational Principle, we have
	\begin{equation}\label{VP for g^*}
		P(g^* ,\varphi^*) = \sup_{\mu \in \MMM( \widetilde{K(\tau)} \cap \overline{\D^*},  g^*  )} 
		      \bigg\{h_\mu (g^*) +\int_{\widetilde{K(\tau)} \cap \overline{\D^*}} \! \varphi^* \, \mathrm{d} \mu\bigg\}.
	\end{equation}

	We proceed in a manner similar to the previous case. Fix an arbitrary $\mu \in \MMM \bigl( \widetilde{K(\tau)} \cap \overline{\D^*}, g^*\bigr)$. We choose a Borel measurable branch $a_1$ of $(\fC^*)^{-1}$. Let $\cS \in \tpk \bigl( \CDach, \CDach  \bigr)$ be given by 
	\begin{equation*}
		\cS (z, A) \= 
		\begin{cases}
		\mathbbold{1}_A (g^*(z)) &\text{if } z\in \widetilde{K(\tau)} \cap \overline{\D^*};\\
		\mathbbold{1}_A (a_1(z)) &\text{if } z\notin \widetilde{K(\tau)} \cap \overline{\D^*}
		\end{cases}
	\end{equation*}
	for all $z\in \CDach$ and $A \in \SBB \bigl(\CDach\bigr)$. It is supported by $(\fC^*)^{-1}$, and thus the measure $\hmu \cS^{[1]}$ is supported on $\cO_2 \bigl((\fC^*)^{-1} \bigr)$. By Proposition~\ref{h_mu(F)=h_mu(hF)}, Lemma~\ref{2ex invariant measure (1)}, and Corollary~\ref{restrict measure-theoretic entropy}, $\hmu$ is $\cS$-invariant and $h_{\hmu} (\cS) =h_\mu (g^*)$. Recall $\varphi^* (z) =\phi (g^* (z), z)$ for all $z \in \widetilde{K(\tau)} \cap \overline{\D^*}$. By (\ref{int_O_+infty(T)tphidmuQ^N|_T=intphidmu}) in Lemma~\ref{l:int_O_+infty(T)tphidmuQ^N|_T=intphidmu}, we have
	\begin{equation}\label{728f3e222222}
		\int_{\CDach^2} \! \phi (w ,z) \diff \bigl(\hmu \cS^{[1]}\bigr) (z, w) =\int_{\CDach} \! \int_{\CDach} \! \phi (w, z) \diff \cS_z (w) \diff \hmu (z) =\int_{\widetilde{K(\tau)} \cap \overline{\D^*}} \! \varphi^* \diff \mu,
	\end{equation}
	which corresponds to (\ref{72w827d723}) in the previous case.
	By (\ref{muQ^[1]circhpi_2^-1=muQ}), we have $\bigl(\hmu \cS^{[1]} \bigr) \circ \tpi_2^{-1} =\hmu \cS =\hmu$. Choose a backward conditional transition probability kernel $\cR$ of $\hmu \cS^{[1]}$ from $\CDach$ to $\CDach$ supported by $\cO_2 ((\fC^*)^{-1})$. Definition~\ref{backward conditional transition probability kernel}~(a) and the fact that $\mu \cS^{[1]}$ is supported by $(\fC^*)^{-1}$ indicate that $\cR$ is supported on $\fC^*$. By Remark~\ref{nu=muQ^[1]}, Definition~\ref{backward conditional transition probability kernel}~(b) leads to $\bigl( \hmu \cS^{[1]}\bigr) \circ \gamma_2^{-1} =\hmu \cR^{[1]}$,  where $\gamma_2 (z ,w) =(w, z)$ for all $z ,\, w\in \CDach$. By Proposition~\ref{measure-theoretic entropy is inversely invariant}, $\hmu$ is $\cR$-invariant and $h_{\hmu} (\cR) =h_{\hmu} (\cS) =h_\mu (g^*)$. By (\ref{int_O_+infty(T)tphidmuQ^N|_T=intphidmu}), (\ref{728f3e222222}), and $\bigl( \hmu \cS^{[1]}\bigr) \circ \gamma_2^{-1} =\hmu \cR^{[1]}$, we have
	\begin{equation}\label{28yw121w}
		\begin{aligned}
			\int_{\CDach} \! \int_{\fC^* (z)} \! \phi (z,w) \diff \cR_z (w) \diff \hmu (z) 
			=\int_{\CDach^2} \! \phi \diff \bigl(\hmu \cR^{[1]}\bigr) 
			=\int_{\CDach^2} \! \phi\circ \gamma_2 \diff \bigl(\hmu \cS^{[1]}\bigr) 
			=\int_{\widetilde{K(\tau)} \cap \overline{\D^*}} \! \varphi^* \diff \mu.			
		\end{aligned}
	\end{equation}
	Recall $\cR$ is supported by $\fC^*$. By (\ref{VP for g^*}), (\ref{28yw121w}), $h_{\hmu} (\cR) =h_\mu (g^*)$, and $P(\fC^* ,\phi) =P(g^*, \varphi^*)$, we get that (\ref{P(T,phi)=sup_Q,mu(h_a)}) holds, and therefore, (\ref{P(T,phi)=sup_Q,mu(h_mu(Q)+int_Xphidmu)}) in this setting follows by Theorem~\ref{t_HVP}.
\end{proof}

\subsection{A family of hyperbolic holomorphic correspondences}

In this subsection, we aim to establish Theorems~\ref{theo_hyperboliccorrespondence} and~\ref{theo_hyperboliccorrespondence2}. To begin with, we explain in detail the definition of Julia set of $\boldsymbol{f}_c =z^{q/p} +c$ mentioned in Section~\ref{sct_Introduction}, see e.g., \cite[Definition~6.31]{Siq15} or \cite[Section~2.1]{SS17}.

A \defn{periodic orbit}\index{periodic orbit} is a sequence $\{ z_k \}_{k=0}^n$, where $n\in \N$, $z_k \in \CDach$ for each $k\in \zeroto{n}$, satisfying $z_k \in \boldsymbol{f}_c (z_{k-1})$ for each $k\in \oneto{n}$ and $z_n =z_0$. For each $k\in \oneto{n}$, if $z_{k-1}$ does not belong to $\{ 0 ,\, \infty \}$, we can choose a branch of the holomorphic function $\phi_k \: z\mapsto \exp \bigl( \frac{1}{q} \log (z^p) \bigr) +c$ in a neighborhood of $z_{k-1}$ which maps $z_{k-1}$ to $z_k$. We say that a periodic orbit $\{ z_k \}_{k=0}^n$ is \defn{repelling}\index{repelling} if none of elements in that orbit belong to $\{ 0 ,\, \infty \}$ and $ \abs{( \phi_n \circ \cdots \circ \phi_2 \circ \phi_1 )' (z_0)} >1$. The \defn{Julia set}\index{Julia set} $J(\boldsymbol{f}_c)$ is defined as the closure of the union of all repelling periodic orbits of $\boldsymbol{f}_c$.

Recall $P_c \= \overline{\bigcup_{n\in \N} \boldsymbol{f}_c^n (0)}$. %
The following proposition is formulated from \cite[Theorems~4.4,~5.1, and~5.8]{Siq22}.
    
\begin{prop}\label{dynamics of z^q/p+c for cinH_q/p}
	There is an open set $H_{q/p}$ containing both $\C \smallsetminus M_{q/p ,  0}$ and every simple center with the property that for each $c \in H_{q/p}$, the following statements are true:
	\begin{enumerate}
		\smallskip
		\item[(i)] $\boldsymbol{f}_c^{-1} (J (\boldsymbol{f}_c)) =J (\boldsymbol{f}_c) \neq \emptyset$.
		\smallskip
		\item[(ii)] The set $\C \smallsetminus P_c$ is a hyperbolic Riemann surface and $J (\boldsymbol{f}_c) \subseteq \C \smallsetminus P_c$.
		\smallskip
		\item[(iii)] If we denote by $d_c$ the hyperbolic metric on $\C \smallsetminus P_c$, then there exist constants $\lambda >1$ and $\delta >0$ depending on $p$, $q$, and $c$ with the following property:
		
		\smallskip
		
		If a pair of distinct points $z_1 ,\, z_2 \in J(\boldsymbol{f}_c)$ satisfy $d_c (z_1 ,  z_2) <\delta$, then for each $w_1 \in J (\boldsymbol{f}_c) \cap \boldsymbol{f}_c (z_1)$ and each $w_2 \in J (\boldsymbol{f}_c) \cap \boldsymbol{f}_c (z_2)$, we have $d_c (w_1 ,  w_2) >\lambda d_c (z_1 ,  z_2)$.
		\smallskip
		\item[(iv)] For every open set $V$ in $\C$ that intersects $J (\boldsymbol{f}_c)$, there exists $n \in \N$ with $\boldsymbol{f}_c^n (V \cap J (\boldsymbol{f}_c)) =J (\boldsymbol{f}_c)$.
	\end{enumerate}
\end{prop}

The following proposition about $\boldsymbol{f}_c$ when $c$ is closed to $0$ is formulated from \cite[Corollaries~4.6, 4.8, Theorems~3.5, and~2.7]{Siq15}.

\begin{prop}\label{dynamics of z^q/p+c for cinU_q/p}
	There is an open neighborhood $U_{q/p}$ of $0$ such that for every $c \in U_{q/p}$, the following statements are true:
	\begin{enumerate}
		\smallskip
		\item[(i)] $\boldsymbol{f}_c^{-1} (J (\boldsymbol{f}_c)) =J (\boldsymbol{f}_c) \neq \emptyset$.
		\smallskip
		\item[(ii)] There exist constants $\lambda >1$ and $\delta >0$ depending on $p$, $q$, and $c$ with the following property:
		
		\smallskip
		
		If a pair of distinct points $z_1 ,\, z_2 \in J(\boldsymbol{f}_c)$ satisfy $ \abs{z_1 -z_2} <\delta$, then for each $w_1 \in J (\boldsymbol{f}_c) \cap \boldsymbol{f}_c (z_1)$ and each $w_2 \in J (\boldsymbol{f}_c) \cap \boldsymbol{f}_c (z_2)$, we have $ \abs{w_1 -w_2} >\lambda  \abs{z_1 -z_2}$.
		\smallskip
		\item[(iii)] For every open set $V$ in $\C$ that intersects $J (\boldsymbol{f}_c)$, there exists $n \in \N$ with $\boldsymbol{f}_c^n (V \cap J (\boldsymbol{f}_c)) =J (\boldsymbol{f}_c)$.
	\end{enumerate}
\end{prop}

Recall $\boldsymbol{f}_c |_J (z) =J(\boldsymbol{f}_c) \cap \boldsymbol{f}_c (z)$ for all $z \in J (\boldsymbol{f}_c)$. 

\begin{proof}[Proof of Theorems~\ref{theo_hyperboliccorrespondence} and~\ref{theo_hyperboliccorrespondence2}]
	Note that $0\in M_{q/p}$ is not a simple center. We choose the open sets $H_{q/p}$ and $U_{q/p}$ as in Propositions~\ref{dynamics of z^q/p+c for cinH_q/p} and~\ref{dynamics of z^q/p+c for cinU_q/p}, respectively, such that $H_{q/p} \cap U_{q/p} =\emptyset$. For every $c\in H_{q/p}$, we denote by $d_c$ the hyperbolic metric on the hyperbolic Riemann surface $\C \smallsetminus P_c$, where the hyperbolicity of $\C \smallsetminus P_c$ is ensured by Proposition~\ref{dynamics of z^q/p+c for cinH_q/p}~(ii). For every $c \in U_{q/p}$, we denote by $d_c$ the Euclidian metric on $\C$. By Theorem~\ref{phwfq9}, it suffices to show that $\boldsymbol{f}_c |_J$ is an open, distance-expanding, topologically exact correspondence on the compact metric space $(J (\boldsymbol{f}_c),  d_c)$ for all $c \in H_{q/p} \cup U_{q/p}$.
	
	Fix an arbitrary $c \in H_{q/p} \cup U_{q/p}$.
	
	First, we show that $\boldsymbol{f}_c |_J$ is a correspondence on $J (\boldsymbol{f}_c)$. Indeed, for every $z \in J (\boldsymbol{f}_c)$, by Propositions~\ref{dynamics of z^q/p+c for cinH_q/p}~(i) and~\ref{dynamics of z^q/p+c for cinU_q/p}~(i), there is $w \in J(\boldsymbol{f}_c)$ with $z \in \boldsymbol{f}_c^{-1} (w)$, i.e., $w \in \boldsymbol{f}_c (z)$. Consequently, $\boldsymbol{f}_c |_J (z) =\boldsymbol{f}_c (z) \cap J(\boldsymbol{f}_c)$ is non-empty and closed for all $z \in J(\boldsymbol{f}_c)$. Moreover, the set $\cO_2 (\boldsymbol{f}_c |_J) =\cO_2 (\boldsymbol{f}_c) \cap J(\boldsymbol{f}_c)^2$ is closed in $J(\boldsymbol{f}_c)^2$. Hence, it follows that $\boldsymbol{f}_c |_J$ is a correspondence on $J (\boldsymbol{f}_c)$.
	
	Second, the openness of $\boldsymbol{f}_c |_J$ follows from $\boldsymbol{f}_c^{-1} (J (\boldsymbol{f}_c)) =J (\boldsymbol{f}_c)$, i.e., Propositions~\ref{dynamics of z^q/p+c for cinH_q/p}~(i) and~\ref{dynamics of z^q/p+c for cinU_q/p}~(i). Specifically, we fix arbitrary $z \in J (\boldsymbol{f}_c)$, an open neighborhood $V$ of $z$ in $J (\boldsymbol{f}_c)$, and $w \in \boldsymbol{f}_c |_J (z)$. For every point $w' \in J (\boldsymbol{f}_c)$ which is sufficiently close to $w$, a branch of $\boldsymbol{f}_c^{-1}$ gives a point $z' \in V$ such that $w' \in \boldsymbol{f}_c (z')$. This implies $z' \in \boldsymbol{f}_c^{-1} (J (\boldsymbol{f}_c)) =J (\boldsymbol{f}_c)$, so $w' \in \boldsymbol{f}_c |_J (z') \subseteq \boldsymbol{f}_c |_J (V)$. The argument above shows that $\boldsymbol{f}_c |_J (V)$ contains a neighborhood of $w$ in $J (\boldsymbol{f}_c)$. Hence, we conclude that $\boldsymbol{f}_c |_J$ is open.

    Third, Propositions~\ref{dynamics of z^q/p+c for cinH_q/p}~(iii) and~\ref{dynamics of z^q/p+c for cinU_q/p}~(ii) indicate that $\boldsymbol{f}_c |_J$, as correspondence on the compact metric space $(J (\boldsymbol{f}_c) ,  d_c)$, is distance-expanding.
    
    Fourth, by $\boldsymbol{f}_c^{-1} (J (\boldsymbol{f}_c)) =J (\boldsymbol{f}_c)$, for arbitrary $W \subseteq \CDach$ and $n\in \N$, we have $(\boldsymbol{f}_c |_J)^n (W \cap J (\boldsymbol{f}_c)) =\boldsymbol{f}_c^n (W) \cap J (\boldsymbol{f}_c)$. Thus, Propositions~\ref{dynamics of z^q/p+c for cinH_q/p}~(vi) and~\ref{dynamics of z^q/p+c for cinU_q/p}~(iii) imply that $\boldsymbol{f}_c |_J$ is topologically exact.
    
    Hence, for all $c \in H_{q/p} \cup U_{q/p}$, the correspondence $\boldsymbol{f}_c |_J$ satisfies all the hypotheses in Theorem~\ref{phwfq9}, and therefore Theorem~\ref{phwfq9} directly yields Theorems~\ref{theo_hyperboliccorrespondence} and~\ref{theo_hyperboliccorrespondence2}.
\end{proof}

\subsection{Finite cases: $(0,1)$-matrices and transition matrices}

Here we focus on the case where $X$ is a finite set.

\subsubsection{$(0,1)$-matrices and topological pressure}

Let $d\in \N$, $X=\oneto{d}$ be a finite space equipped with the discrete topology, and $A=(a_{ij})_{1\leq i,  j \leq d}$ be a $(0,1)$-matrix with at least one entry $1$ in each row. Denote by $\cC_A \: X\to \cF (X)$ the correspondence on $X$ that assigns each point $i\in X$ the subset $\{ j\in X  :  a_{ij} =1 \}$ of $X$. We have $\cC_A (i)\neq \emptyset$ for each $i\in X$, which ensures $\cC_A (i) \in \cF (X)$. %

We first compute the topological pressure of $\cC_A$. Let $\phi \: \cO_2 (\cC_A) \to \R$ be a function and $A_\phi \= \bigl( a_{ij}\cdot e^{\phi (i ,  j)} \bigr)_{1\leq i,  j\leq d}$ be a $d\times d$ matrix (if $(i ,  j) \notin \cO_2 (\cC_A)$, then $a_{ij}=0$, so in this case we do not need to define $\phi (i ,  j)$).

\begin{lemma}\label{P(C_A,phi)=logrho(A_phi)}
		If we denote by $\rho (A_\phi)$ the spectral radius of $A_\phi$, then $P(\cC_A , \phi) =\log (\rho(A_\phi))$.
\end{lemma}

\begin{proof}
	Let $n\in \N$. By definition of the metric $d_{n+1}$, the only $\epsilon$-spanning subset of $(\cO_{n+1} (\cC_A) ,  d_{n+1})$ is $\cO_{n+1} (\cC_A)$ for $\epsilon >0$ small enough. As a result, by (\ref{P(T,phi)=lim_epsilonto0^+limsup_nto+infty1/nlogsup_E_n(epsilon)sum_xinE_n(epsilon)expS_nphi(x)}) we get
	\begin{equation*}%
		\begin{aligned}
			P(\cC_A ,  \phi) &=\limsup_{n\to \infty} \frac{1}{n} \log \biggl( \sum_{\udx \in \cO_{n+1} (\cC_A)}   \exp (S_n \phi (\udx))\biggr)\\
			&= \limsup_{n\to +\infty} \frac{1}{n} \log \biggl( \sum_{i_1 ,  \dots ,  i_{n+1} =1}^d    \prod_{j=1}^{n-1}   \Bigl( a_{i_j i_{j+1}} \cdot e^{\phi (i_j ,  i_{j+1})} \Bigr) \biggr)  
			= \limsup_{n\to +\infty} \frac{ \log  (  \norm{A_\phi^n}_1  )}{n} ,
		\end{aligned}
	\end{equation*}
	where the norm $ \norm{\cdot}_1$ is given by $ \norm{B}_1 \= \sum_{i,  j=1}^d    \abs{b_{ij}}$ for every $d\times d$ matrix $B=(b_{ij})_{1\leq i,  j\leq d}$.

	By Gelfand's formula, 
	$P(\cC_A ,  \phi)
		= \limsup\limits_{n\to +\infty} \frac{1}{n} \log \bigl(   \Normbig{A_\phi^n}_1  \bigr) 
		=\log \Bigl( \limsup\limits_{n\to +\infty}  \Normbig{A_\phi^n}_1^{1/n} \Bigr) 
		= \log (\rho (A_\phi))$.
\end{proof}

Note that $(\cO_\omega (\cC_A),  \sigma)$ is the one-sided subshift of finite type defined by $A$. By Theorem~\ref{topological pressure coincide with the lift} and Lemma~\ref{P(C_A,phi)=logrho(A_phi)}, the topological entropy of $(\cO_\omega (\cC_A),  \sigma)$ is $\log( \rho (A))$ (see e.g., \cite[Theorem~7]{Pa64}).

\subsubsection{Transition matrices and measure-theoretic entropy}

Now we discuss the transition probability kernels on a finite space and compute its measure-theoretic entropy. We use transition matrices to represent transition probability kernels, where a matrix $(p_{ij})_{1\leq i,  j\leq d}$ is called a \emph{transition matrix} if $p_{ij}\geq 0$ for all $1\leq i,  j\leq d$ and $\sum_{j=1}^d  p_{ij}=1$ for all $1\leq i\leq d$

\begin{definition}\label{2ex transition probability kernel (2)}
	Let $d\in \N$, $X=Y=\oneto{d}$, $\SBB (X)=\SBB (Y)= 2^X$, the set of all subsets of $X$, and $P=(p_{ij})_{1\leq i,  j\leq d}$ be a transition matrix. The transition probability kernel $\hP$ induced by $P$ is defined as
	\begin{equation*}
		\hP (i,  B)\= \sum_{j\in B}  p_{ij}
	\end{equation*}
	for all $i\in \oneto{d}$ and $A\subseteq \oneto{d}$. In particular, for arbitrary $i,\, j\in X$, we have $\hP (i,  \{ j\})= p_{ij}$.
\end{definition}

Let $d\in \N$ and $P=(p_{ij})_{1\leq i,  j\leq d}$ be a transition matrix. Set $X\=\oneto{d}$. We use a column vector $v_f \=(f(1),  f(2),  \dots ,  f(d))^T$ to denote a function $f\: X\to \R$. Additionally, for a distribution $p$ on $X$, we write $p=(p_1 ,  p_2 ,  \dots ,  p_d)$, where $p_j =p (\{ j \})$ for each $j\in X$. For a function $f\: X\to \R$, we have $\hP f(i)=\int_X \! f(j)  \diff\hP_i(j)= \sum_{j=1}^d   f(j)p_{ij}$, and thus $v_{\hP f}=Pv_f$. Let $p=(p_1 ,  p_2 ,  \dots ,  p_d)$ be a distribution on $X$. For each $i\in X$, we have $p\hP (\{ i\})= \sum_{j=1}^d   \hP (j,  \{i \}) p(\{j \})= \sum_{j=1}^d   p_j p_{ji}$, so
\begin{equation}\label{2ex transition probability kernel act on measures (2)}
	p\hP =pP.
\end{equation}

This leads to the following lemma.

\begin{lemma}\label{2ex invariant measure (2)}
	A distribution $p$ on $X$ is $\hP$-invariant if and only if $pP=p$.
\end{lemma}

Note that a transition matrix $P=(p_{ij})_{1\leq i,  j\leq d}$ and an initial distribution $p=(p_1 ,\dots ,p_d)$ on $X =\oneto{d}$ can form a Markov chain. 
We have
\begin{equation}\label{muhP^[n](j_0,dots,j_n)=mu(j_0)prod_k=0^n-1hP(j_k,j_k+1)}
	p \hP^{\zeroto{n}} (\{ j_0 ,\, j_1 ,\, \dots ,\, j_n\})= p_{j_0} p_{j_0 j_1} p_{j_1 j_2} \dots p_{j_{n-1} j_n}.
\end{equation}
Suppose $p$ is $\hP$-invariant. This yields $pP=p$ by Lemma~\ref{2ex invariant measure (2)}, so (\ref{muhP^[n](j_0,dots,j_n)=mu(j_0)prod_k=0^n-1hP(j_k,j_k+1)}) and~(\ref{muQ^N(A*X^infty)=muQ^[n](A)}) reveal that the measure-preserving system $\bigl(X^\omega ,  \SBB (X^\omega) ,  \mu \hP^\omega ,  \sigma \bigr)$ is a one-sided $(p ,  P)$-Markov shift. %
About the Markov shift, we have $h_{p \hP^\omega} (\sigma)= - \sum_{i,  j=1}^d   p_i p_{ij} \log (p_{ij})$ (see e.g., \cite[Theorem~4.27]{Wa82}), 
where we take $0 \log 0 \=0$. Thus, by Theorem~\ref{measure-theoretic entropy coincide with its lift}, we get the following lemma:

\begin{lemma}\label{h_mu(hP)=sum_i,j=1^d-p_ip_ijlogp_ij}
		Let $d\in \N$, $P= (p_{ij})_{1\leq i, j\leq d}$ be a transition matrix, and $p=(p_1 ,\dots ,p_d)$ be a $\hP$-invariant distribution on $X=\oneto{d}$. Then
		$h_p \bigl( \hP \bigr)= - \sum_{i,  j=1}^d   p_i p_{ij} \log (p_{ij})$.
\end{lemma}

\subsubsection{Variational Principle}
By the discreteness of the finite space, all correspondences on $X= \oneto{d}$ are forward expansive, so by Theorem~\ref{t_VP_forward_exp}, the Variational Principle always holds and equilibrium states always exist in this case. Specifically, Lemmas~\ref{P(C_A,phi)=logrho(A_phi)} and~\ref{h_mu(hP)=sum_i,j=1^d-p_ip_ijlogp_ij} yield the following result:
	
	\begin{prop}
		Let $d\in \N$ and $A=(a_{ij})_{1\leq i,  j \leq d}$ be a $(0,1)$-matrix with at least one entry $1$ in each row. Set $\Gamma_A \= \{(i,j) \in \oneto{d} \times \oneto{d} : a_{ij} =1 \}$. Suppose $\phi \: \Gamma_A \to \R$ is a function. Write $A_\phi \= \bigl( a_{ij}\cdot e^{\phi (i ,  j)} \bigr)_{1\leq i,  j\leq d}$. Then
		\begin{equation*}
			\log (\rho(A_\phi)) =\sup_{P,\, p} \biggl\{ \sum_{(i,j)\in \Gamma_A} p_i p_{ij} \phi(i,j) -\sum_{i, j=1}^d p_i p_{ij} \log (p_{ij}) \biggr\},
		\end{equation*}
		where $\rho (A_\phi)$ is the spectral radius of $A_\phi$, $P =(p_{ij})_{1\leq i,j \leq d}$ ranges over all $d \times d$ transition matrices satisfying $p_{ij} =0$ for all $(i,j) \in \oneto{d} \times \oneto{d} \smallsetminus \Gamma_A$, and $p =(p_1 ,\dots ,p_d)$ ranges over all probability vectors satisfying $pP=p$. Moreover, there exists such a pair $(p, P)$ that attains the supremum.
	\end{prop}

\subsection{Single-valued maps}\label{sct_Example: single-valued map}

Here we focus on a degenerate case where the correspondence is induced by a single-valued map, and show that our theory is compatible with the classical ergodic theory for single-valued maps. In particular, we will explain why the conjectured (\ref{P(T,phi)=sup_Q,mu(h_mu(Q)+int_Xphidmu)}) coincides with the classical Variational Principle for single-valued maps.

\subsubsection{Correspondences and topological pressure}\label{subsct_ex1_Correspondences and topological pressure}

For a continuous map $f\: X\to X$ on a compact metric space $(X ,  d)$ and $\varphi \in \CCC( X, \R)$. Recall topological pressure $P(f,  \varphi)$ from \cite[Section~3.3]{PU10}:
\begin{equation}\label{P(T,phi) single-valued}
	\begin{aligned}
		P(f,  \varphi)\= {}& \lim_{\epsilon\to 0^+} \limsup_{n\to +\infty} \frac{1}{n} \log \biggl( \sup_{E_n (\epsilon)} \sum_{x\in E_n (\epsilon)}  \exp \biggl( \sum_{j=0}^{n-1}  \varphi (f^j (x)) \biggr) \biggr)  \\
		= {}&\lim_{\epsilon\to 0^+} \limsup_{n\to +\infty} \frac{1}{n} \log \biggl(  \inf_{F_n (\epsilon)} \sum_{x\in F_n (\epsilon)}  \exp \biggl( \sum_{j=0}^{n-1}  \varphi (f^j (x))  \biggr) \biggr)   ,
	\end{aligned}
\end{equation}
where $E_n (\epsilon)$ (resp.\ $F_n (\epsilon)$) ranges over all $(n,  \epsilon)$-separated (resp.\ $(n,  \epsilon)$-spanning) subsets of $X$.

Recall the associated correspondence $\cC_F$ from Section~\ref{sct_notation}. We will show that $P(f,  \varphi)$ and $P(\cC_f,  \hvarphi)$ are equal, where $\hvarphi \: \cO_2 (\cC_f) \to \R$ is a function induced by $\varphi$ (see (\ref{lifts of potential})), and thus the topological pressure of correspondences generalizes the topological pressure of single-valued continuous maps.

\begin{prop}\label{P(f,phi)=P(C_f,phi)}
	Let $f \: X \to X$ be a continuous map on a compact metric space $(X,  d)$ and $\varphi \in \CCC( X , \R)$. Then $P(f ,  \phi)= P(\cC_f ,  \hvarphi)$.
\end{prop}

\begin{proof}
	Fix an arbitrary $n\in \N$. Since for each $x\in X$, $\cC_f (x)= \{ f(x)\}$ is a singleton, we can see that $\vect{x}{1}{n+1} \in \cO_{n+1} (\cC_f)$ depends on $x_1$ in the way that $x_i =f^{i-1} (x_1)$ for every $i\in \{ 2,\, \dots ,\, n+1\}$. Thus, the map $\Phi_{n+1}$ that assigns each point $x\in X$ the orbit $(x,  f(x),  \dots ,  f^n (x))\in \cO_{n+1} (\cC_f)$ is a bijection from $X$ to $\cO_{n+1} (\cC_f)$. Recall that a subset $E\subseteq X$ is $(n+1,  \epsilon)$-separated in $(X,  d)$ if and only if $\Phi_{n+1} (E)= \{(x,  f(x),  \dots ,  f^n (x)) :  x\in E\}$ is $\epsilon$-separated in $(\cO_{n+1} (\cC_f) ,  d_{n+1})$, so
	\begin{equation*}
		\begin{aligned}
			\sup_{E_{n+1} (\epsilon)} \sum_{x\in E_{n+1} (\epsilon)}  \exp \biggl( \sum_{j=0}^n   \varphi (f^j (x)) \biggr) 
			& = \sup_{E_{n+1} (\epsilon)} \sum_{\vect{x}{1}{n+1} \in \Phi_{n+1} (E_{n+1} (\epsilon))}  \exp \bigl( S_n \hvarphi \bigl( \vect{x}{1}{n+1} \bigr) +\varphi (x_{n+1})\bigr)\\
			& = \sup_E \sum_{\udx =\vect{x}{1}{n+1} \in E}  \exp (S_n \hvarphi (\udx) +\varphi (x_{n+1})),
		\end{aligned}
	\end{equation*}
	where $E_{n+1} (\epsilon)$ ranges over all $(n+1 ,  \epsilon)$-separated subset of $X$ and $E$ ranges over all $\epsilon$-separated subset of $\cO_n (\cC_f)$. Since $s_n (\cC_f ,  \hvarphi ,  \epsilon)= \sup_E \sum_{\udx \in E}   \exp (S_n \hvarphi (\udx))$, we have
	\begin{equation*}
		e^{- \norm{\varphi}_\infty} s_n (\cC_f ,  \hvarphi ,  \epsilon) 
		\leq \sup_{E_{n+1} (\epsilon)} \sum_{x\in E_{n+1} (\epsilon)}  \exp \biggl(  \sum_{j=0}^n   \varphi \bigl(f^j (x)\bigr) \biggr)
		\leq e^{ \norm{\varphi}_\infty} s_n (\cC_f ,  \hvarphi ,  \epsilon).
	\end{equation*}
	
	Therefore, by (\ref{P(T,phi) single-valued}) and Definition~\ref{topological pressure}, we have 
	\begin{equation*}
		P(f,  \phi)=\lim\limits_{\epsilon\to 0^+} \limsup\limits_{n\to +\infty} \frac{1}{n} \log (s_n (\cC_f ,  \phi ,  \epsilon))= \lim\limits_{\epsilon\to 0^+} s (\cC_f ,  \phi ,\epsilon)= P (\cC_f ,  \phi). \qedhere
	\end{equation*}
\end{proof}

\subsubsection{Transition probability kernels and measure-theoretic entropy}

Let $(X,  \SAA (X))$ and $(Y,  \SAA (Y))$ be measurable spaces.

\begin{definition}\label{2ex transition probability kernel (1)}
	Let $F\: Y\to X$ be a measurable map. The transition probability kernel $\hF$ induced by $F$ is defined as
	\begin{equation*}
		\hF(y,  A)\= \mathbbold{1}_{F^{-1}(A)} (y)=
		\begin{cases}
			1 & \text{if } F(y)\in A; \\
			0 & \text{if } F(y)\notin A
		\end{cases}
	\end{equation*}
	for all $y\in Y$ and $A\in \SAA(X)$.%
\end{definition}

\begin{remark}
	In this case, $\hF_y=\delta_{F(y)}$ for each $y\in Y$, where $\delta_{F(y)}$ is the Dirac measure at the point $F(y)$.
\end{remark}

Let $F\: Y\to X$ be a measurable map and $f\: X\to \R$ be a measurable function. For each $y\in Y$, we have $\hF f(y)=\int_X \! f(x)  \diff\hF_y(x)=\int_X \! f(x)  \diff\delta_{F(y)} (x)=f(F(y))$, so $\hF f=f\circ F.$

Suppose that $\mu$ is a probability measure on $(Y ,  \SAA (Y))$. For each $A\in \SAA (X)$, we have $\bigl(\mu \hF\bigr) (A)=\int_Y \! \hF (y,  A)  \diff\mu (y)= \int_Y \! \mathbbold{1}_{F^{-1}(A)} (y)  \diff\mu (y)= \mu \bigl(F^{-1}(A)\bigr)$, so
\begin{equation}\label{2ex transition probability kernel act on measures (1)}
	\mu \hF =\mu \circ F^{-1}.
\end{equation}

This leads to the following lemma.

\begin{lemma}\label{2ex invariant measure (1)}
	Let $F\: X\to X$ be a measurable map on $(X ,  \SAA (X))$. A probability measure on $X$ is $\hF$-invariant if and only if it is $F$-invariant.
\end{lemma}
Let $F\: X\to X$ be a measurable map, $\mu \in \PPP(X)$, $n\in \N_0$, and $\vect{B}{0}{n}   \in (\SAA (X))^{n+1}$. Then
\begin{equation}\label{muT^[n](B_0*...*B_n)=mu(B_0cap...capT^-n(B_n))}
	\bigl(\mu \hF^{\zeroto{n}}\bigr) (B_0 \times B_1 \times \cdots \times B_n)
	= \mu \bigl(B_0 \cap F^{-1} (B_1) \cap \dots \cap F^{-n} (B_n) \bigr),
\end{equation}
which can be verified by induction on $n$ based on (\ref{muQ^{n}(A_0**A_n)}) in Lemma~\ref{l_Qun}.
Now suppose that $\mu$ is $F$-invariant. By Lemma~\ref{2ex invariant measure (1)}, it is also $\hF$-invatiant.

We recall some conventions from \cite[Chapter~2]{PU10}:

Let $\cA$ be a finite measurable partition of $(X , \SAA (X))$ and $n\in \N$. The finite measurable partition $F^{-n} (\cA)$ is given by $F^{-n} (\cA)\=  \{ F^{-n} (A) :  A\in \cA  \}$.
The entropy $h_\mu (F,  \cA)$ is given by
\begin{equation}\label{h_mu(T,A)=}
	h_\mu (F,  \cA)\= \lim_{n\to +\infty} \frac{1}{n} H_\mu \bigl(\cA \vee F^{-1} (\cA) \vee \cdots \vee F^{-(n-1)} (\cA)\bigr).
\end{equation}

By (\ref{muT^[n](B_0*...*B_n)=mu(B_0cap...capT^-n(B_n))}), $H_{\mu \hF^{\zeroto{n-1}}} (\cA^n)= H_\mu \bigl(\cA \vee F^{-1} (\cA) \vee \cdots \vee F^{-(n-1)} (\cA)\bigr)$. Hence, by Definition~\ref{h_mu (Q,A)} we get
\begin{equation}\label{h_mu(T,A)=h_mu(hT,A)}
	h_\mu \bigl(\hF ,  \cA \bigr)= h_\mu (F,  \cA).
\end{equation}

Recall $h_\mu (F)\= \sup_\cA h_\mu (F,  \cA)$ from \cite[Chapter~2]{PU10}, where $\cA$ ranges over all finite measurable partitions of $X$. By (\ref{h_mu(T,A)=h_mu(hT,A)}) and Definition~\ref{measure-theoretic entropy of transition probability kernels}, we conclude the following:

\begin{prop}\label{h_mu(F)=h_mu(hF)}
		Let $F\: X\to X$ be a measurable map on a measurable space $(X, \SAA (X))$ and $\mu \in \MMM(X,F)$. Then $h_\mu (F)= h_\mu \bigl(\hF \bigr)$.
\end{prop}

\subsubsection{Variational Principle}\label{subsubsct_VP}
The only transition probability kernel supported by $\cC_f$ is $\hf$, defined in Definition~\ref{2ex transition probability kernel (1)}, and what we shall consider is the Borel probability measure $\mu$ which is $\hf$-invariant, or equivalently, $f$-invariant, where the equivalence has been shown in Lemma~\ref{2ex invariant measure (1)}.

By applying the classical Variational Principle to $\varphi$ and the dynamical system $(X,  f)$, we have
\begin{equation}\label{VP for f and varphi}
	P(f,  \varphi)= \sup \biggl\{ h_\mu (f)+ \int_X \! \varphi   \diff \mu  :  \mu \text{ is } f \text{-invariant} \biggr\}.
\end{equation}

Recall $\cC_f (x_1) =\{f (x_1)\}$ and $\hf_{x_1} =\delta_{f(x_1)}$ for all $x_1 \in X$. By (\ref{lifts of potential}) we have
\begin{equation}\label{aaaaaaaaaaaoo}
	\int_X \! \int_{\cC_f (x_1)} \! \hvarphi (x_1 ,  x_2)   \diff \hf_{x_1} (x_2)   \diff \mu (x_2)
	= \int_X \! \int_{\{f(x_1)\}} \! \varphi (x_1)   \diff \delta_{f(x_1)} (x_2)   \diff \mu (x_2)
	= \int_X \! \varphi   \diff \mu.
\end{equation}

By (\ref{VP for f and varphi}), Propositions~\ref{P(f,phi)=P(C_f,phi)},~\ref{h_mu(F)=h_mu(hF)}, and~(\ref{aaaaaaaaaaaoo}), we get
\begin{equation*}
	P\bigl(\cC_f ,  \hvarphi \bigr)
	= \sup \biggl\{ h_\mu \bigl(\hf \bigr)+ \int_X \! \int_{\cC_f (x_1)} \! \hvarphi (x_1 ,  x_2)   \diff \hf_{x_1} (x_2)   \diff \mu (x_2)  :  \mu \text{ is } \hf \text{-invariant} \biggr\}.
\end{equation*}

Therefore, the corresponding Variational Principle for the correspondence $\cC_f$ holds.

\subsubsection{Several properties for $\cC_f$}\label{subsct_Several properties for correspondences}

Here we point out some relations between properties for the correspondence $\cC_f$ and for the map $f$, all of which are not difficult to check from their definitions.

\begin{rem}   \label{r:C_f_properties}
	Let $f\: X\to X$ be a single-valued continuous map on a compact metric space $(X,  d)$.
\begin{enumerate}
	\smallskip
	\item[(i)] $\cC_f$ is forward expansive in the sense of Definition~\ref{forward expansiveness} with an expansive constant $\epsilon >0$ if and only if $f$ is forward expansive with an expansive constant $\epsilon$.%
	\smallskip
	\item[(ii)] $\cC_f$ has the specification property in the sense of Definition~\ref{specification correspondence} if and only if $f$ has the specification property in the sense of Definition~\ref{specification single-valued map}.
	\smallskip
	\item[(iii)] $\cC_f$ is distance-expanding in the sense of Definition~\ref{distance-expanding correspondence} if and only if $f$ is distance-expanding.
	\smallskip
	\item[(iv)] $\cC_f$ is open in the sense of Definition~\ref{open correspondence} if and only if $f$ is open.
	\smallskip
	\item[(v)] If $\cC_f$ is strongly transitive in the sense of Definition~\ref{strongly transitive correspondence}, then $f$ is topologically transitive.
	\smallskip
	\item[(vi)] $\cC_f$ is topologically exact in the sense of Definition~\ref{topologically exact} if and only if $f$ is topologically exact.
	\smallskip
	\item[(vii)] $\cC_f$ is continuous in the sense of Definition~\ref{continuity}.
\end{enumerate}
\end{rem}

\section{Topological pressure of correspondences}\label{sct_Topological_pressure_of_correspondences}

In this section, we introduce and discuss the topological pressure of correspondences. First, we recall the definition of correspondences in Subsection~\ref{subsct_Definition of correspondences}. Then in Subsection~\ref{subsct_Definition of topological pressure for correspondences}, we introduce the topological pressure of a correspondence with respect to a continuous potential function. Finally, in Subsection~\ref{subsct_A characterization of the topological pressure}, we define a shift map for a correspondence and relate the topological pressure of this shift map to that of the correspondence (Theorem~\ref{topological pressure coincide with the lift}).

\subsection{Definition of correspondences}\label{subsct_Definition of correspondences}

Here we provide our definition of correspondences on compact metric spaces. Recall from Section~\ref{sct_notation} that for a compact metric space $X$, the set $\cF (X)$ consists of all non-empty closed subsets of $X$.
The following lemma is established in \cite[Theorems~1--3]{IM06}.

\begin{lemma}\label{upper-semicontinuity equivalent to closedness of graph}
	Let $(X,  d)$ be a compact metric space. For a map $T\: X\to \cF (X)$, the following statements are equivalent:
	\begin{enumerate}
		\smallskip
		\item[(i)] {\rm(Upper-semicontinuity)} For every $x\in X$ and an arbitrary open neighborhood $\cU$ of $T(x)$, there exists an open neighborhood $\cV$ of $x$ such that $T(y)\subseteq \cU$ for each $y\in \cV$.
		\smallskip
		\item[(ii)] $\cO_2 (T)=\bigl\{(x_1 ,  x_2)\in X^2  :  x_2 \in T(x_1) \bigr\}$ is closed in $X^2$.
		\smallskip
		\item[(iii)] $\cO_n (T)$ is closed in $X^n$ for each $n\in \wh{\N}$.
	\end{enumerate}
\end{lemma}

\begin{definition}[Correspondence]\label{correspondence}
	Let $(X,  d)$ be a compact metric space. A map $T\: X\to \cF (X)$ is a \defn{correspondence on $X$}\index{correspondence} if one of the equivalent statements~(i), (ii), and~(iii) in Lemma~\ref{upper-semicontinuity equivalent to closedness of graph}.
\end{definition}

Let us recall the notion of continuity for correspondences on compact metric spaces from \cite[Section~9.4.1,~footnote~6]{AF90} \footnote{{\ifFirstInitial J.~\fi}Aubin and {\ifFirstInitial H.~\fi}Frankowska discussed upper semi-continuity (see Lemma~\ref{upper-semicontinuity equivalent to closedness of graph}~(i)), lower-semicontinuity, and continuity for what they called ``set-valued maps'' in \cite[Chapter~1]{AF90}.}.

\begin{definition}[Continuity]\label{continuity}
	Let $(X, d)$ be a compact metric space and $T\: X \to \cF (X)$ be a correspondence on $X$. If $T$ is continuous with respect to the metric $d$ on $X$ and the Hausdorff distance on $\cF (X)$, then we say that $T$ is a \defn{continuous}\index{continuous} correspondence.
\end{definition}

Recall $T^{-1} (x)= \{ y\in X  :  x\in T(y) \}$ for all $x\in X$. The lemma below follows easily from the definition.

\begin{lemma}  \label{l:inverse_correspondence}
If $T$ is a correspondence on a compact metric space $X$ with $T(X) =X$, then so is $T^{-1}$. 
\end{lemma}

\subsection{Definition of topological pressure for correspondences}\label{subsct_Definition of topological pressure for correspondences}

We introduce a new version of topological pressure of a correspondence through the $(n,  \epsilon)$-separated sets and $(n,  \epsilon)$-spanning sets. This naturally generalizes the topological pressure of a single-valued continuous map.

For $\epsilon >0$ and a metric space $(Y,  \rho)$, $E\subseteq Y$ is \defn{$\epsilon$-separated}\index{$\epsilon$-separated} if for each pair of distinct points $x,\, y \in E$, we have $\rho (x,  y)\geq \epsilon$; $F\subseteq Y$ is \defn{$\epsilon$-spanning}\index{$\epsilon$-spanning} if for each $y\in Y$ there exists $x\in F$ such that $\rho (x,  y)<\epsilon$. For each $\delta >0$ and each $g\in \CCC( Y, \R)$, set $\Delta (g,  \delta)\= \sup \{ \abs{g(x)-g(y)} :    x,\, y\in Y \text{ and } \rho (x,  y)\leq \delta\}$.

Let $T$ be a correspondence on a compact metric space $(X,  d)$ and $\phi \in \CCC(\cO_2 (T), \R)$. Write
\begin{align*}
	s_n (T,  \phi ,  \epsilon) &\= \sup\Bigl\{ \sum_{\udx \in E}   \exp (S_n \phi (\udx)) :  E \text{ is an } \epsilon \text{-separated subset of } \cO_{n+1} (T)\Bigr\},\\
	r_n (T,  \phi ,  \epsilon)  &\= \inf\Bigl\{ \sum_{\udx \in F}   \exp (S_n \phi (\udx)) :  F \text{ is an } \epsilon \text{-spanning subset of } \cO_{n+1} (T)\Bigr\},\\
	s(T,  \phi ,  \epsilon)      &\= \limsup_{n\to +\infty} \frac{1}{n} \log (s_n (T,  \phi ,  \epsilon)), \text{ and } \\
	r(T,  \phi ,  \epsilon)      &\= \limsup_{n\to +\infty} \frac{1}{n} \log (r_n (T,  \phi ,  \epsilon)),
\end{align*}
for each $n\in \N$ and each $\epsilon >0$. We now establish some estimates for these quantities.

By choosing an orbit $\udx_0 \in \cO_{n+1} (T)$ and focusing on the $\epsilon$-separated subset $\{ \udx_0 \}$ of $\cO_{n+1} (T)$, we have
$
	s_n (T,\, \phi ,\, \epsilon)
	\geq \exp (S_n \phi (\udx_0))
	\geq \exp (-n  \norm{\phi }_\infty)
$
and 
\begin{equation*}
	s(T,  \phi ,  \epsilon)\geq \limsup_{n\to +\infty} \frac{1}{n} \log (\exp (-n \norm{\phi }_\infty))= - \norm{\phi }_\infty.
\end{equation*}

For an arbitrary $\epsilon$-spanning set $F \subseteq \cO_{n+1} (T)$, we can choose an orbit $\udx_0 \in F$, and thus we have $\sum_{\udx \in F}  \exp (S_n \phi (\udx)) \geq \exp (S_n \phi (\udx_0)) \geq \exp (-n  \norm{\phi }_\infty)$ and %
\begin{equation}\label{r(T,phi,epsilon)geqslant-||phi||_infty}
\begin{aligned}
	r(T,  \phi ,  \epsilon)
	= \limsup_{n\to +\infty} \frac{1}{n} \log (r_n (T ,  \phi ,  \epsilon))
	\geq \limsup_{n\to +\infty} \frac{1}{n} \log (\exp(-n  \norm{\phi}_\infty) )
	= - \norm{\phi }_\infty.
\end{aligned}
\end{equation}

On the other hand, both $r(T,  \phi ,  \epsilon)$ and $s(T,  \phi ,  \epsilon)$ may be $+\infty$.

Since for arbitrary $\epsilon_2 > \epsilon_1 >0$, an $\epsilon_2$-separated set is also $\epsilon_1$-separated, we have $s_n (T,  \phi ,  \epsilon_2)\leq s_n (T,  \phi ,  \epsilon_1)$, i.e., $s_n (T,  \phi ,  \epsilon)$ is decreasing in $\epsilon$, and thus $s(T,  \phi ,  \epsilon)$ is decreasing in $\epsilon$. Similarly, an $\epsilon_1$-spanning set is also $\epsilon_2$-spanning, so we have $r_n (T,  \phi ,  \epsilon_2)\leq r_n (T,  \phi ,  \epsilon_1)$, i.e., $r_n (T,  \phi ,  \epsilon)$ is decreasing in $\epsilon$, and thus $r(T,  \phi ,  \epsilon)$ is decreasing in $\epsilon$. As a result, the following limits exist:
\begin{equation*}
	P_s (T,  \phi)\= \lim_{\epsilon \to 0^+} s(T,  \phi ,  \epsilon) \quad \text{and} \quad
	P_r (T,  \phi)\= \lim_{\epsilon \to 0^+} r(T,  \phi ,  \epsilon).
\end{equation*}

\begin{prop}\label{separated entropy coincide with spanning entropy}
	Let $T$ be a correspondence on a compact metric space $(X,  d)$ and $\phi \in \CCC(\cO_2 (T), \R)$. Then
	$P_s (T,  \phi)= P_r (T,  \phi)$.
\end{prop}

\begin{proof}
	For each $n\in \N$ and each $\epsilon >0$, choose a maximal $\epsilon$-separated subset $E\subseteq \cO_{n+1} (T)$. For each $\udy \in \cO_{n+1} (T)$, since $E\cup \{ \udy \}$ is not $\epsilon$-separated, there exists $\udx \in E$ such that $d_{n+1} (\udx ,  \udy)< \epsilon$. Thus, $E$ is $\epsilon$-spanning in $\cO_{n+1} (T)$. Thereby, we have $s_n (T,  \phi ,  \epsilon)\geq \sum_{\udx \in E}  \exp (S_n \phi (\udx)) \geq r_n (T,  \phi ,  \epsilon)$. This implies $s(T,  \phi ,  \epsilon)\geq r(T,  \phi ,  \epsilon)$, and hence we get $P_s (T,  \phi)\geq P_r (T,  \phi)$.
	
	For each $n\in \N$ and each $\epsilon >0$, choose an arbitrary $\epsilon$-separated set $E\subseteq \cO_{n+1} (T)$ and an arbitrary $\frac{\epsilon}{2}$-spanning set $F\subseteq \cO_{n+1} (T)$. For each orbit $\udx \in E$, since $F$ is $\frac{\epsilon}{2}$-spanning, there exists $\gamma (\udx)\in F$ with $d_{n+1} (\udx,  \gamma (\udx))< \epsilon / 2$. For distinct $\udx ,\, \udy \in E$, since $d_{n+1} (\udx ,  \udy)\geq \epsilon$, $d_{n+1} (\udx ,  \gamma (\udx))<  \epsilon / 2 $, and $d_{n+1} (\udy ,  \gamma (\udy))<  \epsilon / 2 $, we have $\gamma (\udx)\neq \gamma (\udy)$. Thereby, $\gamma \: E\to F$ is injective, and thus
	\begin{equation*}
			\sum_{\udy \in F}  e^{S_n \phi (\udy)}
			\geq \sum_{\udx \in E}  e^{S_n \phi (\gamma (\udx))}   %
			\geq \sum_{\udx \in E}  e^{S_n \phi (\udx)- \Delta (S_n \phi ,   \epsilon / 2 )}
			= e^{-\Delta (S_n \phi ,   \epsilon / 2  )} \sum_{\udx \in E}  e^{ S_n \phi (\udx)},
	\end{equation*}
	where 
	$\Delta (S_n \phi ,   \epsilon / 2 ) \=\sup \bigl\{ \Absbig{S_n \phi \bigl( \udy_1 \bigr)- S_n \phi \bigl( \udy_2 \bigr)} :    \udy_1 ,\, \udy_2 \in \cO_{n+1} (T), \, d_{n+1} \bigl(\udy_1 ,  \udy_2 \bigr) \leq  \epsilon / 2 \bigr\}$.

	Recall $d_{n+1} \bigl( \vect{x}{1}{n+1} , \vect{y}{1}{n+1} \bigr)= \max \{ d(x_i ,  y_i)  :  1\leq i\leq n+1 \}$. 
	If $\vect{x}{1}{n+1} , \,  \vect{y}{1}{n+1}  \in \cO_n (T)$ and $d_{n+1} \bigl( \vect{x}{1}{n+1} , \vect{y}{1}{n+1} \bigr)  \leq  \epsilon / 2$,
	then $d(x_i ,  x_i ')\leq \epsilon / 2$ for all $i\in \oneto{n+1}$, and thereby, we have
	\begin{equation*}
			\Absbig{S_n \phi  \bigl( \vect{x}{1}{n+1} \bigr)- S_n \phi \bigl( \vect{y}{1}{n+1} \bigr) }
			\leq \sum_{i=1}^n    \abs{\phi (x_i ,  x_{i+1})- \phi (x_i ' ,  x_{i+1} ')}
			\leq n\Delta (\phi ,   \epsilon / 2  ).
	\end{equation*}
	This implies that $\Delta (S_n \phi ,   \epsilon / 2  )\leq n\Delta (\phi ,   \epsilon / 2)$. As a result,
	$\sum_{\udy \in F}  e^{S_n \phi (\udy)}
	\geq e^{-n\Delta (\phi ,   \epsilon / 2  )} \sum_{\udx \in E}  e^{S_n \phi (\udx)}$.

	Since $E$ and $F$ are chosen arbitrarily, we have 
	$r_n (T,  \phi ,   \epsilon / 2  )\geq e^{-n\Delta (\phi ,  \epsilon / 2 )} s_n (T,  \phi ,  \epsilon)$.
	Thus,
	\begin{equation}\label{r>s-Delta}
		r (T,  \phi ,   \epsilon / 2  )\geq s(T,  \phi ,  \epsilon)- \Delta (\phi ,  \epsilon / 2  ).
	\end{equation}
	
	Since $X$ is compact and $\phi$ is continuous, $\phi$ is uniformly continuous, i.e., for an arbitrary $\delta >0$, there exists $\lambda >0$ such that $\Delta (\phi ,  \lambda)<\delta$. Thus, we have $\lim_{\epsilon\to 0^+} \Delta (\phi , \epsilon / 2)=0$. Consequently, by taking $\epsilon\to 0^+$ in (\ref{r>s-Delta}), we get $P_r (T,  \phi)\geq P_s (T,  \phi)$.
	
	Therefore, we conclude that $P_r (T,  \phi)=P_s (T,  \phi)$.
\end{proof}

\begin{definition}[Topological pressure]\label{topological pressure}
	Let $T$ be a correspondence on a compact metric space $(X,  d)$ and $\phi \in \CCC (\cO_2 (T), \R)$. The \defn{topological pressure}\index{topological pressure} $P(T,  \phi)$ is defined as
	\begin{equation}  \label{P(T,phi)=lim_epsilonto0^+limsup_nto+infty1/nlogsup_E_n(epsilon)sum_xinE_n(epsilon)expS_nphi(x)}
		P(T,  \phi)\= P_s (T,  \phi)= P_r (T,  \phi).
	\end{equation}
	In particular, if $\phi \equiv 0$, we call $P (T ,  0)$ the \defn{topological entropy}\index{topological entropy} of $T$ and denote it by $h(T)$.\footnote{Our notion of topological entropy coincides with that in \cite[Definition~2.5]{KT17}.}
\end{definition}

\begin{rem}\label{P(T,phi)>-infty}
	Recall from (\ref{r(T,phi,epsilon)geqslant-||phi||_infty}) that $r(T,  \phi ,  \epsilon)\geq - \norm{\phi }_\infty$. This implies
	\begin{equation}\label{P(T,phi)geqslant-||phi||_infty}
		P(T,  \phi) =P_r (T ,  \phi) =\lim_{\epsilon\to 0^+} r(T ,  \phi ,  \epsilon) \geq - \norm{\phi}_\infty >-\infty.
	\end{equation}
\end{rem}

For each $\phi \in \CCC (\cO_2 (T) , \R)$, denote by $\ophi \in \CCC \bigl( \cO_2  \bigl(T^{-1}\bigr) , \R \bigr)$ the conjugate function given by $\ophi (x ,  y)\= \phi (y ,  x)$ for all $(x ,  y) \in \cO_2 \bigl( T^{-1} \bigr)$.

\begin{prop}\label{q3f33333333333333333f}
	Let $T$ be a correspondence on a compact metric space $X$ satisfying $T(X) =X$ and $\phi \in \CCC ( \cO_2 (T) , \R)$. Then
	$P(T ,  \phi) =P\bigl(T^{-1} ,  \ophi \bigr)$.
\end{prop}

\begin{proof}
	For each $n\in \N$ and $\udx =\vect{x}{1}{n}  \in X^n$, $\udx  \in \cO_n (T)$ if and only if $\gamma_n (\udx) \in \cO_n \bigl( T^{-1} \bigr)$. Consequently, the isometry $\gamma_n$ sends $\cO_n (T)$ onto $\cO_n \bigl( T^{-1} \bigr)$. From (\ref{e:S_nphi=}) we can see that $S_n \phi (\udx) =S_n \ophi (\gamma_{n+1} (\udx))$ holds for all $\udx \in \cO_{n+1} (T)$. Since $\gamma_{n+1}$ is an isometry, for each $\epsilon >0$, $E_n (\epsilon) \subseteq \cO_{n+1} (T)$ is $\epsilon$-separated if and only if $\gamma_{n+1} (E_n (\epsilon)) \subseteq \cO_{n+1} \bigl( T^{-1} \bigr)$ is $\epsilon$-separated, so by (\ref{P(T,phi)=lim_epsilonto0^+limsup_nto+infty1/nlogsup_E_n(epsilon)sum_xinE_n(epsilon)expS_nphi(x)}) we conclude $P(T ,  \phi) =P\bigl(T^{-1} ,  \ophi \bigr)$.
\end{proof}

\subsection{A characterization of the topological pressure}\label{subsct_A characterization of the topological pressure}

We will prove in this subsection that our topological pressure of a correspondence $T$ with respect to a continuous potential function $\phi$ is equal to $P\bigl( \sigma,  \tphi \bigr)$, the topological pressure of the shift map $\sigma$ on the orbit space $\cO_\omega (T)$ with respect to the potential function $\tphi$ (given in (\ref{lifts of potential})) induced by $\phi$ (see Theorem~\ref{topological pressure coincide with the lift} for the precise statement).

Let $T$ be a correspondence on a compact metric space $(X,  d)$. We consider a dynamical system $(\cO_\omega (T),  \sigma)$, where $\cO_\omega (T)$ is equipped with the metric $d_\omega$ and $\sigma\: \cO_\omega (T)\to \cO_\omega (T)$ is the shift map given by
$\sigma (\vect{x}{1}{\infty})\= \vect{x}{2}{\infty}$
for all $\vect{x}{1}{\infty} \in X^\omega$.

\begin{theorem}\label{topological pressure coincide with the lift}
	Let $T$ be a correspondence on a compact metric space $(X ,  d)$, $\phi\in \CCC ( \cO_2 (T), \R)$, and $\sigma$ be the shift map on $\cO_\omega (T)$. Then we have
	\begin{equation*}
		P(T,  \phi)= P\bigl(\sigma ,  \tphi\bigr),
	\end{equation*}
	where $P\bigl(\sigma ,  \tphi\bigr)$ refers to the topological pressure of $(\cO_\omega (T),  \sigma)$ with the potential $\tphi$ given in (\ref{lifts of potential}).
\end{theorem}

\begin{proof}
	We divide this proof into two steps. Let $\epsilon>0$ be arbitrary and denote $\tepsilon \= \epsilon/ (1+\epsilon)$.
	
	\smallskip
	
	Step 1. We show $P(T ,  \phi) \leq P\bigl(\sigma ,  \tphi \bigr)$.
	
	Let $n\in \N$. For every $\vect{x}{1}{n+1} \in \cO_{n+1} (T)$, we choose $x_{n+2} ,\, x_{n+3} ,\, \dots \in X$ such that $\vect{x}{1}{\infty}\in \cO_\omega (T)$. Denote by $\tau\: E_n (\epsilon)\to \cO_\omega (T)$ the map that extends each $\vect{x}{1}{n+1}  \in E_n (\epsilon)$ to the orbit $\vect{x}{1}{\infty} \in \cO_\omega (T)$. The map $\tau$ is injective.
	
	Fix arbitrary $n\in \N$ and $\epsilon$-separated subset $E_n (\epsilon)$ of $(\cO_{n+1} (T),  d_{n+1} )$.
	
	For an arbitrary pair of distinct orbits $\vect{x}{1}{n+1} ,\,  \vect{y}{1}{n+1}  \in E_n (\epsilon)$, we have
	\begin{equation*}
		\epsilon \leq d_{n+1} \bigl( \vect{x}{1}{n+1} ,  \vect{y}{1}{n+1}  \bigr)= \max \{ d(x_i ,  y_i)  :  1\leq i\leq n+1 \}.
	\end{equation*}
	
	Choose $k\in \oneto{n+1}$ such that $d(x_k ,  y_k)\geq \epsilon$, then
	\begin{equation*}
			 d_\omega \bigl(\sigma^{k-1} \bigl( \tau \bigl( \vect{x}{1}{n+1} \bigr) \bigr),  \sigma^{k-1} \bigl( \tau \bigl( \vect{y}{1}{n+1} \bigr) \bigr) \bigr)
			= d_\omega (\vect{x}{k}{\infty},  \vect{y}{k}{\infty})
			\geq \frac{1}{2} \cdot \frac{d(x_k ,  y_k)}{1+d(x_k ,  y_k)}\geq \frac{\tepsilon}{2}.
	\end{equation*}
	This implies that $\tau (E_n (\epsilon))$ is $(n+1,   \tepsilon / 2)$-separated for the dynamical system $(\cO_\omega (T),  \sigma)$, i.e., for each pair of distinct orbits $\udx ,\, \udy \in \tau (E_n (\epsilon))$, there exists $k\in [n]$ such that $d_\omega \bigl( \sigma^k (\udx) , \sigma^k (\udy) \bigr) \geq \tepsilon /2$.
	
	Since
	\begin{equation*}
				\sum_{\udx \in E_n (\epsilon)}  e^{S_n \phi (\udx)}  
		= \sum_{\vect{x}{1}{\infty}\in \tau (E_n (\epsilon))}  e^  {\sum_{j=1}^n   \phi (x_j ,  x_{j+1})}  
		= \sum_{\udx \in \tau (E_n (\epsilon))}  e^{ \sum_{j=1}^n   \tphi (\sigma^{j-1} (\udx) )  },
	\end{equation*}
	and the $\epsilon$-separated set $E_n (\epsilon)$ is chosen arbitrarily, we have
	\begin{equation*}
		\begin{aligned}
			& \sup \Bigl\{ \sum_{\udx \in E_n (\epsilon)}  e^{ S_n \phi (\udx) } :  E_n (\epsilon) \text{ is } \epsilon \text{-separated in } \cO_{n+1} (T)\Bigr\}\\
			&\qquad \leq \sup \biggl\{ \sum_{\udx \in \tE_n ( \frac{\tepsilon}{2} )}  e^{  \sum_{j=0}^{n-1}  \tphi (\sigma^j (\udx)) }    :  
			\tE_n  ( \tepsilon / 2 )  \text{ is } (n+1,   \tepsilon / 2) \text{-separated in } (\cO_\omega (T),  \sigma) \biggr\}.
		\end{aligned}
	\end{equation*}
	This implies that
	\begin{equation*}
		\begin{aligned}
			P(T,  \phi)&= \lim_{\epsilon\to 0^+} \limsup_{n\to +\infty} \frac{1}{n} \log \biggl(  \sup_{E_n (\epsilon)} \sum_{\udx \in E_n (\epsilon)}  \exp (S_n \phi (\udx)) \biggr) \\
			&\leq \lim_{\epsilon\to 0^+} \limsup_{n\to +\infty} \frac{1}{n} \log \biggl( \sup_{\tE_n ( \tepsilon / 2 )} \sum_{\udx \in \tE_n (\tepsilon / 2)}  \exp \biggl( \sum_{j=0}^{n-1}  \tphi (\sigma^j (\udx)) \biggr) \biggr)  %
			=P\bigl(\sigma ,  \tphi\bigr),
		\end{aligned}
	\end{equation*}
	where $E_n (\epsilon)$ ranges over all $\epsilon$-separated subsets of $\cO_{n+1} (T)$, $\tE_n (\tepsilon / 2)$ ranges over all $(n+1,  \tepsilon / 2 )$-separated sets of the dynamical system $(\cO_\omega (T),  \sigma)$, and the last equality holds because $\lim_{\epsilon\to 0^+} \tepsilon / 2 =0$. Hence, we conclude $P(T,  \phi)\leq P\bigl(\sigma ,  \tphi \bigr)$.
	
	\smallskip
	
	Step 2. We show $P(T,  \phi)\geq P\bigl(\sigma ,  \tphi \bigr)$.
	
	Fix arbitrary $n\in \N$ and $\epsilon$-spanning subset $F_n (\epsilon)$ of $(\cO_{n+1} (T) ,  d_{n+1} )$.

	For every $\vect{y}{1}{\infty}\in \cO_\omega (T)$, since $\vect{y}{1}{n+1} $ is in $\cO_{n+1} (T)$, we can choose an orbit $\vect{x}{1}{n+1} \in F_n (\epsilon)$ such that $d_{n+1} \bigl( \vect{y}{1}{n+1} ,  \vect{x}{1}{n+1} \bigr) < \epsilon$, i.e., $d(x_k ,  y_k )<\epsilon$ for all $k\in \oneto{n+1}$. For each $k\in \zeroto{n-\lfloor \sqrt{n} \rfloor-1}$, we have
	\begin{equation*}
		\begin{aligned}
			 d_\omega \bigl( \sigma^k ( \vect{y}{1}{\infty} ),  \sigma^k \bigl(\tau \bigl(\vect{x}{1}{n+1} \bigr) \bigr) \bigr)
			= d_\omega (\vect{y}{k+1}{\infty},  \vect{x}{k+1}{\infty} )
			&= \sum_{j=1}^{+\infty}   \frac{1}{2^j} \frac{d(x_{k+j} ' ,  x_{k+j})}{1+d(x_{k+j} ' ,  x_{k+j})}  
			< \sum_{j=1}^{n-k}  \frac{1}{2^j} \tepsilon  +\sum_{j=n-k+1}^{+\infty}   \frac{1}{2^j}  \\
			&<  \tepsilon  +2^{-(n-k)} 
			 \leq \tepsilon + 2^{- \lfloor \sqrt{n} \rfloor - 1}  
			 \leq \tepsilon  +2^{-\sqrt{n}}.
		\end{aligned}
	\end{equation*}
	Hence, $\tau (F_n (\epsilon))$ is an $\bigl(n-\lfloor \sqrt{n} \rfloor,  \tepsilon  +2^{-\sqrt{n}}\bigr)$-spanning set for $(\cO_\omega (T) ,  \sigma )$, i.e., for each orbit $\udx \in \cO_\omega (T)$, there exists $\udy \in \tau (F_n (\epsilon))$ such that $d_\omega \bigl( \sigma^k (\udx) ,\sigma^k (\udy) \bigr) <\tepsilon +2^{-\sqrt{n}}$ holds for all $k\in \zeroto{n- \lfloor \sqrt{n} \rfloor -1}$.
	
	Recall $ \norm{\phi}_\infty =\sup \{  \abs{\phi (x_1 ,  x_2)}  :  (x_1 ,  x_2)\in \cO_2 (T) \}$%
	. We have
	
	\begin{align}\label{F_n(epsilon)>tF_n(epsilon)}
			\sum_{\udx \in F_n (\epsilon)}  e^{ S_n \phi (\udx)}
			&=\sum_{\vect{x}{1}{n+1} \in F_n (\epsilon)}  e^{  \sum_{j=1}^n  \phi(x_j ,  x_{j+1}) }  
			 = \sum_{\udx\in \tau (F_n (\epsilon))}  e^{ \sum_{j=0}^{n-1}   \tphi (\sigma^j (\udx)) } \notag \\
			&\geq \sum_{\udx \in \tau (F_n (\epsilon))}  e^{  \sum_{j=0}^{n-\lfloor \sqrt{n} \rfloor-1}   \tphi (\sigma^j (\udx))- \lfloor \sqrt{n} \rfloor \norm{\phi}_\infty }
			\geq e^{-\sqrt{n} \norm{\phi}_\infty} \sum_{\udx \in \tau (F_n (\epsilon))}  e^{  \sum_{j=0}^{n-\lfloor \sqrt{n} \rfloor-1}   \tphi (\sigma^j (\udx)) }   .
	\end{align}
	
	For each $\delta >0$ and each $m\in \N$, write
	\begin{equation*}
		\alpha (m ,  \delta)\= \inf \bigg\{\sum_{\udx \in \tF_m (\delta)} e^{  \sum_{j=0}^{m-1}   \tphi (\sigma^{j}(\udx))}   :  \tF_m (\delta) \text{ is } (m,  \delta) \text{-spanning in } (\cO_\omega (T),  \sigma)\bigg\}.
	\end{equation*}
	Then we have $P\bigl( \sigma ,  \tphi \bigr)=\lim_{\delta\to 0^+} \limsup_{m\to +\infty} \frac{1}{m}\log (\alpha (m,  \delta))$.
	
	Since for $\delta_2> \delta_1> 0$, an $(m,  \delta_1)$-spanning set is also $(m,  \delta_2)$-spanning, we can see that $\alpha (m,  \delta)$ is decreasing in $\delta$ for each $m\in\N$.
	
	Since $\tau (F_n (\epsilon))$ is an $\bigl(n-\lfloor \sqrt{n} \rfloor,   \tepsilon +2^{-\sqrt{n}}\bigr)$-spanning set in $(\cO_\omega (T) ,  \sigma )$, (\ref{F_n(epsilon)>tF_n(epsilon)}) implies
	\begin{equation*}
		\sum_{\udx \in F_n (\epsilon)}  e^{S_n \phi (\udx) }
		\geq e^{-\sqrt{n} \norm{\phi}_\infty} \alpha \bigl(n-\lfloor \sqrt{n} \rfloor, \tepsilon +2^{-\sqrt{n}}\bigr).
	\end{equation*}
	
	Let $F_n (\epsilon)$ range over all $\epsilon$-spanning subsets of $(\cO_{n+1} (T),  d_{n+1})$ and take an infimum in the inequality above, we get
	$r_n (T,  \phi ,  \epsilon)
		= \inf_{F_n (\epsilon)} \sum_{\udx \in F_n (\epsilon)}  e^{ S_n \phi (\udx) }
		\geq e^{-\sqrt{n} \norm{\phi}_\infty} \alpha \bigl(n-\lfloor \sqrt{n} \rfloor,   \tepsilon  +2^{-\sqrt{n}}\bigr)$.
	This implies
	\begin{equation*}
		\begin{aligned}
			P(T,  \phi )= \lim_{\epsilon\to 0^+} \limsup_{n\to +\infty} \frac{ \log (r_n (T,  \phi ,  \epsilon)) }{n} 
			&\geq \lim_{\epsilon\to 0^+} \limsup_{n\to +\infty} \bigl( n^{-1} \log \bigl(  \alpha \bigl(n-\lfloor \sqrt{n} \rfloor,    \tepsilon  +2^{-\sqrt{n}}\bigr)  \bigr) -  \norm{\phi}_\infty n^{-1/2}\bigr)\\
			&\geq \lim_{\epsilon\to 0^+} \limsup_{n\to +\infty} (n-\lfloor \sqrt{n} \rfloor)^{-1} \log \bigl(  \alpha \bigl(n-\lfloor \sqrt{n} \rfloor,   \tepsilon +\epsilon \bigr)  \bigr) %
			= P\bigl(\sigma ,  \tphi \bigr),
		\end{aligned}
	\end{equation*}
	where the last equality holds because $n-\lfloor \sqrt{n} \rfloor$ ranges over all positive integers as $n$ ranges over all positive integers and $\lim_{\epsilon\to 0^+}  (\tepsilon  +\epsilon) =0$.
\end{proof}

\section{Measure-theoretic entropy of transition probability kernels}\label{sct_Measure-theoretic_entropy_of_transition_probability_kernels}

This section is devoted to introducing and discussing the measure-theoretic entropy of transition probability kernels. We discuss some basic notions and properties about transition probability kernels and introduce the measure-theoretic entropy of transition probability kernels. Finally, we define another shift map for a correspondence and relate the measure-theoretic entropy of this shift map to that of the correspondence, see Theorem~\ref{measure-theoretic entropy coincide with its lift}.

\subsection{Basic properties of transition probability kernels}\label{subsct_Basic properties of transition probability kernels}

Here we recall the notion of transition probability kernels (see e.g., \cite[Section~3.4.1]{MT12}), which are also called Markovian transition kernels (see e.g., \cite[Section~6.1]{Le16}). %
Moreover, we recall how a transition probability kernel pushes a function forward and pulls a measure back.%

\begin{definition}[Transition probability kernels]  \label{transition probability kernel}
    Let $(X,  \SAA (X))$ and $(Y,  \SAA (Y))$ be measurable spaces, where $X$ and $Y$ are sets and $\SAA (X)$ and $\SAA (Y)$ are $\sigma$-algebras on $X$ and  $Y$, respectively. A \defn{transition probability kernel from $Y$ to $X$}\index{transition probability kernel} is a map $\cQ \: Y\times \SAA(X) \to [0,1]$ satisfying the following two properties:
    \begin{enumerate}
		\smallskip
        \item[(i)] For every $y\in Y$, the map $\SAA(X) \ni A \mapsto \cQ(y,  A)$ is a probability measure on the measurable space $(X,  \SAA(X))$.
		\smallskip
        \item[(ii)] For every $A\in \SAA(X)$, the map $Y\ni y\mapsto \cQ(y,  A)$ is $\SAA (Y)$-measurable.
    \end{enumerate}
    The set of transition probability kernels from $Y$ to $X$ is denoted by $\tpk(Y,X)$.

    With the notations above, for every $y\in Y$, denote by $\cQ_y$ the probability measure that assigns each measurable set $A\in \SAA (X)$ the value $\cQ (y,  A)$. In other words,
	\begin{equation*}
		\cQ_y (A)\= \cQ (y,  A).
	\end{equation*}

    Moreover, if $Y=X$, we also call a transition probability kernel from $Y$ to $X$ a \emph{transition probability kernel on $X$}. The set of transition probability kernel on $X$ is also denoted by $\tpk(X)$.
\end{definition}

%

%

\begin{definition}\label{support}
	Let $T$ be a correspondence on a compact metric space $X$ and $\cQ$ be a transition probability kernel on $(X, \SBB (X))$. We say that $\cQ$ is \defn{supported by}\index{support} $T$ if the measure $\cQ_x$ is supported on the closed set $T(x)$ (i.e., $\cQ_x (T(x))=1$) for every $x\in X$. Denote by $\tpk(X;T)$ be the set of all $\cQ \in \tpk(X)$ supported by $T$.
\end{definition}

In the remaining of this section, let $(X,  \SAA (X))$, $(Y,  \SAA (Y))$, and $(Z,  \SAA (Z))$ be measurable spaces. If $F\: Y\to X$ is a measurable map, we can pull back a function $f\: X\to \R$ (with the resulting pullback $F^* \: f\mapsto f\circ F$) and push forward a probability measure $\mu$ on $Y$ (with the resulting pushforward $F_* \: \mu \mapsto \mu\circ F^{-1}$). For a Markov chain with the state space $X=\oneto{d}$ and a transition matrix $P$, a distribution $p=(p_1,  \dots   p_d)$ on $X$ becomes $pP$ after one step of the Markov process. By Definitions~\ref{2ex transition probability kernel (1)} and~\ref{2ex transition probability kernel (2)}, transition probability kernels generalize measurable maps and transition matrices. Their actions on functions and measures are standard; see e.g., \cite[Section~6.1]{Le16} and \cite[Section~3.4.2]{MT12}. We recall them below.

\begin{definition}  \label{transition probability kernel act on functions}
   For $f \in \BBB(X,\R)$ and $\cQ \in \tpk( Y, X  )$,
    the \defn{pullback}\index{pullback} function $\cQ f\: Y\to \R$ of $f$ by $\cQ$ is given by
    \begin{equation*}
    	\cQ f(y)\= \int_X\! f(x)  \diff\cQ_y (x)    \qquad \text{for all } y\in Y.
    \end{equation*}
\end{definition}

As an operator acting on $\BBB(X,\R)$, $\cQ$ is linear and continuous (see Lemma~\ref{continuity of Q}).

\begin{definition}  \label{transition probability kernel act on measures}
    For $\mu \in \PPP(Y)$ and $\cQ \in \tpk(Y,X)$, the \defn{pushforward}\index{pushforward} probability measure $\mu \cQ$ on $X$ of $\mu$ by $\cQ$ is given by
    \begin{equation*}
    	(\mu \cQ) (A)\= \int_Y \! \cQ (y,  A)   \diff\mu(y)   \qquad \text{for all } A\in \SAA(X).
    \end{equation*}
\end{definition}

It is straightforward to check that $\mu \cQ \in \PPP(X)$.

\begin{definition}\label{invariant measure}
	For $\cQ \in \tpk( X )$, we say that $\mu \in \PPP(X)$ is \defn{$\cQ$-invariant}\index{invariant probability measure} if $\mu \cQ =\mu$. Denote by $\MMM(X,  \cQ)$ the set of all $\cQ$-invariant probability measures on $X$.
\end{definition}

\begin{definition}\label{compose of transition probability kernels}
	For $\cQ\in \tpk( Y,X )$ and $\cQ ' \in \tpk( Z, Y  )$, the transition probability kernel $\cQ ' \cQ \in \tpk(  Z, X)$ is given by
	\begin{equation*}
		(\cQ ' \cQ )(z,  A)\= (\cQ_z '\cQ ) (A)
	\end{equation*}
	for all $z\in Z$ and $A\in \SAA (X)$, where $\cQ_z '\cQ$ is a probability measure on $X$ defined in Definition~\ref{transition probability kernel act on measures}.
\end{definition}

It is straightforward to check from the definition that $\cQ ' \cQ$ is indeed a transition probability kernel from $Z$ to $X$ (cf.~\cite[Lemma~3.3~(i)]{Ka21}). The next lemma ensures that we can write $\mu \cQ \cQ '$ without parentheses. Its proof follows easily from the definition (cf.~\cite[Lemma~3.3~(iii)]{Ka21}).

\begin{lemma}\label{law of association for measures}
	For $\cQ \in \tpk( Y,X )$, $\cQ ' \in \tpk( Z, X )$, and $\mu \in \PPP(Z)$, we have the law of association:
$\mu (\cQ ' \cQ)= (\mu \cQ ')\cQ$.
\end{lemma}

%
%
%

\subsection{Transition probability kernels $\cQ^{\zeroto{n}}$ and $\cQ^\omega$}\label{subsct_Transition probability kernel Q^[n] and Q^omega}

Here we give the definition of transition probability kernels $\cQ^{\zeroto{n}}$ ($n\in \N_0$) and $\cQ^\omega$, which are repeatedly used in the sequel.

Consider $m \in \N \smallsetminus \{ 1 \}$, $\vect{n}{1}{m-1} \in \N^{m-1}$, $n_m \in \wh{\N}$, and a subset $B_i \subseteq X^{n_i}$ for each $i\in \oneto{m}$. Set $N_0 \=0$ and $N_i \= \sum_{j=1}^i   n_j \in \wh{\N}$ for each $i\in \oneto{m}$. The set $B_1 \times \cdots \times B_m$ is defined as
\begin{equation*}
\begin{aligned}
	B_1 \times \cdots \times B_m \= \bigl\{\vect{x}{1}{\infty}\in X^{N_m}  :  \vect{x}{N_{k-1} +1}{N_{k-1} +n_k}  \in B_k \text{ for each } k \in \oneto{m} \bigr\},
\end{aligned}
\end{equation*}
where $\vect{x}{N_{m-1} +1}{N_{m-1} +n_m}$ means $\vect{x}{N_{m-1} +1}{\infty}$ if $n_m = \omega$.

For each $n\in \N$, denote by $\SAA (X^n)$ the $\sigma$-algebra on $X^n$ generated by $\bigcup_{i=0}^{n-1} \bigl\{ X^i \times A \times X^{n-1-i}  :  A\in \SAA (X) \bigr\}$. Denote by $\SAA (X^\omega)$ the $\sigma$-algebra on $X^\omega$ generated by $\bigcup_{i=0}^{+\infty} \bigl\{ X^i \times A \times X^\omega  :  A\in \SAA (X) \bigr\}$.

For each $A_{n+1}\subseteq X^{n+1}$ and each $\vect{x}{1}{n} \in X^n$, write
\begin{equation}\label{pi_n+1(x_1,...,x_n;A_n+1)=}
    \pi_{n+1} (\vect{x}{1}{n}  ;  A_{n+1})\= \bigl\{x_{n+1}\in X :  \vect{x}{1}{n+1}  \in A_{n+1}\bigr\}.
\end{equation}

\begin{definition}\label{Q^[n]}
	For $\cQ \in \tpk( X )$, define the transition probability kernel $\cQ^{\zeroto{n}}$ from $X$ to $X^{n+1}$ recursively on $n\in \N_0$ as follows:

	First, $\cQ^{\zeroto{0}}\= \wh{\operatorname{id}_X}$, where $\wh{\operatorname{id}_X}$ is a transition probability kernel given by $\wh{\operatorname{id}_X} (x,  A)\= \mathbbold{1}_A (x)$ for all $x\in X$ and $A\in \SAA (X)$. This means $\cQ^{\zeroto{0}}_x =\delta_x$, the Dirac measure at $x\in X$, for all $x\in X$. If $\cQ^{\zeroto{n-1}}$ has been defined for some $n\in \N$, we define $\cQ^{\zeroto{n}}$ as:
	\begin{equation}\label{cQ^n+1(x,A_n+1)=}
		\cQ^{\zeroto{n}} (x,  A_{n+1})\= \int_{X^n}\! \cQ (x_n,  \pi_{n+1} (\vect{x}{1}{n}  ;  A_{n+1}))   \diff\cQ_x^{\zeroto{n-1}} (\vect{x}{1}{n} )
	\end{equation}
	for all $x\in X$ and $A_{n+1}\in \SAA \bigl(X^{n+1} \bigr)$.
\end{definition}

It is standard and straightforward to check from the definition inductively that $\cQ^{\zeroto{n}}$ defined above is indeed a transition probability kernel for each $n\in \N$ (cf.~\cite[Lemma~3.3~(i)]{Ka21}).
Applying the Kolmogorov extension theorem, we get the following definition (cf.~\cite[Sections~14.3--14.4]{Ka21}).

\begin{definition}\label{Q^omega}
	For $\cQ \in \tpk( X )$, define $\cQ^\omega\in \tpk( X,X^\omega )$ as the unique transition probability kernel from $X$ to $X^\omega$ with the property that for all $x\in X$, $n\in \N_0$, and $A\in \SAA (X^{n+1})$, 
    \begin{equation}\label{Q^N(x,A*X^infty)=Q^[n](x,A)}
		\cQ^\omega (x,  A\times X^\omega)= \cQ^{\zeroto{n}} (x,  A).
    \end{equation}
\end{definition}

\begin{rem}
	For each $\mu \in \PPP(X)$, Definition~\ref{transition probability kernel act on measures} and~(\ref{Q^N(x,A*X^infty)=Q^[n](x,A)}) imply that
	\begin{equation}\label{muQ^N(A*X^infty)=muQ^[n](A)}
		(\mu \cQ^\omega) (A\times X^\omega)= \bigl(\mu \cQ^{\zeroto{n}} \bigr) (A).
	\end{equation}
\end{rem}

\subsection{Definition of measure-theoretic entropy for transition probability kernels}\label{subsct_Definition of measure-theoretic entropy for transition probability kernels}

In this subsection, we introduce the measure-theoretic entropy for transition probability kernels (Definition~\ref{measure-theoretic entropy of transition probability kernels}). As we can see in Definition~\ref{2ex transition probability kernel (1)}, transition probability kernels generalize measurable maps. Our definition of the measure-theoretic entropy of transition probability kernels uses the entropy of partitions and generalizes naturally the definition of the measure-theoretic entropy of measurable maps.

A \defn{finite measurable partition}\index{finite measurable partition} $\cA$ of a measurable space $(Y,  \SAA (Y))$ is a finite collection of mutually disjoint measurable subsets $\{A_1 ,\, \dots ,\, A_n\}$ satisfying $\bigcup_{i=1}^n A_i =Y$, where $n \in \N$.

Let $m \in \N \smallsetminus \{1\}$, $\vect{n}{1}{m-1}   \in \N^{m-1}$, and $n_m \in \hN$. For arbitrary finite measurable partitions $\cA_1 ,\, \dots ,\, \cA_m$ of measurable spaces $(X^{n_1} ,  \SAA (X^{n_1}))$, $\dots$, $( X_{n_m} ,  \SAA (X^{n_m}))$, respectively, their product is given by
\begin{equation*}
	\cA_1 \times \cdots \times \cA_m 
	\= \{A_1 \times \cdots \times A_m  :  A_i \in \cA_i \text{ for every } i\in \oneto{m} \} \subseteq \SAA (X^{n_1 +\dots +n_m}).
\end{equation*}
It is a finite measurable partition of $(X^{n_1 +\dots +n_m} ,  \SAA (X^{n_1 +\dots +n_m}))$.

For a finite measurable partition $\cA$ of $(X ,\SAA (X))$ and $n\in \N$, write
\begin{equation*}
	\cA^n \= \underbrace{\cA \times \cA \times \cdots \times \cA}_{n \text{ copies of } \cA}.
\end{equation*}

Recall that for a finite measurable partition $\cA$ of $(X ,\SAA (X))$ and $\mu \in \PPP(X)$, the entropy of $\cA$ is given by
\begin{equation}\label{H_mu(A)=}
	H_\mu (\cA)\= -\sum_{A\in \cA}   \mu (A) \log (\mu (A)).
\end{equation}

We refer the reader to \cite[Chapter~2]{PU10} for basic properties of the entropy $H_\nu(\cA_1)$ of a finite measurable partition $\cA_1$ and the \emph{conditional entropy} $H_{\nu}(\cA_1 | \cA_2)$ of a finite measurable partition $\cA_1$ given another finite measurable partition $\cA_2$.

\begin{prop}\label{q9pg237pq}
	Let $\cQ \in \tpk( X  )$, $\mu \in \MMM(X, \cQ)$, and $\cA$ be a finite measurable partition of $X$. Then
	\begin{equation*}%
		H_{\mu \cQ^{\zeroto{n-1}}} (\cA^n)= \sum_{k=0}^{n-1}   H_{\mu \cQ^\omega} \bigl( \cA \times \{ X^\omega\}\big| \{ X\} \times \cA^k \times \{ X^\omega \} \bigr),
	\end{equation*}
	where $\{ X^\omega \}$ (resp.\ $\{ X \}$) is the partition of $X^\omega$ (resp.\ $X$) into only one subset.

	Moreover, the limit $\lim_{n\to +\infty} \frac{1}{n} H_{\mu \cQ^{\zeroto{n-1}}} (\cA^n)$ exists.
\end{prop}

\begin{proof}    
    By (\ref{muQ^N(A*X^infty)=muQ^[n](A)}) %
    and~(\ref{H_mu(A)=}), we have $H_{\mu \cQ^{\zeroto{n-1}}} (\cA^n)= H_{\mu \cQ^\omega} (\cA^n \times \{ X^\omega \})$. By (\ref{muQ^omega(B)=muQ^omega(X*B)}) in Corollary~\ref{c:muQ} and~(\ref{H_mu(A)=}), we have $H_{\mu \cQ^\omega} \bigl(\{ X\} \times \cA^k \times \{ X^\omega \} \bigr)= H_{\mu \cQ^\omega} \bigl(\cA^k \times \{ X^\omega \} \bigr)$. Thus,
	\begin{alignat*}{4}
		H_{\mu \cQ^{\zeroto{n-1}}} (\cA^n)
		= H_{\mu \cQ^\omega} (\cA^n \times \{ X^\omega \})  
		&= \sum_{k=0}^{n-1}  \bigl(H_{\mu \cQ^\omega} \bigl(\cA^{k+1} \times \{ X^\omega \}\bigr)- H_{\mu \cQ^\omega} \bigl(\cA^k \times \{ X^\omega \}\bigr)\bigr)\\
		&= \sum_{k=0}^{n-1}  \bigl(H_{\mu \cQ^\omega} \bigl(\cA^{k+1} \times \{ X^\omega \} \bigr)- H_{\mu \cQ^\omega} \bigl(\{ X\} \times \cA^k \times \{ X^\omega \}\bigr)\bigr)\\
		&= \sum_{k=0}^{n-1}   H_{\mu \cQ^\omega} \bigl(\cA \times \{ X^\omega\} \big| \{ X\} \times \cA^k \times \{ X^\omega \} \bigr).
	\end{alignat*}
   Here $\cA^0 \times \{ X^\omega \}$ is $\{ X^\omega \}$, whose entropy is $0$.

	Since $H_{\mu \cQ^\omega} \bigl(\cA \times \{ X^\omega\} \big| \{ X\} \times \cA^k \times \{ X^\omega \}\bigr)$ is non-negative and it decreases as $k$ increases, we get
	\begin{equation*}
		\lim_{n\to +\infty} \frac{1}{n} H_{\mu \cQ^{\zeroto{n-1}}} (\cA^n) = \lim_{k\to +\infty} H_{\mu \cQ^\omega} \bigl(\cA \times \{ X^\omega\} \big| \{ X\} \times \cA^k \times \{ X^\omega \}\bigr).
	\end{equation*}

	Therefore, the limit $\lim_{n\to +\infty}\frac{1}{n} H_{\mu \cQ^{\zeroto{n-1}}} (\cA^n)$ exists.
\end{proof}

This proposition guarantees that $h_\mu (\cQ,  \cA)$ in the following definition is well-defined.

\begin{definition}\label{h_mu (Q,A)}
	Let $\cQ \in \tpk( X )$, $\mu \in \MMM(X,\cQ)$, and $\cA$ be a finite measurable partition of $X$. Then $h_\mu (\cQ ,  \cA)$, the \defn{measure-theoretic entropy of $\cQ$ with respect to the partition}\index{measure-theoretic entropy with respect to partition} $\cA$, is defined as
	\begin{equation*}
		h_\mu (\cQ ,  \cA)\= \lim_{n\to +\infty} \frac{1}{n} H_{\mu \cQ^{\zeroto{n-1}}} (\cA^n).
	\end{equation*}
\end{definition}

We now formulate our definition of the measure-theoretic entropy of a transition probability kernel.

\begin{definition}\label{measure-theoretic entropy of transition probability kernels}
	For $\cQ \in \tpk( X )$ and $\mu \in \MMM(X,\cQ)$, the \defn{measure-theoretic entropy}\index{measure-theoretic entropy} $h_\mu (\cQ)$ (of $\cQ$ for $\mu$) is given by
	\begin{equation*}
		h_\mu (\cQ)\= \sup_{\cA} h_\mu (\cQ ,  \cA),
	\end{equation*}
	where $\cA$ ranges over all finite measurable partitions of $X$.
\end{definition}

Recall that $\gamma_2$ is the reversal map on $X^2$ given by $\gamma_2 (x,y) =(y,x)$.

\begin{prop}\label{measure-theoretic entropy is inversely invariant}
	For $\cQ ,\, \cR \in \tpk( X )$, and $\mu \in \cP (X)$, if $\mu \cQ^{[1]} =\bigl(\mu \cR^{[1]}\bigr) \circ \gamma_2^{-1}$, then $\mu \in \MMM (X , \cQ) \cap \MMM (X ,\cR)$ and $h_\mu (\cQ) =h_\mu (\cR)$.
\end{prop}

\begin{proof}
	Suppose $\mu \cQ^{[1]} =\bigl(\mu \cR^{[1]}\bigr) \circ \gamma_2^{-1}$. Lemma~\ref{A9} indicates that $\mu \in \MMM (X , \cQ) \cap \MMM (X ,\cR)$.
	
	By Lemma~\ref{A9} and~(\ref{H_mu(A)=}), for every finite measurable partition $\cA$ of $X$ and every $n\in \N$, we have $h_{\mu \cQ^{[n]}} \bigl(\cA^{n+1} \bigr) =h_{\mu \cR^{[n]}} \bigl(\cA^{n+1} \bigr)$. Consequently, by Definition~\ref{h_mu (Q,A)}, we have $h_\mu (\cQ , \cA) =h_\mu (\cR ,\cA)$. Therefore, by Definition~\ref{measure-theoretic entropy of transition probability kernels}, $h_\mu (\cQ) =h_\mu (\cR)$.
\end{proof}

\subsection{A characterization of the measure-theoretic entropy}\label{subsct_Measure-theoretic entropy of transition probability kernels and the shift map are equal}

In this subsection, we aim to establish Theorem~\ref{measure-theoretic entropy coincide with its lift}.
Denote by $\sigma$ the shift map on $X^\omega$ given by
$\sigma (\vect{x}{1}{\infty} )\= \vect{x}{2}{\infty}$
for $\vect{x}{1}{\infty} \in X^\omega$. 
For two arbitrary finite measurable partitions $\cA_1$ and $\cA_2$ of $X$, the finite measurable partition $\cA_1 \vee \cA_2$ is given by $\cA_1 \vee \cA_2 \= \{A_1 \cap A_2  :  A_1 \in \cA_1 ,\, A_2 \in \cA_2 \}$.

Let $\cQ \in \tpk( X )$ and $\mu \in \MMM(X,  \cQ)$. For arbitrary $n\in \N$ and measurable set $A\in \SAA (X^n)$, by (\ref{muQ^omega(B)=muQ^omega(X*B)}) in Corollary~\ref{c:muQ} we have
\begin{equation*}
	(\mu \cQ^\omega) \bigl( \sigma^{-1} (A\times X^\omega) \bigr)= (\mu \cQ^\omega) (X \times A \times X^\omega)= (\mu \cQ^\omega) (A\times X^\omega).
\end{equation*}

By Dynkin's $\pi$-$\lambda$ theorem, we get $(\mu \cQ^\omega) \circ \sigma^{-1} =\mu \cQ^\omega$, which means that $\mu \cQ^\omega$ is $\sigma$-invariant. As a result, the \emph{measure-theoretic entropy $h_{\mu \cQ^\omega} (\sigma)$ of $\sigma$ for $\mu \cQ^\omega$} is well-defined in the sense of classical ergodic theory for single-valued maps (see e.g., \cite[Chapter~2]{PU10}).

\begin{theorem}\label{measure-theoretic entropy coincide with its lift}
    Let $\cQ$ be a transition probability kernel on a measurable space $(X,  \SAA (X))$, $\mu \in \MMM(X, \cQ)$, and $\sigma$ be the shift map on $X^\omega$. Then we have
    $h_\mu (\cQ)= h_{\mu \cQ^\omega} (\sigma)$.
\end{theorem}

Before establishing Theorem~\ref{measure-theoretic entropy coincide with its lift}, we give three lemmas.

\begin{lemma}\label{h_mu(Q,A)=h_muQ^N(sigma,A*X^+infty)}
    Let $\cQ \in \tpk( X )$, $\mu \in \MMM( X, \cQ)$, $\sigma$ be the shift map on $X^\omega$, and $\cA$ be a finite measurable partition of $X$. Then
    \begin{equation*}
    	h_\mu (\cQ ,  \cA)= h_{\mu \cQ^\omega} (\sigma ,  \cA \times \{ X^\omega \}).
    \end{equation*}
\end{lemma}

Here the definition of the \emph{measure-theoretic entropy $h_{\mu \cQ^\omega} (\sigma ,  \cA \times \{ X^\omega \})$ of the single-valued map $\sigma$ with respect to the partition} $\cA \times \{ X^\omega \}$ can be found in \cite[Lemma~2.4.2]{PU10}.

\begin{proof}
    For each $n\in \N$, we have 
    $\bigvee_{j=0}^{n-1} \sigma^{-j} (\cA \times \{ X^\omega \})= \bigvee_{j=0}^{n-1} \{ X^j \} \times \cA \times \{ X^\omega \}= \cA^n \times \{ X^\omega \}$.
    Recall $H_{\mu \cQ^{\zeroto{n-1}}} (\cA^n)= H_{\mu \cQ^\omega} (\cA^n \times \{ X^\omega \})$ from the proof of Proposition~\ref{q9pg237pq}. Therefore,
    \begin{alignat*}{4}
        h_\mu (\cQ ,  \cA)
        &= \lim_{n\to +\infty} \frac{1}{n} H_{\mu \cQ^{\zeroto{n-1}}} (\cA^n)
        = \lim_{n\to +\infty} \frac{1}{n} H_{\mu \cQ^\omega} (\cA^n \times \{ X^\omega \})\\
        &= \lim_{n\to +\infty} \frac{1}{n} H_{\mu \cQ^\omega} \biggl( \bigvee_{j=0}^{n-1} \sigma^{-j} (\cA \times \{ X^\omega \}) \biggr)
        = h_{\mu \cQ^\omega} (\sigma ,  \cA \times \{ X^\omega \}),
    \end{alignat*}
    where the last equality follows from \cite[Lemma~2.4.2]{PU10}. %
\end{proof}

\begin{lemma}[{\cite[Theorem~4.12~(vi)]{Wa82}}]\label{h_mu(T,A)=h_mu(T,vee_i=0^k T^-i(A))}
     Let $F\: X\to X$ be a measurable map, $\mu \in \MMM(X,F)$, $\cA$ be a finite measurable partition of $X$, and $k\in \N$. Then $h_\mu (F,  \cA)= h_\mu \bigl(F,  \cA \vee F^{-1} (\cA) \vee \cdots \vee F^{-k} (\cA)\bigr)$.
\end{lemma}

\begin{lemma}[{\cite[Theorem~4.21]{Wa82}}]\label{some partitions determine the measure-theoretic entropy}%

     Let $F\: X\to X$ be a measurable map, $\mu \in \MMM(X,F)$, and $\cS$ be a sub-algebra %
    of $\SAA (X)$ satisfying that $\SAA (X)$ is the $\sigma$-algebra generated by $\cS$.
    Then 
    \begin{equation*}
    	h_\mu (F)= \sup_{\cA} h_\mu (F,  \cA),
    \end{equation*}
    where $\cA$ ranges over all finite partitions of $Y$ satisfying $\cA \subseteq \cS$.
\end{lemma}

\begin{proof}[Proof of Theorem~\ref{measure-theoretic entropy coincide with its lift}]
    First, we set some notations:
    \begin{align*}
    \fA_0 &\= \{\cA^n \times \{ X^\omega \}  :  \cA \text{ is a finite measurable partition of } X \text{ and } n\in \N\},\\
    \cG     &\= \{B_1 \times \cdots \times B_n \times X^\omega  :  n\in \N \text{ and } B_1 ,\, \dots ,\, B_n \in \SAA (X)\},\\
    \cS      &\= \{ E  :  E \text{ is a finite union of sets in } \cG \},\\
    \fA      &\= \{ \cD  :  \cD \text{ is a finite measurable partition of } X^\omega \text{ and } \cD \subseteq \cS \}.
    \end{align*}

    We can verify that $\cS$ is a sub-algebra of $\SAA (X^\omega)$ and that $\SAA (X^\omega)$ is the $\sigma$-algebra generated by $\cS$. So by Lemma~\ref{some partitions determine the measure-theoretic entropy}, we get
    \begin{equation}\label{3q8740r}
        h_{\mu \cQ^\omega} (\sigma)= \sup \{ h_{\mu \cQ^\omega} (\sigma ,  \cD)  :  \cD \in \fA \}.
    \end{equation}

    Fix an arbitrary $\cD \in \fA$. Suppose $\cD =\{ D_1 ,\, \dots ,\, D_p \}$, $D_i =G_{i1} \cup \cdots \cup G_{iq_i}\in \cS$ for all $i\in \oneto{p}$, and $G_{ij}=B_{ij1} \times \cdots \times B_{ijr_{ij}} \times X^\omega \in \cG$ for all $i\in \oneto{p}$ and $j\in \oneto{q_i}$, where $B_{ijk} \in \SAA (X)$ for all $i\in \oneto{p}$, $j\in \oneto{q_i}$, and $k\in \oneto{r_{ij}}$. Set $\cA \=\bigvee_{i=1}^p \bigvee_{j=1}^{q_i} \bigvee_{k=1}^{r_{ij}} \bigl\{ B_{ijk},\, B_{ijk}^c \bigr\}$, a finite measurable partition of $X$, then $\cA^n \times \{ X^\omega \} \in \fA_0$ is a finer partition of $X^\omega$ than $\cD$, where $n\=\max \{ r_{ij}:  i\in \oneto{p},\, j\in \oneto{q_i} \}$.

    Hence, for each partition in $\fA$, we can find a finer partition in $\fA_0$.

    We can conclude by (\ref{h_mu(T,A)=}) that if a finite measurable partition $\cA '$ of $X^\omega$ is finer than another finite measurable partition $\cA ''$, then $h_{\mu \cQ^\omega} (\sigma ,  \cA ')\geq h_{\mu \cQ^\omega} (\sigma ,  \cA '')$. Thus, $\sup \{ h_{\mu \cQ^\omega} (\sigma ,  \cD)  :  \cD \in \fA \} \leq \sup \{ h_{\mu \cQ^\omega} (\sigma ,  \cA_0)  :  \cA_0 \in \fA_0 \}$. Moreover, we can check that $\fA_0 \subseteq \fA$, and thus
    \begin{equation}\label{sf9e}
        \sup \{ h_{\mu \cQ^\omega} (\sigma ,  \cD)  :  \cD \in \fA \} = \sup \{ h_{\mu \cQ^\omega} (\sigma ,  \cA_0)  :  \cA_0 \in \fA_0 \}.
    \end{equation}

    For each $\cA^n \times \{ X^\omega \} \in \fA_0$,  where $\cA$ is a measurable partition of $X$ and $n\in \N$, recall $\bigvee_{j=0}^{n-1} \sigma^{-j} (\cA \times \{ X^\omega \})= \cA^n \times \{ X^\omega \}$ from the proof of Lemma~\ref{h_mu(Q,A)=h_muQ^N(sigma,A*X^+infty)}. By (\ref{3q8740r}), (\ref{sf9e}), and Lemma~\ref{h_mu(T,A)=h_mu(T,vee_i=0^k T^-i(A))}, we have
    \begin{equation*}
    \begin{aligned}
        h_{\mu \cQ^\omega} (\sigma)
        &= \sup \{ h_{\mu \cQ^\omega} (\sigma ,  \cD)  :  \cD \in \fA \} 
        = \sup \{ h_{\mu \cQ^\omega} (\sigma ,  \cA_0)  :  \cA_0 \in \fA_0 \}
        = \sup_{\cA} h_{\mu \cQ^\omega} (\sigma ,  \cA^n \times \{ X^\omega \})\\ %
        &= \sup_{\cA} h_{\mu \cQ^\omega} \biggl(\sigma ,  \bigvee_{i=0}^{n-1} \sigma^{-i} (\cA \times \{ X^\omega \})\biggr)
        = \sup_{\cA} h_{\mu \cQ^\omega} (\sigma ,  \cA \times \{ X^\omega \}),
    \end{aligned}
    \end{equation*}
    where $\cA$ ranges over all finite measurable partitions of $X$.

    Therefore, by Definition~\ref{measure-theoretic entropy of transition probability kernels} and Lemma~\ref{h_mu(Q,A)=h_muQ^N(sigma,A*X^+infty)} we get $h_{\mu \cQ^\omega} (\sigma) =\sup_\cA h_\mu (\cQ ,  \cA)= h_\mu (\cQ)$, where $\cA$ ranges over all finite measurable partitions of $X$.
\end{proof}

The next corollary, which is useful in Subsection~\ref{subsct_Lee--Lyubich--Markorov--Mukherjee anti-holomorphic correspondences}, follows from Theorem~\ref{measure-theoretic entropy coincide with its lift}:

\begin{cor}\label{restrict measure-theoretic entropy}
	Let $\cQ$ be a transition probability kernel on a measurable space $(X , \SAA (X))$, $Y \in \SAA (X)$, and $\cQ '$ be a transition probability kernel on the measurable space $(Y ,\SAA (Y))$, where $\SAA (Y)$ refers to the $\sigma$-algebra induced by $\SAA (X)$. Suppose that $\mu \in \MMM (Y , \cQ ')$ and $\hmu \in \cP (X)$ satisfy $\hmu (A) =\mu (A)$ for all $A \in \SAA (Y)$. If for each $A \in \SAA (Y)$, the equality $\cQ (y ,A) =\cQ '(y ,A)$ holds for $\mu$-almost every $y\in Y$, then we have $\hmu \in \MMM (X ,\cQ)$ and $h_{\hmu} (\cQ) =h_\mu (\cQ ')$.
\end{cor}

\begin{proof}
	Denote by $\sigma_X$ the shift map on $X^\omega$, and $\sigma_Y \= \sigma_X |_{Y^\omega}$ be the shift map on $Y^\omega$. The conditions above indicate that $\hmu \cQ =\hmu$ by (\ref{transition probability kernel act on measures}), and that the measure $\hmu \cQ^\omega$ is the extension of the measure $\mu \cQ'^\omega$ from $Y^\omega$ to $X^\omega$, i.e., $(\hmu \cQ^\omega) (B) =(\mu \cQ'^\omega) (B \cap Y^\omega)$ holds for all $B \in \SAA (X^\omega)$. Thus, the inclusion map from $Y^\omega$ to $X^\omega$ is an isomorphism between measure-preserving systems $(Y^\omega ,\SAA (Y^\omega), \sigma_Y ,\mu \cQ '^\omega)$ and $(X^\omega ,\SAA (X^\omega) ,\sigma_X ,\hmu \cQ^\omega)$, so $h_{\mu \cQ '^\omega} (\sigma_Y) =h_{\hmu \cQ^\omega} (\sigma_X)$. Then $h_{\hmu} (\cQ) =h_\mu (\cQ ')$ by Theorem~\ref{measure-theoretic entropy coincide with its lift}.
\end{proof}

\section{Variational Principle for forward expansive correspondences}\label{sct_Variational_principle_for_positively_RW-expansive_correspondences}

In this section, we establish the Variational Principle when the correspondence $T$ on a compact metric space $(X,  d)$ has a property called \emph{forward expansive}. More precisely, we will prove Theorem~\ref{t_VP_forward_exp}, the first main result of this article. The proof of this theorem is the most technical part of this article. %
In Subsection~\ref{subsct_Forward expansiveness}, we introduce forward expansiveness for correspondences. Subsection~\ref{subsct_Rokhlin formula for measure-theoretic entropy of transition probability kernels and of the corresponding shift maps} is devoted to establishing a version of the Rokhlin formula for measure-theoretic entropy of transition probability kernels and of the corresponding shift maps. Theorem~\ref{t_HVP} is established in Subsection~\ref{subsct_HVP}, and is used in the proof of Theorem~\ref{t_VP_forward_exp}. Finally, in Subsection~\ref{subsct_proof of Theorem}, we derive an inequality about measure-theoretic entropy (Proposition~\ref{measure-theoretic entropy less than the corresponding Markov measure-theoretic entropy}) from the Rokhlin formulas and establish Theorem~\ref{t_VP_forward_exp}.

\subsection{Forward expansiveness}\label{subsct_Forward expansiveness}

{\ifFirstInitial R.~K.~\fi}Williams has defined a type of expansiveness for correspondences or set-valued functions in \cite[Definition~3]{Wi70}, which is called \defn{RW-expansiveness}\index{RW-expansiveness} by {\ifFirstInitial M.~J.~\fi}Pac\'ifico and {\ifFirstInitial J.~L.~\fi}Vieitez in \cite[Definition~3.2]{PV17}. In Definition~\ref{forward expansiveness} below, what we call \emph{forward expansiveness} is inspired but different from what {\ifFirstInitial M.~J.~\fi}Pac\'ifico and {\ifFirstInitial J.~L.~\fi}Vieitez called \defn{RW-expansiveness}.

\begin{definition}[Forward expansiveness]\label{forward expansiveness}
    Let $T$ be a correspondence on a compact metric space $(X,  d)$. We say that $T$ is \defn{forward expansive}\index{forward expansive}, if there exists a number $\epsilon >0$ such that for each pair of distinct orbits $\vect{x}{1}{\infty},\, \vect{y}{1}{\infty}\in \cO_\omega (T)$, we have $d(x_n ,  y_n)>\epsilon$ for some $n\in \N$. Such a positive number $\epsilon$ is called an \defn{expansive constant}\index{expansive constant} of $T$.
\end{definition}

\begin{rem}\label{preimage of forward expansive correspondence is finite}
    Let $T$ be a forward expansive correspondence on a compact metric space $(X,  d)$ with an expansive constant $\epsilon >0$. Fix an arbitrary point $x_1 \in X$ and choose an orbit $\vect{x}{1}{\infty} \in \cO_\omega (T)$. For a pair of distinct points $x_0 ,\, x_0 ' \in X$ satisfying $x_1\in T(x_0) \cap T(x_0 ')$, we have $\vect{x}{0}{\infty} \in \cO_\omega (T)$ and $(x_0 ' ,  x_1 ,  \dots )\in \cO_\omega (T)$. Then the forward expansiveness of $T$ yields $d(x_0 ,  x_0 ')> \epsilon$. Since $X$ is compact, we know that $T^{-1} (x)= \{ y\in X :  x\in T(y) \} \leq M_\epsilon <+\infty$ for all $x\in X$, where $M_\epsilon$ refers to the largest cardinality of an $\epsilon$-separated subset of $X$.
\end{rem}

If a correspondence $T$ on $X$ degenerates to a singe-valued continuous map, then the forward expansiveness for $T$ is equivalent to the forward expansiveness for the corresponding single-valued map. Moreover, the following proposition indicate that the forward expansiveness of $(\cO_\omega (T),  \sigma)$ and the forward expansiveness of $T$ are equivalent, and is straightforward to verify from direct calculations.

\begin{prop}\label{expansiveness of correspondence with its lift}
	Let $T$ be a correspondence on a compact metric space $(X, d)$ and $\sigma \: \cO_\omega (T) \to \cO_\omega (T)$ be the shift map. If $T$ is forward expansive with an expansive constant $\epsilon>0$, then $\sigma$ is forward expansive with an expansive constant $\frac{\epsilon}{2(1+\epsilon)}$.
	Conversely, if $\sigma$ is forward expansive with an expansive constant $\epsilon \in (0,1)$, then $T$ is forward expansive with an expansive constant $\frac{\epsilon}{1-\epsilon}$.
\end{prop}

Let $X$ be a compact metric space and $T$ be a correspondence on $X$. In Subsections~\ref{subsct_Measure-theoretic entropy of transition probability kernels and the shift map are equal} and~\ref{subsct_A characterization of the topological pressure}, we gave two different shift maps, one on $X^\omega$ and one on $\cO_\omega (T)$. We denote below by $\sigma ' \: X^\omega \to X^\omega$ the shift map in Subsection~\ref{subsct_Measure-theoretic entropy of transition probability kernels and the shift map are equal}, and by $\sigma \: \cO_\omega (T) \to \cO_\omega (T)$ the shift map in Subsection~\ref{subsct_A characterization of the topological pressure}.

\begin{lemma}\label{positive measurability of forward expansive correspondence}
	If a correspondence $T$ on a compact metric space $X$ is forward expansive, then for every $A \in \SBB (X)$, the set $T(A)$ is Borel measurable.
\end{lemma}

\begin{proof}
    By Remark~\ref{preimage of forward expansive correspondence is finite}, if distinct points $x,\, y \in X$ satisfy $T (x) \cap T(y) \neq \emptyset$, then $d(x,  y)> \epsilon$.
    
    Fix an arbitrary $x \in X$. Then $T(y_1) \cap T(y_2) =\emptyset$ for all $y_1 ,\, y_2 \in \overline{B_{\epsilon /2} (x)} \= \{y \in X : d(x,y) \leq \epsilon /2 \}$ with $y_1 \neq y_2$, so for each $z \in T\bigl(\overline{B_{\epsilon /2} (x)}\bigr)$, there is exactly one $y\in \overline{B_{\epsilon /2} (x)}$ satisfying $z\in T(y)$. Suppose $g_x \: T\bigl(\overline{B_{\epsilon /2} (x)}\bigr) \to \overline{B_{\epsilon /2} (x)}$ is the map with the property that $g_x (z)$ is the unique point $y \in \overline{B_{\epsilon /2} (x)}$ with $z\in T(y)$, i.e., $T= g_x^{-1}$ on $\overline{B_{\epsilon /2} (x)}$. Since $\cO_2 (T)$ is compact by the definition of correspondences, $T\bigl(\overline{B_{\epsilon /2} (x)}\bigr)$, the projection of $\cO_2 (T) \cap \overline{B_{\epsilon /2} (x)} \times X$ on the second coordinate, is compact. Since $\overline{B_{\epsilon /2} (x)}$ is compact and $\bigl\{(y,z) : z\in T\bigl(\overline{B_{\epsilon /2} (x)}\bigr) ,\, y=g_x (z) \bigr\} =\bigl\{(y,z) : y\in \overline{B_{\epsilon /2} (x)} ,\, z\in T(y)\bigr\} =\cO_2 (T) \cap \overline{B_{\epsilon /2} (x)} \times X$ is compact, we get that $g_x \: T\bigl(\overline{B_{\epsilon /2} (x)}\bigr) \to \overline{B_{\epsilon /2} (x)}$ is continuous. Thereby, for each Borel set $A \subseteq \overline{B_{\epsilon /2} (x)}$, $T(A) =g_x^{-1} (A)$ is Borel measurable in $T\bigl(\overline{B_{\epsilon /2} (x)}\bigr)$, and thus is Borel measurable in $X$ due to the fact that $T\bigl(\overline{B_{\epsilon /2} (x)}\bigr)$ is a closed subset of $X$.

    Since $X$ is compact, we can choose a finite collection of points $\{x_1,\, \dots,\, x_n\} \subseteq X$ such that $\bigl\{ \overline{B_{\epsilon / 2} (x_i)} \bigr\}_{i=1}^n$ covers $X$. Let $A \in \SBB (X)$ be arbitrary. For each $i\in \oneto{n}$, the Borel set $A \cap \overline{B_{\epsilon / 2} (x_i)}$ is contained in $\overline{B_{\epsilon / 2} (x_i)}$, and thus $T \bigl( A \cap \overline{B_{\epsilon / 2} (x_i)} \bigr)$ is Borel measurable. Therefore, $T(A) =\bigcup_{i=1}^n T\bigl( A \cap \overline{B_{\epsilon / 2} (x_i)} \bigr)$ is Borel measurable.
\end{proof}

\begin{cor}\label{positive measurability of forward expansive maps}
    Let $(X,  d)$ be a compact metric space. If a continuous map $f\: X \to X$ is forward expansive, then for each Borel measurable set $A \in \SBB (X)$, the set $f(A)$ is Borel measurable.
\end{cor}

\begin{proof}
    By Remark~\ref{r:C_f_properties}~(i), $\cC_f$ is forward expansive. Therefore, by Lemma~\ref{positive measurability of forward expansive correspondence}, $f(A) =\cC_f (A)$ is Borel measurable for all $A \in \SBB (X)$.
\end{proof}

Lemma~\ref{positive measurability of forward expansive correspondence} and Corollary~\ref{positive measurability of forward expansive maps} are important in Sections~\ref{sct_Variational_principle_for_positively_RW-expansive_correspondences} and~\ref{sct_Thermodynamic_formalism_of_expansive_correspondences}, because Theorems~\ref{t_VP_forward_exp},~\ref{equilibrium state 1}, and~\ref{phwfq9} all assume that the correspondence $T$ is forward expansive. When the correspondence $T$ is forward expansive and at the same time $\sigma \: \cO_\omega (T) \to \cO_\omega (T)$ is forward expansive, we can write $T(A)$ as an Borel subset of $X$ for every $A \in \SBB (X)$ and $\sigma (B)$ as an Borel subset of $\cO_\omega (T)$ for every $B \in \SBB (\cO_\omega (T))$.

Let $T$ be a correspondence on a compact metric space $(X,  d)$ and $\cA$ be a finite Borel measurable partition of $X$. Set
\begin{equation*}
	\mesh \cA \= \sup \{ \diam B  :  B\in \cA \}.
\end{equation*}
For each $n\in \N$, the pair $(T ,  \cA)$ induces a finite Borel measurable partition $\wt{\cA}^n_T$ of the orbit space $\cO_\omega (T)$ given by
\begin{equation*}\label{28h3ewicndi8}
	\wt{\cA}^n_T \= \{ B_1 \times \cdots \times B_n \times X^\omega \cap \cO_\omega (T)  :  B_1 ,\, \dots ,\, B_n \in \cA \}.
\end{equation*}
Note that $\wt{\cA}_T^n = \bigvee_{k=0}^{n-1} \sigma^{-k} \bigl( \wt{\cA}_T^1 \bigr)$ for all $n\in \N$.
For each $x\in \cO_\omega (T)$ and each $n\in \N$, denote by $\wt{\cA}^n_T (x)$ the element in $\wt{\cA}^n_T$ containing $x$.

\begin{lemma}\label{locally long behavior of forward expansive correspondence}
    Let $T$ be a forward expansive correspondence on a compact metric space $(X,  d)$ with expansive constants $\epsilon_1 >0$ and $\epsilon_2 >0$. There exists $L\in \N$ with the following property:
    
    \smallskip

    For each $n\in \N$ greater than $L$, if two orbits $\vect{x}{1}{n}  ,\, \vect{y}{1}{n} \in \cO_n (T)$ satisfy $d(x_k ,  y_k)< \epsilon_2$ for all $k\in \oneto{n}$, then $d(x_k ,  y_k)< \epsilon_1$ holds for all $k\in \oneto{ n-L}$.
\end{lemma}

\begin{proof}
    We argue by contradiction and assume that for every $l\in \N$, we can choose two orbits $\vect{x^l}{1}{n_l}   ,\, \vect{y^l}{1}{n_l}  \in \cO_{n_l} (T)$,  $n_l \in \N$, $n_l >l$ satisfying that $d\bigl(x_k^l ,  y_k^l\bigr)<\epsilon_2$ for every $k\in \oneto{n_l}$, and that there exists $j\in \oneto{n_l -l}$ such that $d\bigl(x_j^l ,  y_j^l\bigr)\geq \epsilon_1$. Assume $j=1$, otherwise substitute $\vect{x^l}{j}{n_l}$ and $\vect{y^l}{j}{n_l}$ for $\vect{x^l}{1}{n_l}$ and $\vect{y^l}{1}{n_l}$, respectively.

    Extend each pair of orbits $\vect{x^l}{1}{n_l}   ,\, \vect{y^l}{1}{n_l} \in \cO_{n_l} (T)$ to $\vect{x^l}{1}{\infty}   ,\, \vect{y^l}{1}{\infty}\in \cO_\omega (T)$. Since $\cO_\omega (T) \times \cO_\omega (T)$ is compact, we can choose an increasing sequence of positive integers $l_r\in \N$, $r\in \N$, such that $\vect{x^{l_r}}{1}{\infty}$ and $\vect{y^{l_r}}{1}{\infty}$ converge to $\vect{x^0}{1}{\infty}\in \cO_\omega (T)$ and $\vect{y^0}{1}{\infty} \in \cO_\omega (T)$ as $r\to +\infty$, respectively, i.e., for each $k\in \N$, $x_k^{l_r}$ and $y_k^{l_r}$ converge to $x_k^0$ and $y_k^0$ as $r \to +\infty$, respectively.

    Fix an arbitrary $k\in \N$, since for each $r\in \N$ with $n_{l_r} >l_r \geq k$, we have $d \bigl( x_k^{l_r} ,  y_k^{l_r} \bigr)< \epsilon_2$, and since $l_r$ tends to $+\infty$ as $r\to +\infty$, we get $d\bigl(x_k^0 ,  y_k^0\bigr)= \lim_{r\to +\infty} d\bigl( x_k^{l_r} ,  y_k^{l_r} \bigr)\leq \epsilon_2$. This implies $x_k^0 =y_k^0$ for each $k\in \N$ because $\epsilon_2$ is an expansive constant for $T$.

    Recall $d\bigl(x_1^l ,  y_1^l\bigr)\geq \epsilon_1$ for all $l\in \N$. Thus, we have 
    $0= d(x_1^0 ,  y_1^0)= \lim_{r\to +\infty} d \bigl( x_1^{l_r} ,  y_1^{l_r} \bigr)\geq \epsilon_1 >0$,
    which is impossible.
\end{proof}

This lemma leads to the following corollaries.

\begin{cor}\label{diameter of n level partition tends to 0}
    Let $T$ be a forward expansive correspondence on a compact metric space $(X,  d)$ with an expansive constant $\epsilon>0$ and $\cA$ be a finite Borel measurable partition of $X$ with $\mesh \cA < \epsilon$. Then
    $
    	\lim_{n\to +\infty} \mesh \wt{\cA}^n_T =0.
    $
\end{cor}

\begin{proof}
Fix an arbitrary $\delta \in (0 ,\epsilon)$. Since $\epsilon$ is an expansive constant for $T$, $\delta$ is also an expansive constant for $T$. Choose $N \in \N$ such that $\frac{1}{2^N} <\frac{\delta}{2}$. By Lemma~\ref{locally long behavior of forward expansive correspondence}, we can choose $L\in \N$ greater than $N$ with the following property:
for each $n\in \N$ greater than $L$, if two orbits $\vect{x}{1}{n}  ,\, \vect{y}{1}{n} \in \cO_n (T)$ satisfy $d(x_k ,  y_k)< \epsilon$ for all $k\in \oneto{n}$, then $d(x_k ,  y_k)< \frac{\delta}{2}$ holds for all $k\in \oneto{ n-L}$.

Fix an arbitrary $n \in \N$ greater than $N+L$. If two orbits $\vect{x}{1}{\infty} , \, \vect{y}{1}{\infty} \in \cO_\omega (T)$ belong to the same element in the partition $\wt{\cA}^n_T$, then we have $d(x_k ,y_k) < \epsilon$ for all $k \in \oneto{n}$ by $\mesh \cA <\epsilon$. Then it follows from the property of $L$ that $d(x_k ,y_k) <\frac{\delta}{2}$ for all $k\in \oneto{ n-L}$, so
\begin{equation*}
	d_\omega (\vect{x}{1}{\infty} , \vect{y}{1}{\infty}) \leq \sum_{k=1}^{n-L} \frac{1}{2^k} \frac{\delta}{1+\delta} +\sum_{k= n-L+1}^{+\infty} \frac{1}{2^k} <\frac{\delta}{2} +\frac{1}{2^N} \leq \delta.
\end{equation*}
Thus, $\mesh \wt{\cA}^n_T \leq \delta$. Since $\delta$ is chosen arbitrarily, we conclude that $\lim_{n\to +\infty} \mesh \wt{\cA}^n_T =0$.
\end{proof}

\begin{cor}\label{measure-theoretic entropy coincide with that with respect to a partition whose diameter is sufficiently small}
	Let $T$ be a forward expansive correspondence on a compact metric space $X$ with an expansive constant $\epsilon >0$, $\nu \in \MMM(\cO_\omega (T),\sigma)$, and $\cA$ be a finite Borel measurable partition of $X$ with $\mesh \cA <\epsilon$. Then the partition $\wt{\cA}_T^1$ of $\cO_\omega (T)$ is a finite one-sided generator for $\nu$, i.e., if $\udx ,\, \udy \in \cO_\omega (T)$ satisfy $\wt{\cA}_T^1 (\sigma^n (\udx)) =\wt{\cA}_T^1 (\sigma^n (\udy))$ for all $n\in \N$, then $\udx=\udy$. Moreover, we have
	\begin{equation}\label{33333334onjf039wwwwwwwwww}
		h_\nu (\sigma) =h_\nu \bigl(\sigma ,  \wt{\cA}_T^1 \bigr). 
	\end{equation}
\end{cor}

\begin{proof}
	By Corollary~\ref{diameter of n level partition tends to 0} and Lemma~\ref{expansiveness of correspondence with its lift}, we can choose $n\in \N$ such that $\mesh \wt{\cA}^n_T$ is less than some expansive constant for $\sigma$. By \cite[Lemma~3.5.5]{PU10}, we get that the partition $\wt{\cA}^n_T$ is a finite one-sided generator for $\nu$, i.e., if $\udx ,\, \udy \in \cO_\omega (T)$ satisfy $\wt{\cA}_T^n (\sigma^m (\udx)) =\wt{\cA}_T^n (\sigma^m (\udy))$ for all $m\in \N$, then $\udx=\udy$. Recall $\wt{\cA}_T^n = \bigvee_{k=0}^{n-1} \sigma^{-k} \bigl( \wt{\cA}_T^1 \bigr)$, so the equality $\wt{\cA}_T^n (\sigma^m (\udx)) =\wt{\cA}_T^n (\sigma^m (\udy))$ is equivalent to the statement that $\wt{\cA}_T^1 \bigl(\sigma^{m+k} (\udx)\bigr) =\wt{\cA}_T^1 \bigl(\sigma^{m+k} (\udy)\bigr)$ holds for all $k\in \zeroto{n-1}$. Hence, %
	$\wt{\cA}_T^1$ is a finite one-sided generator for $\nu$. Finally, (\ref{33333334onjf039wwwwwwwwww}) follows by \cite[Theorem~2.8.7~(b)]{PU10}.
\end{proof}

The next corollary follows immediately from Corollary~\ref{measure-theoretic entropy coincide with that with respect to a partition whose diameter is sufficiently small} and the Martingale Convergence Theorem (see e.g., \cite[Theorem~2.1.4]{PU10}).

\begin{cor}\label{almost every point is Lebesgue point}
	Let $T$ be a forward expansive correspondence on a compact metric space $(X ,  d)$ with an expansive constant $\epsilon >0$, $\cA$ be a finite measurable partition of $X$ with $\mesh \cA <\epsilon$, $\nu \in \MMM(\cO_\omega (T), \sigma)$, and $f \in \BBB(\cO_\omega (T) , \R)$. Then for $\nu$-almost every $\udx \in \cO_\omega (T)$,
		\begin{equation*}
			\lim_{n\to +\infty} \frac{1}{\nu \bigl(\wt{\cA}^n_T (\udx)\bigr)} \int_{\wt{\cA}^n_T (\udx)} \! f   \diff \nu =f(\udx).
		\end{equation*}
\end{cor}

\subsection{Rokhlin formulas}\label{subsct_Rokhlin formula for measure-theoretic entropy of transition probability kernels and of the corresponding shift maps}
This subsection is devoted to Proposition~\ref{h_nu(sigma) when T is expansive} and Theorem~\ref{t_h_mu(Q) when T is expansive}. We first use the Shannon--McMillan--Breiman theorem and Corollary~\ref{almost every point is Lebesgue point} to establish Proposition~\ref{h_nu(sigma) when T is expansive}, a variant of the Rokhlin formula for measure-theoretic entropy of shift maps. Then we use Proposition~\ref{h_nu(sigma) when T is expansive} to establish Theorem~\ref{t_h_mu(Q) when T is expansive}, our Rokhlin formula for measure-theoretic entropy of transition probability kernels. Finally, in Remark~\ref{rohlin formula2}, we point out that an equivalent form of the classical Rokhlin formula for forward expansive maps (see e.g., \cite[Theorem~2.9.7]{PU10}) follows from Theorem~\ref{t_h_mu(Q) when T is expansive}. 

Let $\mu$ be a probability measure on some measurable space $(Y,  \SBB (Y))$. If there exists a countable measurable set $A \in \SBB (Y)$ such that $\nu (A)=1$, then we set
\begin{equation}\label{H(P_(x_1,x_2,dots))=}
	H(\nu) \= -\sum_{y\in A}   \nu (\{y\}) \log (\nu (\{y\})),
\end{equation}
where we follow the convention that $0 \log 0 =0$.
If $\nu (A) <1$ for all countable measurable set $A \in \SBB(Y)$, then we set $H(\nu) \=+\infty$.

Let $\nu$ be an arbitrary Borel probability measure on $\cO_\omega (T)$. Denote by $\hnu$ the Borel probability measure on $X^\omega$ given by $\hnu (A) \=\nu (A \cap \cO_\omega (T))$ for all $A \in \SBB (X^\omega)$.

\begin{prop}\label{h_nu(sigma) when T is expansive}
    Let $T$ be a forward expansive correspondence on a compact metric space $(X,  d)$ with an expansive constant $\epsilon >0$, $\nu \in \MMM(\cO_\omega (T), \sigma)$, and $\cP$ be a backward conditional transition probability kernel of $\hnu$ from $X^\omega$ to $X$ supported on $\cO_2 (T) \times X^\omega$. Then we have
    \begin{equation*}
    	h_\nu (\sigma) =\int_{\cO_\omega (T)} \! H(\cP_{\udx})   \diff \nu (\udx).
    \end{equation*}
\end{prop}

See Definition~\ref{backward conditional transition probability kernel} for the notion of backward conditional transition probability kernels.

\begin{proof}
     Choose a finite Borel measurable partition $\cA$ of $X$ with $\mesh \cA <\epsilon$. Recall for each $n\in \N$, $\wt{\cA}^n_T = \{ B_1 \times \cdots \times B_n \times X^\omega \cap \cO_\omega (T)  :  B_1 ,\, \dots ,\, B_n \in \cA \}$, and that for each $\udx \in \cO_\omega (T)$, $\wt{\cA}^n_T (\udx)$ is the element in $\wt{\cA}^n_T$ containing $\udx$.

     For each $n\in \N$, set $\wt{\cA}^{1 ,  n+1}_T \= \bigl\{ X\times A \cap \cO_\omega (T)  :  A\in \wt{\cA}^n_T \bigr\} =\sigma ^{-1} \bigl( \wt{\cA}^n_T \bigr)$, and for each $\udx \in \cO_\omega (T)$, denote by $\wt{\cA}^{1,  n+1}_T (\udx)$ the element in $\wt{\cA}^{1 ,  n+1}_T$ containing $\udx$. We can check that $\wt{\cA}^{1 ,  n+1}_T (\udx)= \sigma^{-1} \bigl( \wt{\cA}^n_T (\sigma (\udx)) \bigr)$ holds for all $\udx \in \cO_\omega (T)$, thus $\nu \bigl( \wt{\cA}^{1 ,  n+1}_T (\udx) \bigr) = \bigl( \nu \circ \sigma^{-1} \bigr) \bigl( \wt{\cA}^n_T (\sigma (\udx)) \bigr) =\nu \bigl( \wt{\cA}^n_T (\sigma (\udx)) \bigr)$.

     Note $\wt{\cA}^n_T =\bigvee_{k=0}^{n-1} \sigma^{-k} \bigl( \wt{\cA}^1_T \bigr) \text{ and } \wt{\cA}^{1 ,  n+1}_T =\bigvee_{k=1}^n \sigma^{-k} \bigl( \wt{\cA}^1_T \bigr)$. By Corollary~\ref{measure-theoretic entropy coincide with that with respect to a partition whose diameter is sufficiently small}, $h_\nu (\sigma) =h_\nu \bigl( \sigma,  \wt{\cA}^1_T \bigr)$.

     Applying the Shannon--McMillan--Breiman theorem to $(\cO_\omega (T) ,  \SBB (\cO_\omega (T)) ,  \nu ,  \sigma)$ with the partition $\wt{\cA}^1_T$, we get that for $\nu$-almost every $\udx \in \cO_\omega (T)$, $\nu \bigl( \wt{\cA}^{1 ,  n+1}_T (\udx) \bigr) >0$ holds for all $n\in \N$, the limit $\lim\limits_{n\to +\infty} \frac{\nu ( \wt{\cA}^{n+1}_T (\udx))}{\nu ( \wt{\cA}^{1 ,  n+1}_T (\udx))}$ exists, and we have
     \begin{equation*}
     	h_\nu (\sigma) =h_\nu \bigl(\sigma ,  \wt{\cA}^1_T \bigr)= \int_{\cO_\omega (T)} \! -\log \biggl(   \lim_{n\to +\infty} \frac{\nu \bigl( \wt{\cA}^{n+1}_T (\udx) \bigr)}{\nu \bigl( \wt{\cA}^{1 ,  n+1}_T (\udx) \bigr)}  \biggr)   \diff \nu (\udx).
     \end{equation*}

    Note that $\sigma ' \: X^\omega \to X^\omega$ is the projection map from $X^\omega =X \times X^\omega$ onto $X^\omega$ and $\wt{\cA}^{1 ,  n+1}_T ( \vect{x}{0}{\infty} ) =\sigma^{-1} \bigl( \wt{\cA}^n_T (\vect{x}{1}{\infty}) \bigr)$. Applying (\ref{int_X_1*X_2f(x_1,x_2)dnu(x_1,x_2)=int_X_1(int_X_2f(x_1,x_2)dQ_x_1(x_2))dmu(x_1)}) in Remark~\ref{two equivalent conditions for the conditional transition probability kernel} and writing 
    \begin{equation*}
    	L_1 
    	\=\lim_{n\to +\infty} \frac{\nu \bigl( \wt{\cA}^{n+1}_T ( \vect{x}{0}{\infty} ) \bigr)}{\nu \bigl( \wt{\cA}^{1 ,  n+1}_T ( \vect{x}{0}{\infty} ) \bigr)} 
    	=\lim_{n\to +\infty} \frac{\nu \bigl( \wt{\cA}^{n+1}_T ( \vect{x}{0}{\infty} ) \bigr)}{\nu \bigl( \wt{\cA}^n_T ( \vect{x}{1}{\infty} ) \bigr)}  ,
    \end{equation*}
    we get
    \begin{equation}\label{0923edcwjszoi}
    \begin{aligned}
        h_\nu (\sigma) 
        &= \int_{X^\omega} \! \biggl( \int_{X} \! -\log L_1    \diff \cP_{\vect{x}{1}{\infty}} (x_0) \biggr)   \diff \bigl( \hnu \circ (\sigma ')^{-1} \bigr) (\vect{x}{1}{\infty})\\
        &= \int_{\cO_\omega (T)} \! \biggl( \sum_{x_0 \in T^{-1} (x_1)}  -\cP_{\vect{x}{1}{\infty}} (\{ x_0 \}) \log L_1 \biggr)   \diff \nu (\vect{x}{1}{\infty}).
    \end{aligned}
    \end{equation}

    Since $\sigma (\cO_\omega (T)) =\cO_\omega (T) \cap T(X) \times X^\omega$, we have $\sigma ^{-1} (\cO_\omega (T) \cap T(X) \times X^\omega) =\cO_\omega (T)$, and thus $\nu (\cO_\omega (T) \cap T(X) \times X^\omega) =\nu (\cO_\omega (T)) =1$. This allows us to assume $\vect{x}{1}{\infty} \in \cO_\omega (T) \cap T(X) \times X^\omega$ in the integrand on the right-hand side of (\ref{0923edcwjszoi}). Now we fix an arbitrary orbit $\vect{x}{1}{\infty} \in \cO_\omega (T)$ with $x_1 \in T(X)$ and compute this integrand.
    
    Fix an arbitrary $x_0 \in T^{-1} (x_1)$. Recall that $\cA (x_0)$ refers to the element of $\cA$ containing $x_0$. Since $\cP$ is a backward conditional transition probability kernel of $\hnu$ from $X^\omega$ to $X$ supported on $\cO_2 (T) \times X^\omega$ and since $x_1 \in T(X)$, the measure $\cP_{\vect{x}{1}{\infty}}$ is supported on $T^{-1} (x_1)$. Note that the diameter of $\cA (x_0)$ is less than $\epsilon$, an expansive constant for $T$, by Remark~\ref{preimage of forward expansive correspondence is finite} we get $\cA (x_0)\cap T^{-1} (x_1) =\{x_0\}$. Consequently,
    \begin{equation}\label{0923edcwjszoi1}
    	\cP_{\vect{x}{1}{\infty}} (\cA (x_0)) =\cP_{(x_1 ,  x_2 ,  \dots)} (\{x_0\}).
    \end{equation}
    Also, by the definition of the partitions $\wt{\cA}_T^n$, we have $\wt{\cA}_T^{n+1} (\vect{x}{0}{\infty}) =\cA (x_0) \times \cA (x_1) \times \cdots \times \cA (x_n) \times X^\omega \cap \cO_\omega (T) =\cA (x_0) \times \wt{\cA}_T^n (\vect{x}{1}{\infty}) \cap \cO_\omega (T)$. Applying Corollary~\ref{almost every point is Lebesgue point} to the Borel measurable function that assigns each $\udx \in \cO_\omega (T)$ the value $\cP (\udx ,  \cA (x_0))$, by Definition~\ref{backward conditional transition probability kernel}~(b), we get
    \begin{equation}\label{0923edcwjszoi2}
    \begin{aligned}
         \cP (\vect{x}{1}{\infty} ,  \cA (x_0))
        = \lim_{n\to +\infty} \frac{\int_{\wt{\cA}^n_T (\vect{x}{1}{\infty})} \! \cP (\udx ,  \cA (x_0))   \diff \nu (\udx)}
        {\nu \bigl( \wt{\cA}^n_T (\vect{x}{1}{\infty}) \bigr)} 
        &= \lim_{n\to +\infty} \frac{\nu \bigl( \cA (x_0)\times \wt{\cA}^n_T ( \vect{x}{1}{\infty} ) \cap \cO_\omega (T) \bigr)}{\nu \bigl( \wt{\cA}^n_T ( \vect{x}{1}{\infty} ) \bigr)}  \\
        &= \lim_{n\to +\infty} \frac{\nu \bigl( \wt{\cA}_T^{n+1} ( \vect{x}{0}{\infty} ) \bigr)}{\nu \bigl( \wt{\cA}^n_T ( \vect{x}{1}{\infty} ) \bigr)}.
    \end{aligned}
    \end{equation}
    By (\ref{0923edcwjszoi}), (\ref{0923edcwjszoi1}), and~(\ref{0923edcwjszoi2}), we have
    $h_\nu (\sigma) = - \int_{\cO_\omega (T)} \! \sum_{x_0 \in T^{-1} (x_1)}  \cP_{ \vect{x}{1}{\infty} } (\{x_0\}) \log \bigl( \cP_{ \vect{x}{1}{\infty} } (\{x_0\}) \bigr)    \diff \nu ( \vect{x}{1}{\infty} )$.

    Therefore, by (\ref{H(P_(x_1,x_2,dots))=}) we conclude $h_\nu (\sigma) =\int_{\cO_\omega (T)} \! H(\cP_{\udx})   \diff \nu (\udx)$ as we want.
\end{proof}

\begin{theorem}[Rokhlin formula]\label{t_h_mu(Q) when T is expansive}
	Let $T$ be a forward expansive correspondence on a compact metric space $(X,  d)$, $\cQ \in \tpk( X ; T )$, and  $\mu \in \MMM(X, \cQ)$. If $\cR$ is a backward conditional transition probability kernel of $\mu \cQ^{\zeroto{1}}$ from $X$ to $X$ supported on $\cO_2 (T)$, then we have
	$h_\mu (\cQ) =\int_X \! H(\cR_x) \diff \mu (x)$.
\end{theorem}

We need the following lemma in the proof of Theorem~\ref{t_h_mu(Q) when T is expansive}.

\begin{lemma}\label{backward conditional transition probability kernel of muQ^N_0}
    Let $T$ be a correspondence on a compact metric space $(X,  d)$, $\cQ \in \tpk(  X ; T)$, and $\mu \in \MMM(X, \cQ)$. %
    If $\cR$ is a backward conditional transition probability kernel of $\mu \cQ^{\zeroto{1}}$ from $X$ to $X$ supported on $\cO_2 (T)$, then $\wt{\cR} \in \tpk( X^\omega, X)$ given by
    \begin{equation}\label{tcR=}
        \wt{\cR} (\vect{x}{1}{\infty} ,  B) \= \cR (x_1 ,  B) \quad \text{for all } \vect{x}{1}{\infty} \in X^\omega \text{ and } B \in \SBB (X)
    \end{equation}
    is a backward conditional transition probability kernel of $\mu \cQ^\omega$ from $X^\omega$ to $X$ supported on $\cO_2 (T)\times X^\omega$.%
\end{lemma}

\begin{proof}
    To verify that $\wt{\cR}$ is a backward conditional transition probability kernel of $\mu \cQ^\omega$ from $\cO_\omega (T)$ to $X$ supported on $\cO_2 (T) \times X^\omega$, we should check  properties~(a) and~(b) in Definition~\ref{backward conditional transition probability kernel}.

    First, because $\cR$ is a backward conditional transition probability kernel of $\mu \cQ^{\zeroto{1}}$ from $X$ to $X$ supported on $\cO_2 (T)$, by Definition~\ref{backward conditional transition probability kernel}~(a), $\cR_x$ is supported on $T^{-1} (x)$ for all $x\in T(X)$. For each $\vect{x}{1}{\infty} \in T(x) \times X^\omega$, since $\cR_{x_1}$ is supported on $T^{-1} (x_1)$, we have $\wt{\cR} \bigl( \vect{x}{1}{\infty} ,  T^{-1} (x_1) \bigr) =\cR \bigl(x_1 , T^{-1} (x_1)\bigr) =1$. Thus, Definition~\ref{backward conditional transition probability kernel}~(a) holds for $\wt{\cR}$ as $T(X) \times X^\omega =\sigma ' (\cO_2 (T) \times X^\omega)$.
    
    By Remark~\ref{nu=muQ^[1]}, $\mu \cQ^{[1]} =\bigl(\mu \cR^{[1]}\bigr) \circ \gamma_2^{-1}$. So Lemma~\ref{A9} indicates $\mu \cQ^{[n]} =\bigl(\mu \cR^{[n]}\bigr) \circ \gamma_{n+1}^{-1}$ for all $n\in \N$. Thus, by (\ref{muQ^N(A*X^infty)=muQ^[n](A)}) and~(\ref{muQ^{n}(A_0**A_n)}) in Lemma~\ref{l_Qun}, for all $n\in \N$ and $A_0 ,\, A_1 ,\, \dots ,\, A_n \in \SBB (X)$, we have
    \begin{alignat*}{4}
    	\int_{A_1 \times \cdots \times A_n \times X^\omega} \! \wt{\cR} (\udx ,  A_0)   \diff (\mu \cQ^\omega) (\udx) 
    	&=\int_{A_1 \times \cdots \times A_n} \! \cR (x_1 ,  A_0)   \diff \bigl( \mu \cQ^{\zeroto{n-1}} \bigr) (x_1 , \dots ,x_n)\\
    	& =\int_{A_n \times \cdots \times A_1} \! \cR (x_1 ,  A_0)   \diff \bigl( \mu \cR^{\zeroto{n-1}} \bigr) (x_n , \dots ,x_1) \\
    	&=\bigl(\mu \cR^{[n]}\bigr) (A_n \times \cdots \times A_0)
    	=\bigl( \mu \cQ^{[n]} \bigr) (A_0 \times \cdots \times A_n) \\
    	&=(\mu \cQ^\omega) (A_0 \times \cdots \times A_n \times X^\omega).
    \end{alignat*}
    This is equivalent to property~(b) in Definition~\ref{backward conditional transition probability kernel} for $\wt{\cR}$ by Dynkin's $\pi$-$\lambda$ theorem (see the equivalence between  properties~(b) and~(b2) in Remark~\ref{two equivalent conditions for the conditional transition probability kernel}). Therefore, $\wt{\cR}$ is a backward conditional transition probability kernel of $\mu \cQ^\omega$ from $X^\omega$ to $X$ supported on $\cO_2 (T)\times X^\omega$.
\end{proof}

\begin{proof}[Proof of Theorem~\ref{t_h_mu(Q) when T is expansive}]
    First, by Lemma~\ref{backward conditional transition probability kernel of muQ^N_0}, the transition probability kernel $\wt{\cR}$ given by (\ref{tcR=})
    is a backward conditional transition probability kernel of $\mu \cQ^\omega$ from $X^\omega$ to $X$ supported on $\cO_2 (T)\times X^\omega$.

    Thus, by Proposition~\ref{h_nu(sigma) when T is expansive} and Theorem~\ref{measure-theoretic entropy coincide with its lift},
    \begin{equation*}
    \begin{aligned}
        h_\mu (\cQ) 
        &=h_{\mu \cQ^\omega |_T} (\sigma) =\int_{\cO_\omega (T)} \! H \bigl( \wt{\cR}_{ \vect{x}{1}{\infty} } \bigr)   \diff (\mu \cQ^\omega |_T) ( \vect{x}{1}{\infty} )\\
        &= \int_{\cO_\omega (T)} \! H(\cR_{x_1})   \diff (\mu \cQ^\omega |_T) ( \vect{x}{1}{\infty} )
        = \int_X \! H(\cR_{x_1})   \diff \mu (x_1)  ,
    \end{aligned}
    \end{equation*}
    where the last equality follows from taking $n=0$ in (\ref{muQ^N(A*X^infty)=muQ^[n](A)}). Therefore, $h_\mu (\cQ) =\int_X H(\cR_x)   \diff \mu (x)$.
\end{proof}

\begin{rem}\label{rohlin formula2}
	If the forward expansive correspondence $T$ in Theorem~\ref{t_h_mu(Q) when T is expansive} is induced by a single-valued forward expansive map $f$, i.e., $T=\cC_f$, we can conclude the following statement, which is equivalent to the Rokhlin formula (see e.g., \cite[Theorem~2.9.7]{PU10}):
	
	\smallskip
	
	Let $X$ be a compact metric space, $f\: X\to X$ be a forward expansive continuous map, and %
	$\mu \in \MMM(X,f)$. %
	If $\cR \in \tpk( X )$ satisfies $\mu \bigl( A\cap f^{-1} (B) \bigr) =\int_B \! \cR (x ,  A)   \diff \mu (x)$ for all $A ,\, B \in \SBB (X)$, then
	$h_\mu (f)= \int_X \! H(\cR_x)   \diff \mu (x)$.
\end{rem}

\subsection{Proof of Theorem~\ref{t_HVP}}\label{subsct_HVP}

\begin{lemma}\label{the orbit space is of full measure}
	Let $(X,  d)$ be a compact metric space, $\cQ \in \tpk(X)$ be supported by a correspondence $T$ on $X$, and $\mu\in \PPP(X)$. Then the measure $\mu \cQ^{\zeroto{n-1}}$ is supported on $\cO_n (T)$ for every $n\in \N$, and the measure $\mu \cQ^\omega$ is supported on $\cO_\omega (T)$.
\end{lemma}

\begin{proof}
	First we show that for each $n\in \N$ and each point $x\in X$, we have
	\begin{equation}\label{Q^[n-1](x,O_n(T)capx*X^n-1)=1}
		\cQ^{\zeroto{n-1}} (x,  \cO_n (T))=1.
	\end{equation}
	
	If $n=1$, (\ref{Q^[n-1](x,O_n(T)capx*X^n-1)=1}) holds because $\cQ^{\zeroto{0}} (x,  \cO_1 (T))= \wh{\operatorname{id}_X} (x,  X)=1$.
	
	Now suppose that (\ref{Q^[n-1](x,O_n(T)capx*X^n-1)=1}) holds for some $n\in \N$. By (\ref{cQ^n+1(x,A_n+1)=}), we have
	\begin{equation*}
		\begin{aligned}
			\cQ^{\zeroto{n}} (x,  \cO_{n+1} (T))
			&= \int_{X^n} \! \cQ (x_n ,  \pi_{n+1} ( \vect{x}{1}{n}  ;  \cO_{n+1} (T)))   \diff \cQ_x^{\zeroto{n-1}} ( \vect{x}{1}{n}  )\\
			&= \int_{\cO_n (T)} \! \cQ (x_n ,  T(x_n))   \diff \cQ_x^{\zeroto{n-1}} ( \vect{x}{1}{n}  )
			= \int_{\cO_n (T)} \diff \cQ_x^{\zeroto{n-1}} ( \vect{x}{1}{n}  )
			= \cQ^{\zeroto{n-1}} (x,  \cO_n (T))
			=1,
		\end{aligned}
	\end{equation*}
	where the third equality follows from $\cQ (x_n ,  T(x_n))=1$ because $\cQ$  is supported by $T$.
	Hence, by induction, we get that (\ref{Q^[n-1](x,O_n(T)capx*X^n-1)=1}) holds for all $n\in \N$.
	
	For each $n\in \N$, by (\ref{Q^[n-1](x,O_n(T)capx*X^n-1)=1}) and Definition~\ref{transition probability kernel act on measures}, we have
	\begin{equation*}
		\bigl(\mu \cQ^{\zeroto{n-1}}\bigr) (\cO_n (T))
		= \int_X \! \cQ^{\zeroto{n-1}} (x,  \cO_n (T))   \diff \mu (x)
		= \int_X \diff \mu 
		=1.
	\end{equation*}
	
	By (\ref{muQ^N(A*X^infty)=muQ^[n](A)}), we get
	\begin{equation*}
		\begin{aligned}
			(\mu \cQ^\omega) (\cO_\omega (T))
			&= (\mu \cQ^\omega) \biggl(\bigcap_{n=1}^{+\infty} \cO_n (T)\times X^\omega\biggr)
			= \lim_{n\to +\infty} (\mu \cQ^\omega) \bigl(\cO_n (T)\times X^\omega\bigr)\\
			&= \lim_{n\to +\infty} \bigl(\mu \cQ^{\zeroto{n-1}}\bigr) (\cO_n (T))
			= \lim_{n\to +\infty} 1
			=1.
		\end{aligned}
	\end{equation*}
	Therefore, $\mu \cQ^{\zeroto{n-1}}$ is supported on $\cO_n (T)$ for every $n\in \N$, and $\mu \cQ^\omega$ is supported on $\cO_\omega (T)$.
\end{proof}

\begin{rem}\label{muQ|_T}
	Suppose that $\cQ$ is supported by $T$. Denote by $\mu \cQ^\omega |_T$ the restricted measure of $\mu \cQ^\omega$ on $\cO_\omega (T)$. Since $(\mu \cQ^\omega) (\cO_\omega (T))=1$ and $\cO_\omega (T)$ is forward-invariant under $\sigma ' \: X^\omega \to X^\omega$, the measure-preserving system $(\cO_\omega (T),  \SBB (\cO_\omega (T)),  \mu \cQ^\omega |_T ,  \sigma)$ is isomorphic to $(X^\omega ,  \SBB (X^\omega) ,  \mu \cQ^\omega ,  \sigma' )$, so their entropies are equal. Thus, we can rewrite Theorem~\ref{measure-theoretic entropy coincide with its lift} as
	\begin{equation}\label{h_mu(Q)=h_muQ^N|_T(sigma)}
		h_\mu (\cQ)= h_{\mu \cQ^\omega |_T} (\sigma).
	\end{equation}
\end{rem}

Recall the projection maps $\tpi_1$, $\tpi_2$, and $\tpi_{12}$ given in (\ref{tpi=}). Let $T$ be a correspondence on a compact metric space $(X,  d)$. In Proposition~\ref{decompose a two-dimentional measure}, if $X_1 =X_2 =X$ and $M= \cO_2 (T)$, then for $\mu \in \PPP(X)$ and $\cQ \in \tpk( X )$ from Proposition~\ref{decompose a two-dimentional measure}, $\cQ$ is supported by $T$, $\mu \cQ^{\zeroto{1}} =\nu$, and $\mu =\nu \circ \tpi_1^{-1}$. Hence, we get the following:

\begin{prop}\label{decompose a two-dimentional measure with support}
	Let $T$ be a correspondence on a compact metric space $(X,  d)$ and $\nu \in \PPP \bigl( X^2 \bigr)$ be supported on $\cO_2 (T)$. Then there exists $\cQ\in \tpk( X ; T )$ such that $\bigl( \nu \circ \tpi_1^{-1} \bigr) \cQ^{\zeroto{1}} =\nu$.
	
	Moreover, if $\cQ, \, \cQ ' \in \tpk( X )$ satisfy $\nu =\bigl(\nu \circ \tpi_1^{-1}\bigr) \cQ^{\zeroto{1}} =\bigl(\nu \circ \tpi_1^{-1}\bigr) \cQ '^{\zeroto{1}}$, then for each $A\in \SBB (X)$, $\cQ (x,  A)= \cQ '(x,  A)$ holds for $\bigl(\nu \circ \tpi_1^{-1}\bigr)$-almost every $x\in X$.
\end{prop}

\begin{lemma}\label{387efwgyuc83w2}
	Let $T$ be a correspondence on a compact metric space $X$ and $\nu \in \MMM(\cO_\omega (T), \sigma)$. Then there exists $\cQ\in \tpk( X ; T)$ such that $\nu \circ \tpi_1^{-1} \in \MMM (X ,\cQ)$ and $\bigl(\nu \circ \tpi_1^{-1}\bigr) \cQ^{[1]} =\nu \circ \tpi_{12}$.
\end{lemma}

\begin{proof}
	The measure $\nu \circ \tpi_{12}^{-1}$ is a Borel probability measure on $X^2$ supported on $\cO_2 (T)$ because the image of $\tpi_{12}$ lies in $\cO_2 (T)$. By Proposition~\ref{decompose a two-dimentional measure with support}, we can choose $\cQ \in \tpk( X ; T )$ such that $\nu \circ \tpi_{12}^{-1} =\bigl(\bigl(\nu \circ \tpi_{12}^{-1} \bigr) \circ \tpi_1^{-1} \bigr) \cQ^{\zeroto{1}} =\bigl(\nu \circ \tpi_1^{-1} \bigr) \cQ^{[1]}$. Since $\nu$ is $\sigma$-invariant and $\bigl( \bigl(\nu \circ \tpi_1^{-1} \bigr) \cQ^{\zeroto{1}} \bigr) \circ \tpi_2^{-1} = \bigl(\nu \circ \tpi_1^{-1} \bigr) \cQ$ from (\ref{muQ^[1]circhpi_2^-1=muQ}), we have
	\begin{equation*}
			\bigl(\nu \circ \tpi_1^{-1} \bigr) \cQ 
			=\bigl(\bigl(\nu \circ \tpi_1^{-1} \bigr) \cQ^{\zeroto{1}} \bigr) \circ \tpi_2^{-1}
			=\bigl(\nu \circ \tpi_{12}^{-1} \bigr) \circ \tpi_2^{-1} 
			=\nu \circ \tpi_2^{-1}
			=\bigl(\nu \circ \sigma^{-1} \bigr) \circ \tpi_1^{-1}
			= \nu \circ \tpi_1^{-1}.
	\end{equation*}
	Therefore, $\nu \circ \tpi_1^{-1} =\mu$ is $\cQ$-invariant.
\end{proof}

\begin{proof}[Proof of Theorem~\ref{t_HVP}]

	(i) 
	By the Bogolyubov–Krylov theorem (see e.g., \cite[Theorem~3.1.8]{PU10}), we can choose $\nu \in \MMM(\cO_\omega (T),\sigma)$. Then statement~(i) follows from Lemma~\ref{387efwgyuc83w2}.
	
	\smallskip
	
	(ii) Fix arbitrary $\cQ \in \tpk( X ; T )$ and $\mu \in \MMM( X, \cQ)$.

	We apply the (classical) Variational Principle to the dynamical system $(\cO_\omega (T),  \sigma)$:
	\begin{equation*}
		P\bigl( \sigma ,  \tphi \bigr)= \sup_{\nu \in \MMM(\cO_\omega (T),  \sigma)} \bigg\{  h_\nu (\sigma)+ \int_{\cO_\omega (T)} \! \tphi   \diff \nu \bigg\}.
	\end{equation*}
	
	Since $\mu \cQ^\omega |_T   \in \MMM(\cO_\omega (T),  \sigma)$ for arbitrary $\mu$ and $\cQ$, we get
	\begin{equation*}
		P \bigl( \sigma ,  \tphi \bigr)\geq \sup_{\cQ ,  \mu} \biggl\{ h_{\mu \cQ^\omega |_T} (\sigma)+ \int_{\cO_\omega (T)} \! \tphi   \diff (\mu \cQ^\omega |_T)\biggr\}.
	\end{equation*}
	
	Recall $P(T ,  \phi)= P\bigl( \sigma ,  \tphi \bigr)$ from Theorem~\ref{topological pressure coincide with the lift} and $h_\mu (\cQ)= h_{\mu \cQ^\omega |_T} (\sigma)$ from (\ref{h_mu(Q)=h_muQ^N|_T(sigma)}). We can rewrite the inequality above as
	$P(T,  \phi)\geq \sup_{\cQ ,  \mu} \bigl\{ h_\mu (\cQ)+ \int_{\cO_\omega (T)} \! \tphi   \diff (\mu \cQ^\omega |_T)\bigr\}$.

	Since $\cQ_x (T(x))=1$ for all $x \in X$ and since $\mu \cQ^\omega (\cO_\omega (T))=1$ from Lemma~\ref{the orbit space is of full measure}, (\ref{int_O_+infty(T)tphidmuQ^N|_T=intphidmu}) can be written as
	$\int_{\cO_\omega (T)} \! \tphi   \diff (\mu \cQ^\omega) =\int_X \! \int_{T(x_1)} \! \phi (x_1 ,  x_2)   \diff \cQ_{x_1} (x_2)   \diff \mu (x_1)$.

	Therefore, we get
	$P(T,  \phi)\geq \sup_{\cQ ,  \mu} \bigl\{ h_\mu (\cQ)+ \int_X \! \int_{T(x_1)} \! \phi (x_1 ,  x_2)   \diff \cQ_{x_1} (x_2)   \diff \mu (x_1) \bigr\}$.
\end{proof}

The following result is used in Subsection~\ref{subsct_Lee--Lyubich--Markorov--Mukherjee anti-holomorphic correspondences} and its proof uses similar techniques in this subsection.

\begin{prop}\label{aiubcw092d9ua3wwc}
	Let $T$ be a correspondences on a compact metric space $X$, $Y\in \cF (X)$ such that $T|_Y$ is a correspondence on $Y$, and $\phi \in C(\cO_2 (X) , \R)$. Assume that for each $\mu \in \PPP(X)$ with the property that there exists $\cQ \in \tpk( X ; T)$ such that $\mu \in \MMM(X,  \cQ)$, we have $\mu (Y)=1$. Then $P(T , \phi) =P(T|_Y , \phi)$.
\end{prop}

\begin{proof}
	Denote by $\sigma_X$ the shift map on $\cO_\omega (T)$, by $\sigma_Y \= \sigma_X |_{Y^\omega \cap \cO_\omega (X)}$ the shift map on $\cO_\omega (T|_Y) =Y^\omega \cap \cO_\omega (X)$, by $\tphi_X \in C(\cO_\omega (T) ,\R)$ the function given by $\tphi_X ( \vect{x}{1}{\infty} ) \= \phi (x_1 ,x_2)$ for all $\vect{x}{1}{\infty}  \in \cO_\omega (T)$, and by $\tphi_Y \in C(\cO_\omega (T|_Y) ,\R)$ the function given by $\tphi_Y ( \vect{x}{1}{\infty} ) \= \phi (x_1 ,x_2)$ for all $\vect{x}{1}{\infty}  \in \cO_\omega (T|_Y)$. By Theorem~\ref{topological pressure coincide with the lift}, we have $P(T ,\phi) =P \bigl( \sigma_X ,\tphi_X\bigr)$ and $P(T|_Y , \phi) =P\bigl( \sigma_Y ,\tphi_Y\bigr)$.
	
	Fix an arbitrary $\nu \in \MMM(\cO_\omega (T), \sigma_X)$. By Lemma~\ref{387efwgyuc83w2}, the measure $\mu$ on $X$ given by $\mu (A) \= \nu (A \times X^\omega \cap \cO_\omega (T))$ for all $A \in \SBB (X)$ is $\cQ$-invariant for some $\cQ \in \tpk( X )$, so $\mu (Y) =1$, i.e., $\nu (Y \times X^\omega \cap \cO_\omega (T)) =1$. Since $\nu$ is $\sigma_X$-invariant, $\nu (X^n \times Y \times X^\omega \cap \cO_\omega (T)) =\nu (\sigma_X^{-n} (Y \times X^\omega \cap \cO_\omega (T))) =\nu (Y \times X^\omega \cap \cO_\omega (T)) =1$ holds for all $ n\in \N_0$. Consequently, we have $\nu (\cO_\omega (T|_Y)) =\nu (Y^\omega \cap \cO_\omega (T)) =\nu \bigl(\bigcap_{n=0}^{+\infty} (X^n \times Y \times X^\omega \cap \cO_\omega (T)) \bigr) =1$. 
	
	We have verified $\nu (\cO_\omega (T|_Y)) =1$ for every $\nu \in \MMM( \cO_\omega (T), \sigma_X )$. Since $\tphi_Y =\tphi_X |_{\cO_\omega (T|_Y)}$, applying the classical Variational Principle for $\sigma_X$ with potential $\tphi_X$ and $\sigma_Y$ with potential $\tphi_Y$, we get $P \bigl( \sigma_X ,\tphi_X\bigr) =P\bigl( \sigma_Y ,\tphi_Y\bigr)$. Recall $P(T ,\phi) =P \bigl( \sigma_X ,\tphi_X\bigr)$ and $P(T|_Y , \phi) =P\bigl( \sigma_Y ,\tphi_Y\bigr)$, so we conclude $P(T , \phi) =P(T|_Y , \phi)$.
\end{proof}

\subsection{Proof of Theorem~\ref{t_VP_forward_exp}}\label{subsct_proof of Theorem}

For a correspondence $T$ on a compact metric space $(X ,  d)$, recall that $\tpi_{12} \: \cO_\omega (T) \to X^2$ is the projection given by $\tpi_{12} ( \vect{x}{1}{\infty} )= (x_1 ,  x_2)$. For each $\nu \in \PPP( \cO_\omega (T) )$, set $\nu_{12} \= \nu \circ \tpi_{12}^{-1} \in \PPP \bigl(X^2\bigr)$. Note that $\nu_{12}$ is supported on $\cO_2 (T)$ because the image of $\tpi_{12}$ lies in $\cO_2 (T)$. The following proposition is useful to prove Theorem~\ref{t_VP_forward_exp}.

\begin{prop}\label{measure-theoretic entropy less than the corresponding Markov measure-theoretic entropy}
    Let $T$ be a forward expansive correspondence on a compact metric space $(X,  d)$ with an expansive constant $\epsilon >0$ and $\nu \in \MMM(\cO_\omega (T),\sigma)$. If $\cQ \in \tpk( X ; T )$ and $\mu \in \MMM(X,\cQ)$ satisfy $\nu_{12} =\mu \cQ^{\zeroto{1}}$, %
    then
    $
    	h_\nu (\sigma) \leq h_\mu (\cQ).
    $
\end{prop}

\begin{proof}

    Recall that $\hnu$ is the measure on $X^\omega$ given by $\hnu (A) = \nu (A \cap \cO_\omega (T))$ for all $A \in \SBB (X^\omega)$. Let $\cP$ be a backward conditional transition probability kernel of $\hnu$ from $X^\omega$ to $X$ supported on $\cO_2 (T) \times X^\omega$ and $\cS$ be a forward conditional transition probability kernel of $\hnu$ from $X$ to $X^\omega$ supported on $\cO_\omega (T)$. Proposition~\ref{h_nu(sigma) when T is expansive} indicates
    $h_\nu (\sigma) =\int_{\cO_\omega (T)} \! H(\cP_{\udx})   \diff \nu (\udx)$.

    By Lemma~\ref{l_muQ^[1]circhpi_1^-1=mu}, $\nu_{12} =\mu \cQ^{\zeroto{1}}$ implies $\mu =\nu_{12} \circ \tpi_1^{-1} = \bigl(\nu \circ \tpi_{12}^{-1}\bigr) \circ \tpi_1^{-1} =\nu \circ \tpi_1^{-1}$. Applying (\ref{int_X_1*X_2f(x_1,x_2)dnu(x_1,x_2)=int_X_1(int_X_2f(x_1,x_2)dQ_x_1(x_2))dmu(x_1)}) in Remark~\ref{two equivalent conditions for the conditional transition probability kernel}~(b3) for $\cS$, we get
    \begin{equation}\label{h_nu(sigma)=int_X(int_O_+infty(T)H(P_(x_1,x_2,dots))dS_x_1(x_2,x_3,dots))dmu(x_1)}
        h_\nu (\sigma) =\int_X \! \biggl( \int_{\cO_\omega (T)} \! H(\cP_{ \vect{x}{1}{\infty} })   \diff \cS_{x_1} ( \vect{x}{2}{\infty} ) \biggr)   \diff \mu (x_1).
    \end{equation}

    Recall from Remark~\ref{preimage of forward expansive correspondence is finite} that $T^{-1} (x_1)$, on which the measures $\cP_{ \vect{x}{1}{\infty} }$ are supported, is a finite set. Since the map $x\mapsto x \log x$ is a convex function for $x\in [0,1]$, by (\ref{H(P_(x_1,x_2,dots))=}) and Jensen's inequality, we have for each $x_1 \in X$,
    \begin{equation}\label{int_O_+inftyTH(P(x_1,x_2,dots),x_0)logP((x_1,x_2,dots),x_0)dS(x_2,x_3,dots)leqslantsum_x_0inT^-1x_1-R(x_1,x_0)logR(x_1,x_0)}
    \begin{aligned}
         \int_{\cO_\omega (T)} \! H(\cP_{ \vect{x}{1}{\infty} })   \diff \cS_{x_1} ( \vect{x}{2}{\infty} )
        & = - \sum_{x_0 \in T^{-1} (x_1)}  \int_{\cO_\omega (T)} \! \cP ( \vect{x}{1}{\infty}  ,  \{ x_0 \}) \log (\cP ( \vect{x}{1}{\infty} ,  \{ x_0 \}))   \diff \cS_{x_1} ( \vect{x}{2}{\infty} )\\
        & \leq - \sum_{x_0 \in T^{-1} (x_1)}  \cR (x_1 ,  \{ x_0 \}) \log (\cR (x_1 ,  \{ x_0 \})),
    \end{aligned}
    \end{equation}
    where $\cR (x_1 ,  \{ x_0 \}) \= \int_{\cO_\omega (T)} \! \cP ( \vect{x}{1}{\infty} ,  \{ x_0 \})   \diff \cS_{x_1} ( \vect{x}{2}{\infty} )$ for all $(x_0 ,  x_1) \in \cO_2 (T)$.
    Moreover, for each $x_1 \in X$ and $A \in \SBB (X)$, define
    \begin{equation}\label{R(x_1,A)=}
        \cR (x_1 ,  A) \= \int_{\cO_\omega (T)} \! \cP ( \vect{x}{1}{\infty} ,  A)   \diff \cS_{x_1} ( \vect{x}{2}{\infty} ).
    \end{equation}
    We can check that $\cR \in \tpk( X )$. Now we verify that $\cR$ is a backward conditional transition probability kernel of $\mu \cQ^{\zeroto{1}}$ from $X$ to $X$ supported on $\cO_2 (T)$.
    
    First, recall that $\cP$ is a backward conditional transition probability kernel of $\hnu$ from $X^\omega$ to $X$ supported on $\cO_2 (T) \times X^\omega$. By Definition~\ref{backward conditional transition probability kernel}~(a), we have $\cP \bigl( \vect{x}{1}{\infty} ,  T^{-1} (x_1) \bigr)=1$ for all $\vect{x}{1}{\infty} \in T(X) \times X^\omega$. By (\ref{R(x_1,A)=}), we have $\cR \bigl( x_1 ,  T^{-1} (x_1) \bigr)=1$ for all $x_1 \in T(X)$.

    Second, applying (\ref{int_X_1*X_2f(x_1,x_2)dnu(x_1,x_2)=int_X_1(int_X_2f(x_1,x_2)dQ_x_1(x_2))dmu(x_1)}) for $\cS$, we have
    \begin{equation*}
        \int_B \! \cR (x_1 ,  A)   \diff \mu (x_1) 
        = \int_B \! \biggl( \int_{\cO_\omega (T)} \! \cP ( \vect{x}{1}{\infty} ,  A)   \diff \cS_{x_1} (  \vect{x}{2}{\infty} )
        \biggr)   \diff \mu (x_1)
        = \int_{B\times X^\omega \cap \cO_\omega (T)} \! \cP ( \vect{x}{1}{\infty} ,  A)   \diff \nu ( \vect{x}{1}{\infty} ).
    \end{equation*}

    Recall that $\cP$ is a backward conditional transition probability kernel of $\hnu$ from $X^\omega$ to $X$ supported on $\cO_2 (T) \times X^\omega$. Definition~\ref{backward conditional transition probability kernel}~(b) implies
    \begin{equation*}
    	\nu (A\times B \times X^\omega \cap \cO_\omega (T)) =\int_{B\times X^\omega \cap \cO_\omega (T)} \! \cP ( \vect{x}{1}{\infty} ,  A)   \diff \nu ( \vect{x}{1}{\infty} ).
    \end{equation*}

    Hence, $\int_B \! \cR (x_1 ,  A)   \diff \mu (x_1) =\nu (A\times B \times X^\omega \cap \cO_\omega (T)) =\nu_{12} (A\times B)= \bigl(\mu \cQ^{\zeroto{1}}\bigr) (A\times B)$. Applying Dynkin's $\pi$-$\lambda$ theorem (see the equivalence between properties~(b) and~(b1) in Remark~\ref{two equivalent conditions for the conditional transition probability kernel}), we conclude that $\cR$ is a backward conditional transition probability kernel of $\mu \cQ^{\zeroto{1}}$ from $X$ to $X$ supported on $\cO_2 (T)$ by Definition~\ref{backward conditional transition probability kernel}. Thus, Theorem~\ref{t_h_mu(Q) when T is expansive} indicates $h_\mu (\cQ) =\int_X \! H(\cR_x)   \diff \mu (x)$. By (\ref{int_O_+inftyTH(P(x_1,x_2,dots),x_0)logP((x_1,x_2,dots),x_0)dS(x_2,x_3,dots)leqslantsum_x_0inT^-1x_1-R(x_1,x_0)logR(x_1,x_0)}), we get
    \begin{equation*}
    \begin{aligned}
        h_\mu (\cQ) 
        =\int_X \! \sum_{x_0 \in T^{-1} (x_1)}   -\cR (x_1 ,  \{ x_0 \}) \log (\cR (x_1 ,  \{ x_0 \}))   \diff \mu (x_1)
        \geq \int_X \!  \int_{\cO_\omega (T)} \! H\bigl( \cP_{ \vect{x}{1}{\infty} } \bigr)   \diff \cS_{x_1} ( \vect{x}{2}{\infty} )    \diff \mu (x_1).
    \end{aligned}
    \end{equation*}

    Therefore, by (\ref{h_nu(sigma)=int_X(int_O_+infty(T)H(P_(x_1,x_2,dots))dS_x_1(x_2,x_3,dots))dmu(x_1)}) we conclude $h_\nu (\sigma) \leq h_\mu (\cQ)$.
\end{proof}

With all the preparations in previous subsections, we are now ready to prove Theorem~\ref{t_VP_forward_exp}.

Recall that a pair $(\mu ,\, \cQ)$ consisting of $\cQ \in \tpk( X ; T)$ and $\mu \in \MMM(X,\cQ)$ is called an equilibrium state for the correspondence $T$ and potential function $\phi$ if it satisfies (\ref{P(T,phi)=sup_Q,mu(h_mu(Q)+int_Xphidmu)}).

\begin{proof}[Proof of Theorem~\ref{t_VP_forward_exp}]
    By Proposition~\ref{expansiveness of correspondence with its lift}, the forward expansiveness of $T$ implies that the shift map $\sigma \: \cO_\omega (T) \to \cO_\omega (T)$ is forward expansive. By \cite[Theorem~3.5.6]{PU10}, we get $\nu \in \MMM(\cO_\omega (T),\sigma)$ with
    \begin{equation*}
    	P\bigl( \sigma ,  \tphi \bigr) =h_\nu (\sigma)+ \int_{\cO_\omega (T)} \! \tphi   \diff \nu.
    \end{equation*}

    We choose $\cQ\in \tpk( X ; T )$ and $\mu \in \MMM(X,\cQ)$ such that $\nu_{12} =\mu \cQ^{\zeroto{1}}$ (the existence of this choice is ensured by Proposition~\ref{decompose a two-dimentional measure with support}). Then Proposition~\ref{measure-theoretic entropy less than the corresponding Markov measure-theoretic entropy} indicates that $h_\nu (\sigma) \leq h_\mu (\cQ)$.

    By Theorem~\ref{topological pressure coincide with the lift} and Lemma~\ref{l:int_O_+infty(T)tphidmuQ^N|_T=intphidmu}, we have
    \begin{equation*}
    \begin{aligned}
        P(T ,  \phi)
        &= P \bigl( \sigma ,  \tphi \bigr) 
        =h_\nu (\sigma)+ \int_{\cO_\omega (T)} \! \tphi   \diff \nu
        \leq h_\mu (\cQ)+ \int_{\cO_2 (T)} \! \phi (x_1 ,  x_2)   \diff \nu_{12}\\
        &=h_\mu (\cQ)+ \int_{\cO_2 (T)} \! \phi (x_1 ,  x_2)   \diff \bigl(\mu \cQ^{\zeroto{1}}\bigr) 
        =h_\mu (\cQ)+ \int_X \! \int_{T(x_1)} \! \phi (x_1 ,  x_2)   \diff \cQ_{x_1} (x_2)   \diff \mu (x_1).
    \end{aligned}
    \end{equation*}

    By Theorem~\ref{t_HVP}, we have $P(T ,  \phi) \geq h_\mu (\cQ) +\int_X \! \int_{T(x_1)} \! \phi (x_1 ,  x_2)   \diff \cQ_{x_1} (x_2)   \diff \mu (x_1)$. Thus, we have $P(T ,  \phi) =h_\mu (\cQ) +\int_X \! \int_{T(x_1)} \! \phi (x_1 ,  x_2)   \diff \cQ_{x_1} (x_2)   \diff \mu (x_1)$. Together with Theorem~\ref{t_HVP}, we manage to prove Theorem~\ref{t_VP_forward_exp} except for $P(T ,  \phi) \in \R$. Now we show it.

    First, Remark~\ref{P(T,phi)>-infty} indicates that $P(T ,  \phi) >-\infty$.
	Recall from Remark~\ref{preimage of forward expansive correspondence is finite} that there exists  $M\in \N$ such that $\# T^{-1} (x)\leq M$ for all $x\in X$. Suppose that $\cR\in \tpk( X )$ satisfies properties~(a) and~(b2) in Lemma~\ref{backward conditional transition probability kernel of muQ^N_0}. Then Theorem~\ref{t_h_mu(Q) when T is expansive} indicates that $h_\mu (\cQ) =\int_X \! H(\cR_x) \diff \mu (x).$ Since for each $x\in X$, the transition probability kernel $\cR_x$ is supported on $T^{-1}(x)$ (property~(a) in Lemma~\ref{backward conditional transition probability kernel of muQ^N_0}), we have $H(\cR_x) =-\sum_{y\in T^{-1} (x)}   \cR_x (\{ y \}) \log ( \cR_x (\{ y \})) \leq \log M$ for all $x\in X$. As a result, $h_\mu (\cQ)\leq \log M$, and thus $P(T ,  \phi)= h_\mu (\cQ) +\int_X \! \int_{T(x_1)} \! \phi (x_1 ,  x_2)   \diff \cQ_{x_1} (x_2)   \diff \mu (x_1) \leq \log M + \norm{\phi}_\infty <+\infty$.
\end{proof}

Let $T$ be a forward expansive correspondence on a compact metric space $X$. If the potential function is identically zero, then Theorem~\ref{t_VP_forward_exp}~(ii) suggests that there exist $\cQ \in \tpk( X ; T )$ and $\mu \in \MMM(X,\cQ)$ such that $h(T)=h_\mu (\cQ)$.
One can show that $h(T)$ and $h_\mu (\cQ)$ are both non-negative, so only in the case that $h(T)>0$ is the equality $h(T)=h_\mu (\cQ)$ non-trivial. There have been some results that show $h(T)>0$ for some kinds of correspondences $T$, see e.g., \cite[Theorem~C]{PV17} and \cite[Theorem~3.3]{RT18}. Moreover, under their restrictions on $T$, we conclude $h_\mu (\cQ) >0$.

\section{Thermodynamic formalism for correspondences}\label{sct_Thermodynamic_formalism_of_expansive_correspondences}

In this section, we develop thermodynamic formalism in two different settings for forward expansive correspondence $T$ on a compact metric space $X$ with a continuous potential function $\phi \: \cO_2 (T)\to \R$.

In the first version, we assume that $T$ has the specification property (Definition~\ref{specification correspondence}) and that $\phi$ is Bowen summable (Definition~\ref{Bowen summable potential}). Then the Variational Principle holds, the equilibrium state exists and is unique in the sense of Theorem~\ref{equilibrium state 1}, and the unique equilibrium state can be obtained by the eigenvectors of the Ruelle operator and its adjoint operator (see Theorem~\ref{equilibrium state 1}).

In the second version, we assume that $T$ is distance-expanding (Definition~\ref{distance-expanding correspondence}), open (Definition~\ref{open correspondence}), and strongly transitive (Definition~\ref{strongly transitive correspondence}) and that $\phi$ is H\"{o}lder continuous. Then similar results hold, and in addition, we get some equidistribution properties (see Theorem~\ref{phwfq9}).

\subsection{Specification property and Bowen summability}

We introduce the specification property for correspondences. %
The notion of specification for correspondences or set-valued maps has been discussed by {\ifFirstInitial B.~E.~\fi}Raines and {\ifFirstInitial T.~\fi}Tennant \cite{RT18}, as well as {\ifFirstInitial W.~\fi}Cordeiro and {\ifFirstInitial M.~J.~\fi}Pac\'ifico \cite{CP16}. However, in order to ensure Remark~\ref{r:C_f_properties}~(ii) and Proposition~\ref{lift of specification property}, we give a definition with subtle differences from theirs and slightly stronger than the specification property given in \cite[Definition~5.1]{CP16}.

\begin{definition}[Specification property]\label{specification correspondence}
    We say that a correspondence $T$ on a compact metric space $(X,  d)$ has the \defn{specification property}\index{specification property} if, for each $\epsilon >0$, there exists $M\in \N$ with the following property:
    
    \smallskip

    For arbitrary $n\in \N$, $\vect{m}{1}{n}, \, \vect{p}{1}{n}   \in \N^n$ with $p_j >M$ every $j\in \oneto{n}$, and orbits $\vect{x^j}{0}{m_j -1} \in \cO_{m_j} (T)$ for $j\in\oneto{n}$, there exists an orbit $\vect{y}{0}{\infty} \in \cO_\omega (T)$ such that $d\bigl(y_{m(j-1)+i} ,  x^j_i \bigr)< \epsilon$ for all $j\in\oneto{n}$ and $i\in \zeroto{m_j -1}$, where $m(j)\= \sum_{k=1}^j   (m_k +p_k)$.
\end{definition}

Recall {\ifFirstInitial D.~\fi}Ruelle's definition of specification property for a continuous map from \cite[Section~1]{Ru92}:

\begin{definition}[Ruelle's specification property]\label{specification single-valued map}
    A continuous map $f\: X\to X$ on a compact space $(X,  d)$ has the \defn{specification property}\index{specification property} if, for each $\epsilon >0$, there exists $M\in \N$ with the following property:
    
    \smallskip

    For arbitrary $n\in \N$, $\vect{x}{1}{n} \in X^n$, and $\vect{m}{1}{n}, \, \vect{p}{1}{n}   \in \N^n$ with $p_j >M$ for every $j\in \oneto{n}$, there exists $z\in X$ such that $d \bigl( f^{m(j-1)+i} (z),  f^i (x_j) \bigr)< \epsilon$ for all $j\in \oneto{n}$ and $i\in \zeroto{m_j -1}$, where $m(j)\= \sum_{k=1}^j   (m_k +p_k)$.%
\end{definition}

In fact, the specification property of a correspondence $T$ in the sense of Definition~\ref{specification correspondence} implies the specification property of the corresponding shift map $\sigma \: \cO_\omega (T) \to \cO_\omega (T)$ in the sense of Definition~\ref{specification single-valued map}. This proposition corresponds to \cite[Theorem~4.1]{RT18}, with a similar proof. For the convenience of our reader, we include a proof here due to the subtle differences between Definition~\ref{specification correspondence} and the definition of specification property for correspondences in \cite{RT18}.

\begin{prop}\label{lift of specification property}
	Let $T$ be a correspondence on a compact metric space $(X, d)$ and $\sigma \: \cO_\omega (T) \to \cO_\omega (T)$ be the shift map. If $T$ has the specification property in the sense of Definition~\ref{specification correspondence}, then $\sigma$ has the specification property in the sense of Definition~\ref{specification single-valued map}.
\end{prop}

\begin{proof}
	Fix an arbitrary number $\epsilon >0$. Choose $K\in \N$ such that $1/2^K <\epsilon /2$.
	
	By the specification property of $T$, suppose that $M\in \N$ satisfies the following property:
	
	For arbitrary $n\in \N$, $\vect{m}{1}{n}, \, \vect{p}{1}{n}   \in \N^n$ with $p_j >M$ for every $j\in \oneto{n}$, and orbits $\vect{x^j}{0}{m_j -1}  \in \cO_{m_j} (T)$ for $j\in \oneto{n}$, there exists an orbit $\udy =\vect{y}{0}{\infty} \in \cO_\omega (T)$ such that $d\bigl(y_{m(j-1)+i} ,  x^j_i \bigr)< \epsilon /2$ for all $j\in \oneto{n}$ and $i\in \zeroto{m_j -1}$, where $m(j)= \sum_{k=1}^j   (m_k +p_k)$.
	
	Now fix arbitrary $n\in \N$, orbits $\vect{x^j}{0}{\infty} \in \cO_\omega (T)$, $j\in \oneto{n}$, and $\vect{m}{1}{n}, \, \vect{p}{1}{n}   \in \N^n$ with $p_j >M+K$. Since $p_j -K>M$, we can choose an orbit $\vect{y}{0}{\infty} \in \cO_\omega (T)$ such that $d \bigl(y_{m(j-1)+i} ,  x_i^j \bigr)< \frac{\epsilon}{2}$ for all $j\in \oneto{n}$ and $i\in \zeroto{m_j +K-1}$, where $m(j)=\sum_{k=1}^j   (m_k +p_k) = \sum_{k=1}^j   (m_k +K+p_k -K)$.
	
	Then for all $j\in \oneto{n}$ and $i\in \zeroto{m_j -1}$, we have
	\begin{equation*}
		\begin{aligned}
			d_\omega \bigl(\sigma^{m(j-1)+i} (z) ,  \sigma^i (x^j)\bigr)
			&= d_\omega \bigl(  \vect{y}{m(j-1)+i}{\infty}  ,  \vect{x^j}{i}{\infty} \bigr)
			= \sum_{r=0}^{+\infty}   \frac{1}{2^{r+1}} \frac{d \bigl( y_{m(j-1)+i+r} ,  x_{i+r}^j \bigr) }{ 1+d \bigl(y_{m(j-1)+i+r} ,  x_{i+r}^j\bigr) }\\
			&\leq \sum_{r=0}^{K-1}   \frac{ d \bigl(y_{m(j-1)+i+r} ,  x_{i+r}^j \bigr) }{2^{r+1}}  +\sum_{r=K}^{+\infty}   \frac{1}{2^{r+1}}
			\leq \sum_{r=1}^K   \frac{1}{2^r} \frac{\epsilon}{2} +\frac{1}{2^K}  
			<\epsilon.
		\end{aligned}
	\end{equation*}
	
	Therefore, the shift map $\sigma$ has the specification property.
\end{proof}

\begin{definition}[Bowen summability]\label{Bowen summable potential}
    Let $T$ be a forward expansive correspondence on a compact metric space $(X,  d)$. For a bounded Borel measurable function $\phi \: \cO_2 (T) \to \R$, denote
    \begin{equation*}\label{K_phi,T(delta,n)=}
    	K_{\phi ,  T} (\delta ,  n) 
    	\= \sup  \{ \abs{ S_n \phi (\udx )  - S_n \phi ( \udy ) }  : 
    	\udx , \udy \in \cO_{n+1} (T), \, d_{n+1}  ( \udx , \udy ) < \delta \} 
    \end{equation*}
    for each $n\in \N$ and each $\delta >0$.

    Choose an expansive constant $\epsilon >0$ for $T$, write $K_{\phi ,  T} (\epsilon) \= \sup \{ K_{\phi ,  T} (\epsilon ,  n)  :  n\in \N \}$, and define $\cV_T \= \{ \phi  :  K_{\phi ,  T} (\epsilon)<+ \infty \}$.
    Functions in $\cV_T$ are called \defn{Bowen summable}\index{Bowen summable} with respect to $T$.
\end{definition}

The notation $\cV_T$ above does not contain $\epsilon$ because it does not depend on $\epsilon$, which we will prove in Proposition~\ref{Bowen's property not depend on the expansive constant}.

\begin{prop}\label{Bowen's property not depend on the expansive constant}
    Let $T$ be a forward expansive correspondence on a compact metric space $(X,  d)$, $\epsilon_1 ,\, \epsilon_2 >0$ be two expansive constants for $T$ with $\epsilon_1 <\epsilon_2$, and $\phi \in \BBB(\cO_2 (T) , \R)$. There exists $L\in \N$ such that for each $n\in \N$ with $n>L$, we have
    \begin{equation}\label{K_phi,T(epsilon_1,n)leqslantK_phi,T(epsilon_2,n)leqslantK_phi,T(epsilon_1,n-L)+2L||phi||_infty}
        K_{\phi ,  T} (\epsilon_1 ,  n) \leq K_{\phi ,  T} (\epsilon_2 ,  n) \leq K_{\phi ,  T} (\epsilon_1 ,  n-L)+2L  \norm{\phi}_\infty.
    \end{equation}
\end{prop}

\begin{proof}
    First, since $\epsilon_1 <\epsilon_2$, if $\vect{x}{1}{n+1} , \, \vect{y}{1}{n+1} \in \cO_{n+1} (T)$ satisfy $d(x_k ,  y_k)< \epsilon_1$ for every $k\in \oneto{n+1}$, then $d(x_k ,  y_k)< \epsilon_2$ for every $k\in \oneto{n+1}$. Thus, by Definition~\ref{Bowen summable potential}, we have $K_{\phi ,  T} (\epsilon_1 ,  n) \leq K_{\phi ,  T} (\epsilon_2 ,  n)$.

    Now we focus on the second inequality in (\ref{K_phi,T(epsilon_1,n)leqslantK_phi,T(epsilon_2,n)leqslantK_phi,T(epsilon_1,n-L)+2L||phi||_infty}).
    By Lemma~\ref{locally long behavior of forward expansive correspondence}, we can choose $L\in \N$ such that
    for each $n\in \N$ greater than $L$, if two orbits $\vect{x}{1}{n+1}  , \,  \vect{y}{1}{n+1}  \in \cO_{n+1} (T)$ satisfy $d(x_k ,  y_k)< \epsilon_2$ for every $k\in \oneto{n+1}$, then $d(x_k ,  y_k)< \epsilon_1$ holds for every $k\in \oneto{n+1-L}$. Since%
    \begin{equation*}
         \Absbigg{ \sum_{k=1}^n   (\phi (x_k ,  x_{k+1}) -\phi (y_k ,  y_{k+1})) }
        \leq \Absbigg{ \sum_{k=1}^{n-L}   (\phi (x_k ,  x_{k+1})- \phi (y_k ,  y_{k+1})) } +2L  \norm{\phi }_\infty,
    \end{equation*}
    by Definition~\ref{Bowen summable potential}, we get $K_{\phi ,  T} (\epsilon_2 ,  n) \leq K_{\phi ,  T} (\epsilon_1 ,  n-L)+2L  \norm{\phi}_\infty$.%
\end{proof}

\begin{definition}[Distance-expanding]\label{distance-expanding correspondence}
    Let $T$ be a correspondence on a compact metric space $(X,  d)$. We say that $T$ is \defn{distance-expanding}\index{distance-expanding} if there exist $\lambda >1$, $\eta >0$, and $n\in \N$ with the property that for each $x,\, y\in X$, if $d(x ,  y)\leq \eta$, then 
    $\inf \{ d(x' ,  y')  :  x' \in T^n (x) ,\, y' \in T^n (y) \} \geq \lambda d(x ,  y)$.
\end{definition}

\begin{rem}
	Let $T$ be a distance-expanding correspondence on a compact metric space $X$. Then $T$ must be forward expansive, and thus we can say whether a bounded function $\phi \in \BBB(\cO_2 (T), \R)$ is Bowen summable. Moreover, one can check that if $\phi \in \CCC(\cO_2 (T), \R)$ is H\"{o}lder continuous with respect to the metric $d_2$ on $\cO_2 (T)$, then $\phi$ is Bowen summable.
\end{rem}

\begin{prop}\label{lift of distance-expanding}
	Let $T$ be a distance-expanding correspondence on a compact metric space $X$ and $\sigma \: \cO_\omega (T) \to \cO_\omega (T)$ be the shift map. Suppose that $\lambda >1$, $\eta >0$, and $n\in \N$ satisfy $\inf \{ d(x' ,  y')  :  x' \in T^n (x),\, y' \in T^n (y) \}\geq \lambda d(x ,  y)$ for all $x ,\, y\in X$ with $d(x ,  y)\leq \eta$. Then for an arbitrary $\lambda '\in (1 ,  \lambda)$, there exists $\eta '>0$ and $k \in \N$ with the following property:
	
	For all $\udx ,\, \udy \in \cO_\omega (T)$, if $d_\omega (\udx,  \udy) <\eta '$, then 
	$d_\omega \bigl( \sigma^{kn} (\udx) ,  \sigma^{kn} ( \udy ) \bigr)\geq \lambda' d_\omega ( \udx , \udy )$.
\end{prop}

In short, if $T$ is distance-expanding, then $\sigma$ is distance-expanding.

\begin{proof}
	Choose $k\in \N$ with $2^{kn} \cdot \frac{\lambda -\lambda '}{2\lambda} \geq \lambda '$ and set $\eta '\=  2^{-2kn} \min \bigl\{\frac{\eta}{1+\eta} ,\, \frac{\lambda -\lambda '}{2 \lambda ' (\lambda -1)}\bigr\}$. Fix arbitrary $\udx =\vect{x}{1}{\infty} ,\, \udy =\vect{y}{1}{\infty} \in \cO_\omega (T)$ with $d_\omega (\udx ,\udy) <\eta '$. We aim to prove $d_\omega \bigl( \sigma^{kn} (\udx) ,  \sigma^{kn} (\udy) \bigr) \geq \lambda ' d_\omega (\udx ,  \udy)$.
	
	For each $j \in \oneto{2kn}$, since
	\begin{equation*}
			 2^{-2kn} \min \Bigl\{\frac{\eta}{1+\eta} ,\, \frac{\lambda -\lambda '}{2 \lambda ' (\lambda -1)}\Bigr\} 
			=\eta' 
			>d_\omega ( \udx , \udy ) 
			\geq  \frac{2^{-j} d(x_j ,y_j)}{1+d(x_j ,y_j)} 
			\geq \frac{2^{-2kn} d(x_j ,y_j)}{1+d(x_j ,y_j)},
	\end{equation*}
	we have $d(x_j ,y_j)< \min\bigl\{\eta ,\, \frac{\lambda -\lambda '}{2\lambda \lambda ' -\lambda -\lambda '}\bigr\} \leq \eta$. This implies $d(x_{j+n} ,y_{j+n}) \geq \lambda d(x_j ,y_j)$ for all $j \in \oneto{2kn}$ since $x_{j+n} \in T^n (x_j)$ and $y_{j+n} \in T^n (y_j)$. Thus, $d(x_{j+kn} ,y_{j+kn}) \geq \lambda^k d(x_j ,y_j) \geq \lambda d(x_j ,y_j)$ for each $j\in \oneto{kn}$. In addition, $d(x_j ,y_j)< \frac{\lambda -\lambda '}{2\lambda \lambda ' -\lambda -\lambda '}$ implies $\frac{\lambda +\lambda '}{2\lambda} \cdot \frac{\lambda d(x_j ,y_j)}{1+ \lambda d(x_j ,y_j)} \geq \lambda ' \frac{d(x_j ,y_j)}{1+ d(x_j ,y_j)}$, which holds for all $j \in \oneto{kn}$. Recall $2^{kn} \cdot \frac{\lambda -\lambda '}{2\lambda} \geq \lambda '$. From the arguments above, for every $j \in \oneto{kn}$ we have
	\begin{equation*} 
		\begin{aligned}
			\frac{2^{kn} d(x_{j+kn} ,y_{j+kn})}{1+d(x_{j+kn} ,y_{j+kn})} 
			&\geq \frac{\lambda +\lambda '}{2\lambda} \cdot \frac{2^{kn} \lambda d(x_j ,y_j)}{1+ \lambda d(x_j ,y_j)} +\frac{\lambda -\lambda '}{2\lambda} \cdot \frac{2^{kn} d(x_{j+kn} ,y_{j+kn})}{1+d(x_{j+kn} ,y_{j+kn})}\\
			&\geq  \frac{2^{kn} \lambda' d(x_j ,y_j)}{1+ d(x_j ,y_j)} + \frac{\lambda' d(x_{j+kn} ,y_{j+kn})}{1+d(x_{j+kn} ,y_{j+kn})}
		\end{aligned}
	\end{equation*}
	Dividing both sides of the inequality above by $2^{kn+j}$ and then summing over $j$ from $1$ to $kn$, we get
	\begin{equation}\label{387d8w326726323}
		\sum_{j=1}^{kn} \frac{1}{2^j} \cdot \frac{d(x_{j+kn} ,y_{j+kn})}{1+d(x_{j+kn} ,y_{j+kn})} 
		\geq \lambda ' \sum_{j=1}^{2kn} \frac{1}{2^j} \cdot \frac{d(x_j ,y_j)}{1+ d(x_j ,y_j)}
	\end{equation}
	
	Additionally, since $\lambda ' \leq 2^{kn} \cdot \frac{\lambda -\lambda '}{2\lambda} \leq 2^{kn}$, we have
	\begin{equation}\label{38ecv372}
		\sum_{j= 2kn+1}^{+\infty} \frac{1}{2^{j-kn}} \cdot \frac{d(x_j ,y_j)}{1+ d(x_j ,y_j)} 
		\geq \lambda ' \sum_{j= 2kn+1}^{+\infty} \frac{1}{2^j} \cdot \frac{d(x_j ,y_j)}{1+ d(x_j ,y_j)}
	\end{equation}
	
	Adding (\ref{387d8w326726323}) and~(\ref{38ecv372}), we get 
	$d_\omega \bigl( \sigma^{kn} (\udx) ,  \sigma^{kn} (\udy) \bigr) \geq \lambda ' d_\omega (\udx ,  \udy)$.
	The proof is complete.
\end{proof}

Recall the notion of Bowen summability for $\varphi \in B(X ,\R)$ with respect to a forward expansive continuous map $f\: X\to X$ from \cite[Section~1]{Ru92}:

\begin{definition}[Bowen summability]
    Let $(X,  d)$ be a compact metric space and $f\: X\to X$ be a forward expansive continuous map. For a bounded Borel measurable function $\varphi \: X\to \R$, denote
    \begin{equation}\label{K_phi,f(delta,n)=}
        K_{\varphi ,  f} (\delta ,  n) \= \sup \biggl\{ \Absbigg{ \sum_{k=0}^{n-1}   \bigl(\varphi \bigl(f^k (x)\bigr) -\varphi \bigl(f^k (y)\bigr)\bigr) }  :  x,\, y \in X \text{ with } y\in B_x (\epsilon ,  n) \biggr\}
    \end{equation}
    for each $n\in \N$ and each $\delta >0$, where $B_x (\epsilon ,  n)$ is the \emph{Bowen ball} given by
    \begin{equation}\label{B_x(epsilon,n)=}
        B_x (\epsilon ,  n)\= \bigl\{ y\in X :  d\bigl(f^k (x) ,  f^k (y)\bigr)<\delta \text{ for every } k\in \zeroto{n-1} \bigr\}.
    \end{equation}

    Choose an expansive constant $\epsilon$ for $f$, we write $K_{\varphi ,  f} (\epsilon) \= \sup \{ K_{\varphi ,  f} (\epsilon ,  n)  :  n\in \N \}$, and define $\cV_f \= \{ \varphi  :  K_{\varphi ,  f} (\epsilon)<+ \infty \}$. Functions in $\cV_f$ are called \defn{Bowen summable}\index{Bowen summable} with respect to $f$.
\end{definition}

 Note that $\cV_f$ does not contain $\epsilon$ because it does not depend on $\epsilon$ (see \cite[Section~1]{Ru92}).

\begin{prop}\label{Bowen's property of phi with its lift}
    Let $T$ be a correspondence on a compact metric space $(X,  d)$. If a function $\phi\: \cO_2 (T) \to \R$ is Bowen summable with respect to $T$, then the corresponding function $\tphi \: \cO_\omega (T)\to \R$ is Bowen summable with respect to the shift map $\sigma \: \cO_\omega (T) \to \cO_\omega (T)$.
\end{prop}

\begin{proof}
    Choose a number $\epsilon >0$ small enough such that $\epsilon$ is an expansive constant for $T$ and that $\tepsilon \= \frac{\epsilon}{2(1+\epsilon)}$ is an expansive constant for the shift map $\sigma$.
    Suppose that $\udx=\vect{x}{1}{\infty}, \, \udy=\vect{y}{1}{\infty}\in \cO_\omega (T)$ satisfy $\udy \in B_{\udx} ( \tepsilon  ,  n+1 )$, i.e., $d_\omega \bigl( \sigma^k (\udx) ,\sigma^k (\udy) \bigr) < \tepsilon$ holds for all $k\in \zeroto{n}$. Then for every $k\in \zeroto{n}$,
    \begin{equation*}
         \tepsilon 
        >d_\omega \bigl( \sigma^k (\udx) ,  \sigma^k (\udy) \bigr)  
         \geq d(x_{k+1} ,  y_{k+1}) 2^{-1} (1+d(x_{k+1} ,  y_{k+1}))^{-1}.
    \end{equation*}
    This implies that $d(x_{k+1} ,\, y_{k+1})< \epsilon$ for every $k\in \zeroto{n}$. By (\ref{K_phi,f(delta,n)=}) and Definition~\ref{Bowen summable potential}, we get $K_{\tphi ,  \sigma} (  \tepsilon ,  n+1)\leq K_{\phi ,  T} (\epsilon ,  n)$. Since $\phi$ is Bowen summable with respect to $T$, we have
    \begin{equation*}
        K_{\tphi ,  \sigma} ( \tepsilon )
                    =  \sup_{n\in \N_0}  K_{\tphi ,  \sigma} (  \tepsilon ,  n+1)  
        \leq \max \Bigl\{ K_{\tphi ,  \sigma}  ( \tepsilon ,  1) , \, \sup_{n\in \N}  K_{\phi ,  T} (\epsilon ,  n)  \Bigr\}
        \leq \max \{ 2 \norm{\phi}_\infty ,  \, K_{\phi ,  T} (\epsilon) \}
        <+\infty.
    \end{equation*}
    Therefore, $\tphi$ is Bowen summable with respect to the shift map $\sigma$.
\end{proof}

\subsection{Forward expansive correspondences with the specification property}\label{subsct_Thermodynamic formalism for forward expansive correspondences with the specification property}

We aim to establish Theorem~\ref{equilibrium state 1}.
We first recall some definitions and results from \cite{RT18}.

Let $Y$ be a compact metric space, $f\: Y\to Y$ be a forward expansive continuous map with specification property, and $\psi \in \CCC ( Y, \R)$ be Bowen summable.

We recall the definition of the Ruelle operator $\cL_\psi$ acting on real-valued functions $\Phi$ on $Y$ given by
\begin{equation}\label{aovewyacbwco8ey0fawe}
    \cL_\psi (\Phi) (x)\= \sum_{y\in f^{-1} (x)}   \Phi (y) \exp (\psi (y)).
\end{equation}

The operator $\cL_\psi$ is linear and maps the space of bounded Borel functions onto itself. The action of $\cL_\psi$ on continuous functions determines completely the adjoint operator $\cL_\psi^*$, a bounded linear map on finite Borel measures on $Y$, i.e., for a finite Borel measure $\nu$ on $Y$, if $\Phi \in \CCC ( Y, \R)$ then
\begin{equation}\label{int_YPhidL_psi*(nu)=int_YL_psi(Phi)dnu}
    \int_Y \! \Phi   \diff \cL_\psi^* (\nu)= \int_Y \cL_\psi (\Phi)   \diff \nu  .
\end{equation}
This implies that (\ref{int_YPhidL_psi*(nu)=int_YL_psi(Phi)dnu}) holds for all bounded Borel measurable functions $\Phi \: Y\to \R$.

If $A\subseteq Y$ is a Borel set satisfying that $f|_A$ is injective, then we have
\begin{equation}\label{L_psi*(nu)(A)=int_f(A)psicirc(f|_A)^-1dnu}
    \cL_\psi^* (\nu) (A)= \int_Y \! \mathbbold{1}_A   \diff \cL_\psi^* (\nu)= \int_Y \cL_\psi (\mathbbold{1}_A)   \diff \nu = \int_{f(A)} \! \exp \psi \circ (f|_A)^{-1}   \diff \nu.
\end{equation}

If a non-zero Borel measure $\nu$ on $Y$ satisfies $\cL_\psi^* (\nu)= \lambda \nu$, then $\cL_\psi$ defines an operator on $L^1 (\nu)$.

We recall \cite[Theorem~2.1]{RT18}, i.e., the Ruelle--Perron--Frobenius theorem, as follows.

\begin{prop}\label{equilibrium state for single-valued maps}
    Let $Y$ be a compact metric space, $f\: Y\to Y$ be a forward expansive continuous map with specification property, and $\psi \in \CCC ( Y, \R)$ be Bowen summable. Then the following statements are true:
    \begin{enumerate}
        \smallskip
        \item[(i)] There is a unique eigenvector $\nu$ (up to a multiplicative constant) of $\cL_\psi^*$ acting on finite Borel measures on $Y$, i.e.,~$\cL_\psi^* (\nu) =\lambda \nu$.
        Moreover, $\lambda =e^{P(f,  \psi)}$ and $\nu$ is a Gibbs measure\footnote{That $\nu$ is a Gibbs measure for $\psi$ means that there is a number $c>0$ such that for all $x\in Y$ and $n\in \N$,
        \begin{equation*}
        	\exp \biggl( \sum_{k=0}^{n-1}  \bigl(\psi \bigl(f^k (x)\bigr)-nP (f,  \psi)-c\bigr) \biggr)
        	\leq \nu (B_x (\epsilon ,  n))\leq \exp \biggl( \sum_{k=0}^{n-1}   \bigl( \psi \bigl(f^k (x)\bigr)-nP (f,  \psi)+c\bigr) \biggr),
        \end{equation*}
        where $B_x (\epsilon ,  n)$ is the Bowen ball given in (\ref{B_x(epsilon,n)=}).} for $\psi$.
        \smallskip
        \item[(ii)] There is a unique non-negative eigenfunction $\Phi \in L^1 (\nu)$ (up to a multiplicative constant) of $\cL_\psi$ acting on $L^1 (\nu)$, i.e., 
        $\cL_\psi (\Phi)= \lambda \Phi \geq 0$.
        Moreover, $\lambda= e^{P(f,  \psi)}$, $\log \Phi$ is $\nu$-essentially bounded, and $\Phi \nu$ is the unique equilibrium state for $\psi$.
        \smallskip
        \item[(iii)] %
        $\lim_{n\to +\infty} e^{- n P(f,  \psi)} \cL_\psi^n (\mathbbold{1}_Y)= \Phi \text{ in } L^1 (\nu)$.
    \end{enumerate}
\end{prop}

Now suppose that $T$ is a forward expansive correspondence with the specification property on a compact metric space $X$ and that $\phi \in \CCC (\cO_2 (T) , \R)$ is Bowen summable.
By Theorem~\ref{topological pressure coincide with the lift}, (\ref{h_mu(Q)=h_muQ^N|_T(sigma)}), and~(\ref{int_O_+infty(T)tphidmuQ^N|_T=intphidmu})%
, (\ref{P(T,phi)=h_mu(Q)+int_Xphidmu intro}) is equivalent to the following equality concerning the dynamical system $(\cO_\omega (T),  \sigma)$:
\begin{equation*}\label{P(sigma,tphi)=h_muQ^N|_T(sigma)+int_O_infty(T)tphidmuQ^N|_T}
    P \bigl( \sigma ,  \tphi \bigr)= h_{\mu_\phi \cQ^\omega|_T} (\sigma) +\int_{\cO_\omega (T)} \! \tphi   \diff (\mu_\phi \cQ^\omega|_T).
\end{equation*}

By Propositions~\ref{expansiveness of correspondence with its lift} and~\ref{lift of specification property}, the forward expansiveness and specification property of $T$ imply the forward expansiveness and specification property of $(\cO_\omega (T),  \sigma)$, respectively. By Proposition~\ref{Bowen's property of phi with its lift}, the Bowen summability of $\phi$ with respect to $T$ implies the Bowen summability of $\tphi\in \CCC ( \cO_\omega (T), \R)$ with respect to $\sigma$.
As a result, we can apply Proposition~\ref{equilibrium state for single-valued maps} for $\sigma$ and $\tphi$ in the proof of Theorem~\ref{equilibrium state 1}.%

\begin{proof}[Proof of Theorem~\ref{equilibrium state 1}]
	It suffices to verify statements~(i), (ii), (iii), and the uniqueness of $(\mu_\phi,  \cQ)$.
	
	\smallskip

    (i). By Proposition~\ref{equilibrium state for single-valued maps}~(i), we can choose $\nu \in \PPP( \cO_\omega (T) )$ with
    \begin{equation}\label{L_tphi^*nu=lambda*nu}
        \cL_{\tphi}^* (\nu) = \lambda \cdot \nu,   \qquad \text{where } \lambda \=\exp \bigl( P\bigl( \sigma ,  \tphi \bigr) \bigr).
    \end{equation}

    Set $\nu_{12} \= \nu \circ \tpi_{12}^{-1}$, a Borel probability measure on $X^2$ supported on $\cO_2 (T)$, and $m_\phi \= \nu_{12} \circ \tpi_1^{-1} =\bigl(\nu \circ \tpi_{12}^{-1}\bigr) \circ \tpi_1^{-1} =\nu \circ \tpi_1^{-1}$ (see (\ref{tpi=})). By Proposition~\ref{decompose a two-dimentional measure with support}, we can choose $\cQ\in \tpk( X ; T)$ such that $m_\phi \cQ^{\zeroto{1}} =\nu_{12}$.

    We will prove $\cL_{\tphi}^* (m_\phi \cQ^\omega|_T) = \lambda \cdot m_\phi \cQ^\omega|_T$, or equivalently, for each $n\in \N \smallsetminus \{ 1\}$ and arbitrary Borel sets $A_1 ,\, \dots ,\, A_n \in \SBB (X)$ with $\diam A_1$ less than some expansive constant $\epsilon$ for $T$, we have
    \begin{equation}\label{L_tphi*(mu'Q^N|_T)(A_1*dots*A_n*X^+inftycapO_+infty(T))=lambdamu'Q^N|_T(A_1*dots*A_n*X^+inftycapO_+infty(T))}
    \begin{aligned}
         \cL_{\tphi}^* \bigl(m_\phi \cQ^\omega|_T \bigr) (A_1 \times \cdots \times A_n \times X^\omega \cap \cO_\omega (T))
        = \lambda \cdot (m_\phi \cQ^\omega|_T) (A_1 \times \cdots \times A_n \times X^\omega \cap \cO_\omega (T)).
    \end{aligned}
    \end{equation}

	Denote by $\overline{A_1}$ the closure of $A_1$. By Remark~\ref{preimage of forward expansive correspondence is finite}, the fact that $\diam \overline{A_1}$ is no more than an expansive constant for $T$ implies that $\overline{A_1} \cap T^{-1} (x_2)$ is a singleton for all $x_2 \in T \bigl(\overline{A_1} \bigr)$. This allows us to define a map $J \: T\bigl(\overline{A_1} \bigr) \to \overline{A_1}$ satisfying $\overline{A_1} \cap T^{-1} (x_2) =\{ J(x_2) \}$ for all $x_2 \in T\bigl(\overline{A_1} \bigr)$. The map $J$ is continuous because its domain, image, and graph are all compact.

    If two orbits $\udx^{(1)} ,\, \udx^{(2)} \in A_1 \times X^\omega \cap \cO_\omega (T)$ satisfy $\sigma \bigl( \udx^{(1)} \bigr)= \sigma \bigl( \udx^{(2)} \bigr)$, then $\udx^{(1)}$ and $\udx^{(2)}$ are of the form $\bigl(x_1^{(1)} ,  x_2 ,  x_3 ,  \dots \bigr)\in \cO_\omega (T)$ and $\bigl(x_1^{(2)} ,  x_2 ,  x_3 ,  \dots \bigr) \in \cO_\omega (T)$, respectively, where $x_1^{(1)} ,\, x_1^{(2)} \in A_1 ,\, x_2 \dots ,\, x_n \in X$. Since $x_2 \in T\bigl( x_1^{(1)} \bigr) \subseteq T\bigl(\overline{A_1} \bigr)$ and $x_1^{(1)} ,\, x_1^{(2)} \in \overline{A_1} \cap T^{-1} (x_2)$, we have $x_1^{(1)} =J(x_2) =x_1^{(2)}$. Thus, $\sigma$ is injective on $A_1 \times X^\omega \cap \cO_\omega (T)$ and we have $(\sigma |_{A_1 \times X^\omega \cap \cO_\omega (T)})^{-1} ( \vect{x}{2}{\infty} ) =(J(x_2) ,  x_2 ,  x_3 ,  \dots)$ for all $\vect{x}{2}{\infty} \in \sigma (A_1 \times X^\omega \cap \cO_\omega (T))$. By (\ref{L_psi*(nu)(A)=int_f(A)psicirc(f|_A)^-1dnu}), we have
    \begin{equation}\label{L_tphi*(mu'Q^N|_T)(A_1*dots*A_n*X^+inftycapO_+infty(T))=int_(x_2,dots,x_n)in(T(A_1)capA_2)*A_3*dots*A_n,x_1inA_1capT^-1(A_2)expphi(x_1)dmu'Q^[n-2](x_2,dots,x_n)}
    \begin{aligned}
        &\cL_{\tphi}^* (m_\phi \cQ^\omega|_T) (A_1 \times \cdots \times A_n \times X^\omega \cap \cO_\omega (T))\\
        &\qquad= \int_{\sigma (A_1 \times \cdots \times A_n \times X^\omega \cap \cO_\omega (T))} \! \exp  \bigl( \tphi \circ  (\sigma|_{A_1 \times X^\omega \cap \cO_\omega (T)} )^{-1} \bigr)   \diff (m_\phi \cQ^\omega|_T)\\
        &\qquad= \int_{(T(A_1) \cap A_2) \times A_3 \times \cdots \times A_n \times X^\omega} \! \exp (\phi (J(x_2) ,  x_2))   \diff (m_\phi \cQ^\omega) (\vect{x}{2}{\infty})\\
        &\qquad= \int_{(T(A_1) \cap A_2) \times A_3 \times \cdots \times A_n} \! \exp (\phi (J(x_2) ,  x_2))   \diff \bigl(m_\phi \cQ^{\zeroto{n-2}} \bigr) (\vect{x}{2}{n} ).
    \end{aligned}
    \end{equation}

    In addition, we have
    \begin{equation}\label{mu'Q^N|_T(A_1*dots*A_n*X^+inftycapO_+infty(T))=mu'Q^[n-1](A_1*dots*A_n)}
    \begin{aligned}
        & (m_\phi \cQ^\omega|_T) (A_1 \times \cdots \times A_n \times X^\omega \cap \cO_\omega (T))\\
        &\qquad= (m_\phi \cQ^\omega) (A_1 \times \cdots \times A_n \times X^\omega)= \bigl(m_\phi \cQ^{\zeroto{n-1}}\bigr) (A_1 \times \cdots \times A_n).
    \end{aligned}
    \end{equation}

    By (\ref{L_tphi*(mu'Q^N|_T)(A_1*dots*A_n*X^+inftycapO_+infty(T))=int_(x_2,dots,x_n)in(T(A_1)capA_2)*A_3*dots*A_n,x_1inA_1capT^-1(A_2)expphi(x_1)dmu'Q^[n-2](x_2,dots,x_n)}) and~(\ref{mu'Q^N|_T(A_1*dots*A_n*X^+inftycapO_+infty(T))=mu'Q^[n-1](A_1*dots*A_n)}), the equality (\ref{L_tphi*(mu'Q^N|_T)(A_1*dots*A_n*X^+inftycapO_+infty(T))=lambdamu'Q^N|_T(A_1*dots*A_n*X^+inftycapO_+infty(T))}) is equivalent to 
    \begin{equation}\label{int_(x_2,dots,x_n)in(T(A_1)capA_2*A_3*cdots*A_n,x_1inA_1capT^-1(A_2))expphi(x_1)dmuQ^[n-1](x_2,dots,x_n)=lambdamuQ^[n-1](A_1*dots*A_n)}
    \int_{(T(A_1) \cap A_2) \times A_3 \times \cdots \times A_n} \! \exp (\phi (J(x_2) ,  x_2))   \diff \bigl(m_\phi \cQ^{\zeroto{n-2}} \bigr) (\vect{x}{2}{n} )  
     = \lambda \cdot \bigl(m_\phi \cQ^{\zeroto{n-1}}\bigr) (A_1 \times \cdots \times A_n).
    \end{equation}

    We prove (\ref{int_(x_2,dots,x_n)in(T(A_1)capA_2*A_3*cdots*A_n,x_1inA_1capT^-1(A_2))expphi(x_1)dmuQ^[n-1](x_2,dots,x_n)=lambdamuQ^[n-1](A_1*dots*A_n)}) by induction on $n$. 
    If $n=2$, we rewrite (\ref{int_(x_2,dots,x_n)in(T(A_1)capA_2*A_3*cdots*A_n,x_1inA_1capT^-1(A_2))expphi(x_1)dmuQ^[n-1](x_2,dots,x_n)=lambdamuQ^[n-1](A_1*dots*A_n)}) as
    \begin{equation}\label{int_x_2inT(A_1)capA_2,x_1inA_1capT^-1(A_2)expphi(x_1)dmu(x_2)=lambda*muQ^[1](A_1*A_2)}
        \int_{T(A_1)\cap A_2} \! \exp (\phi (J(x_2) ,  x_2))   \diff m_\phi (x_2)
        = \lambda \cdot \bigl(m_\phi \cQ^{\zeroto{1}}\bigr) (A_1 \times A_2).
    \end{equation}
    To prove (\ref{int_x_2inT(A_1)capA_2,x_1inA_1capT^-1(A_2)expphi(x_1)dmu(x_2)=lambda*muQ^[1](A_1*A_2)}), we come back to the property of $\nu$. The equality (\ref{L_tphi^*nu=lambda*nu}) implies
    \begin{equation*}
    \begin{aligned}
        & \int_{\sigma (A_1 \times A_2 \times X^\omega \cap \cO_\omega (T))} \! \exp \bigl( \tphi \circ \bigl( \sigma |_{A_1 \times X^\omega \cap \cO_\omega (T)} \bigr)^{-1} \bigr)   \diff \nu \\
        &\qquad=\lambda \cdot \nu (A_1 \times A_2 \times X^\omega \cap \cO_\omega (T))
            =\lambda \cdot \nu_{12} (A_1 \times A_2 \cap \cO_2 (T)) 
           = \lambda \cdot \bigl(m_\phi \cQ^{\zeroto{1}}\bigr) (A_1 \times A_2).
    \end{aligned}
    \end{equation*}
    Moreover, recalling $m_\phi =\nu \circ \tpi_1^{-1}$, we have
    \begin{equation*}
    \begin{aligned}
        &\int_{\sigma (A_1 \times A_2 \times X^\omega \cap \cO_\omega (T))} \! \exp \bigl( \tphi \circ \bigl( \sigma |_{A_1 \times X^\omega \cap \cO_\omega (T)} \bigr)^{-1} \bigr)   \diff \nu\\
        &\qquad= \int_{(T(A_1)\cap A_2) \times X^\omega \cap \cO_\omega (T)} \! \exp (\phi (J(x_2) ,  x_2))   \diff \nu ( \vect{x}{2}{\infty} )  
        = \int_{T(A_1)\cap A_2} \! \exp (\phi (J(x_2) ,  x_2))   \diff m_\phi (x_2).
    \end{aligned}
    \end{equation*}
    Hence, (\ref{int_x_2inT(A_1)capA_2,x_1inA_1capT^-1(A_2)expphi(x_1)dmu(x_2)=lambda*muQ^[1](A_1*A_2)}) holds, i.e.,  (\ref{int_(x_2,dots,x_n)in(T(A_1)capA_2*A_3*cdots*A_n,x_1inA_1capT^-1(A_2))expphi(x_1)dmuQ^[n-1](x_2,dots,x_n)=lambdamuQ^[n-1](A_1*dots*A_n)}) holds for $n=2$.
    Suppose that (\ref{int_(x_2,dots,x_n)in(T(A_1)capA_2*A_3*cdots*A_n,x_1inA_1capT^-1(A_2))expphi(x_1)dmuQ^[n-1](x_2,dots,x_n)=lambdamuQ^[n-1](A_1*dots*A_n)}) holds for $n-1 ,\, n\geq 3$, i.e.,
    \begin{equation*}
    	\begin{aligned}
    		        \int_{(T(A_1) \cap A_2) \times A_3 \times \cdots \times A_{n-1}} \! \exp (\phi (J(x_2) ,  x_2))   \diff \bigl(m_\phi \cQ^{\zeroto{n-3}} \bigr) \bigl( \vect{x}{2}{n-1} \bigr)
    		         = \lambda \cdot \bigl(m_\phi \cQ^{\zeroto{n-2}}\bigr) (A_1 \times \cdots \times A_{n-1}).
        \end{aligned}
    \end{equation*}
    This and~(\ref{muQ^{n}(A_0**A_n)}) in Lemma~\ref{l_Qun} imply that%
    \begin{equation*}
    \begin{aligned}
         \lambda \cdot \bigl(m_\phi \cQ^{\zeroto{n-1}}\bigr) (A_1 \times \cdots \times A_n)    
        & = \lambda \int_{A_1 \times \cdots \times A_{n-1}} \! \cQ (x_{n-1} ,  A_n)   \diff \bigl(m_\phi \cQ^{\zeroto{n-2}} \bigr) \bigl( \vect{x}{1}{n-1} \bigr) \\
        & = \int_{(T(A_1) \cap A_2) \times A_3 \times \cdots \times A_{n-1}} \! \cQ (x_{n-1} ,  A_n) e^{\phi (J(x_2) ,  x_2)}   \diff \bigl(m_\phi \cQ^{\zeroto{n-3}} \bigr) \bigl( \vect{x}{2}{n-1} \bigr)\\
        & = \int_{(T(A_1) \cap A_2) \times A_3 \times \cdots \times A_n} \! e^{\phi (J(x_2) ,  x_2)} \diff \bigl(m_\phi \cQ^{\zeroto{n-2}} \bigr) ( \vect{x}{2}{n} ).
    \end{aligned}
    \end{equation*}

    Hence, (\ref{int_(x_2,dots,x_n)in(T(A_1)capA_2*A_3*cdots*A_n,x_1inA_1capT^-1(A_2))expphi(x_1)dmuQ^[n-1](x_2,dots,x_n)=lambdamuQ^[n-1](A_1*dots*A_n)}) holds for $n$, and therefore $\cL_{\tphi}^* \bigl(m_\phi \cQ^\omega|_T \bigr) = \lambda \cdot m_\phi \cQ^\omega|_T$.

    (ii). Let $v\: X\to \R$ be a non-negative bounded Borel measurable function. For each $\vect{x}{1}{\infty} \in \cO_\omega (T)$, we have
    \begin{equation*}
    \begin{aligned}
        \cL_{\tphi} \tv (\vect{x}{1}{\infty})= \sum_{\vect{x}{0}{\infty}\in \sigma^{-1} ( \vect{x}{1}{\infty} )}   \tv ( \vect{x}{0}{\infty} ) \exp \bigl(\tphi ( \vect{x}{0}{\infty} )\bigr)
        = \sum_{x_0 \in T^{-1} (x_1)}  v(x_0) \exp (\phi (x_0 ,  x_1)),
    \end{aligned}
    \end{equation*}
    which indicates that $\cL_{\tphi} \tv ( \vect{x}{1}{\infty} )$ only depends on $x_1$. Consequently, there exists a function $w\: X\to \R$ such that $\tw =\cL_{\tphi} \tv$. Then one can check the non-negativeness, boundedness, and Borel measurability of $w$.

    Thus,
    $\exp \bigl(-n P\bigl(\sigma ,  \tphi \bigr)\bigr) \cdot \cL_{\tphi}^n (\mathbbold{1}_{\cO_\omega (T)})
    	= \exp \bigl(-n P\bigl(\sigma ,  \tphi \bigr)\bigr) \cdot \cL_{\tphi}^n \bigl( \wt{\mathbbold{1}_X} \bigr)$
    is of the form $\tu_n$ for some non-negative bounded Borel measurable function $u_n \: X\to \R$ for each $n\in \N$. By Proposition~\ref{equilibrium state for single-valued maps}~(ii) and~(iii), $\{ \tu_n \}_{n\in \N}$ converges to $\Phi$ in $L^1 (m_\phi \cQ^\omega|_T)$, where $\Phi$ is the only non-negative eigenfunction of $\cL_{\tphi}$ acting on $L^1 (m_\phi \cQ^\omega|_T)$.

    By taking $n=0$ in (\ref{muQ^N(A*X^infty)=muQ^[n](A)}), we get that the sequence $\{ u_n \}_{n\in \N}$ converges in $L^1 (m_\phi)$. Suppose that $u_n$ converges to $u_\phi \in L^1 (m_\phi)$ as $n\to +\infty$ in $L^1 (m_\phi)$. Then $\tu_n$ converges to $\tu_\phi$ as $n\to +\infty$ in $L^1 (m_\phi\cQ^\omega|_T)$. Thus, $\Phi =\tu_\phi$ in $L^1 (m_\phi)$. By Proposition~\ref{equilibrium state for single-valued maps}~(ii), we have $\cL_{\tphi} (\tu_\phi)= \lambda \tu_\phi$.

    (iii). By Proposition~\ref{equilibrium state for single-valued maps}~(ii), $\tu_\phi (m_\phi\cQ^\omega|_T) \in \PPP (\cO_\omega (T))$ is the unique equilibrium state for $\sigma$.

    Set $\mu_\phi \= u_\phi m_\phi$. %
    For each Borel set $M\in \SBB (X^\omega)$, by Lemmas~\ref{3w94g8r0q7gwd0awedcbx} and~\ref{203q9j}, we have
    \begin{equation*}
    \begin{aligned}
         \tu_\phi ( m_\phi\cQ^\omega|_T) (M\cap \cO_\omega (T))
        &= \int_{M\cap \cO_\omega (T)} \! \tu_\phi ( \vect{x}{1}{\infty} )   \diff (m_\phi\cQ^\omega|_T) ( \vect{x}{1}{\infty} )
        =\int_M \! u_\phi (x_1)   \diff (m_\phi\cQ^\omega) ( \vect{x}{1}{\infty} )\\
        &= \int_X \! \biggl( \int_{X^\omega} \! \mathbbold{1}_M ( \vect{x}{1}{\infty} ) \cdot u_\phi (x_1)   \diff \cQ^\omega_x ( \vect{x}{1}{\infty} ) \biggr)   \diff m_\phi (x) \\
        &= \int_X \! \biggl( \int_{X^\omega} \! \mathbbold{1}_M ( \vect{x}{1}{\infty} ) \cdot u_\phi (x)   \diff \cQ^\omega_x ( \vect{x}{1}{\infty} ) \biggr)   \diff m_\phi (x) \\
        &= \int_X \! u_\phi (x) \cdot \cQ^\omega (x,  M)   \diff m_\phi(x) 
        = \int_X \! \cQ^\omega (x,  M)   \diff (u_\phi m_\phi) (x)\\
        &=((u_\phi m_\phi)\cQ^\omega) (M) 
        =((u_\phi m_\phi) \cQ^\omega|_T) (M\cap \cO_\omega (T)).
    \end{aligned}
    \end{equation*}

    Hence, $\mu_\phi \cQ^\omega|_T = (u_\phi m_\phi)\cQ^\omega|_T = \tu_\phi (m_\phi\cQ^\omega|_T)$ is the unique equilibrium state for $\sigma$, and therefore (\ref{P(T,phi)=h_mu(Q)+int_Xphidmu intro}) holds for $(\mu_\phi ,  \cQ)$.

\smallskip

We have finished constructing an equilibrium state $(\mu_\phi,  \cQ)$ for the correspondence $T$ and potential function $\phi$. %
Now we show that it is unique in the sense of Theorem~\ref{equilibrium state 1}.

    Recall that (\ref{P(T,phi)=h_mu(Q)+int_Xphidmu intro}) is equivalent to that $\mu \cQ^\omega |_T$ is an equilibrium state for $\tphi$ in the dynamical system $(\cO_\omega (T),  \sigma)$. Since $\sigma \: \cO_\omega (T)\to \cO_\omega (T)$ is a forward expansive continuous map with specification property and $\tphi \in \CCC ( \cO_\omega (T) , \R)$ is Bowen summable, Proposition~\ref{equilibrium state for single-valued maps}~(ii) says that the equilibrium state for $\tphi$ in the dynamical system $(\cO_\omega (T),  \sigma)$ is unique.

    Suppose that both $(\mu ,  \cQ)$ and $(\mu ' ,  \cQ ')$ are equilibrium states for the correspondence $T$ and potential function $\phi$, then both $\mu \cQ^\omega |_T$ and $\mu ' (\cQ ')^\omega |_T$ are equilibrium states for $\sigma$ and $\tphi$, and thus $\mu \cQ^\omega |_T = \mu ' (\cQ ')^\omega |_T$. Thereby, we have $\mu \cQ^\omega = \mu ' (\cQ ')^\omega$ by Lemma~\ref{the orbit space is of full measure}. By (\ref{muQ^N(A*X^infty)=muQ^[n](A)}), we have $\mu \cQ^{\zeroto{1}} = \mu ' (\cQ ')^{\zeroto{1}}$. By Proposition~\ref{decompose a two-dimentional measure}, we conclude $\mu =\mu '$ and that for $\mu$-almost $x\in X$ and all $A\in \SBB (X)$, the equality $\cQ (x,  A)= \cQ '(x,  A)$ holds.
\end{proof}

\subsection{Open, distance-expanding, strongly transitive correspondences}

We aim to prove Theorem~\ref{phwfq9}, which provides another version of conditions that ensure the Variational Principle, the existence and uniqueness of the equilibrium state, and some equidistribution properties. 

\begin{definition}[Openness]\label{open correspondence}
    Let $T$ be a correspondence on a compact metric space $(X,d)$. We say that $T$ is \defn{open}\index{open} if, for each open subset $U\subseteq X$, $T(U)$ is an open subset of $X$.
\end{definition}

\begin{prop}\label{lift of open}
	Let $T$ be a correspondence on a compact metric space $X$ and $\sigma \: \cO_\omega (T) \to \cO_\omega (T)$ be the shift map. Then $T$ is open if and only if $\sigma$ is an open map.
\end{prop}

\begin{proof}
	First, we assume that $T$ is open.	
	Fix an arbitrary open set $\udU \subseteq \cO_\omega (T)$ and an arbitrary orbit $\udy \in \sigma (\udU)$. We claim that $\sigma (\udU)$ is a neighborhood of $\udy$. Indeed, we choose $\udx\in \udU$ such that $\sigma (\udx) =\udy$. Since $\udU$ is an open subset of $\cO_\omega (T)$, we can choose $n\in \N$ and open subsets $V_1 ,\, \dots ,\, V_n$ of $X$ such that $\udx \in V_1\times \cdots \times V_n \times X^\omega \cap \cO_\omega (T) \subseteq \udU$. Since $T$ is open in the sense of Definition~\ref{open correspondence}, we have $T(V_1)$ is an open subset of $X$, and thus $\sigma (V_1\times \cdots \times V_n \times X^\omega \cap \cO_\omega (T)) =(T(V_1) \cap V_2) \times V_3 \times \cdots \times V_n \times X^\omega \cap \cO_\omega (T)$ is an open subset of $\cO_\omega (T)$. The claim is thus established because $\udy =\sigma (\udx) \in \sigma (V_1\times \cdots \times V_n \times X^\omega \cap \cO_\omega (T)) \subseteq \sigma (\udU)$. Because $\udy \in \sigma (\udU)$ is chosen arbitrarily, we conclude that $\sigma (\udU)$ is an open subset of $\cO_\omega (T)$. Hence, $\sigma$ is an open map.
	
	Now we assume that $\sigma$ is an open map.
	Choose an arbitrary open subset $U$ of $X$. The openness of $\sigma$ yields that $\sigma (U\times X^\omega \cap \cO_\omega (T)) =T(U) \times X^\omega \cap \cO_\omega (T)$ is an open subset of $\cO_\omega (T)$. We argue by contradiction and assume that $T(U)$ is not open in $X$. This allows us to choose a sequence $\{x_n\}_{n\in \N}$ in $X \smallsetminus T(U)$ that converges to a point $x_0 \in T(U)$. For each $n\in \N$, we choose $\udx_n \in \cO_\omega (T)$ with $\tpi_1 (\udx_n) =x_n$, i.e., the first coordinate of $\udx_n$ is $x_n$. Since $\cO_\omega (T)$ is compact, there is an increasing sequence of positive integers $\{n_l\}_{l \in \N}$ such that $\udx_{n_l}$ converges to some orbit $\udx_0 \in \cO_\omega (T)$ as $l \to +\infty$. We have $\tpi_1 (\udx_0) =\lim_{l\to +\infty} \tpi_1 (\udx_{n_l}) =\lim_{l\to +\infty} x_{n_l} =x_0 \in T(U)$, which indicates $\udx_0 \in T(U) \times X^\omega \cap \cO_\omega (T)$. Since $T(U) \times X^\omega \cap \cO_\omega (T)$ is open in $\cO_\omega (T)$ and $\lim_{l \to +\infty} \udx_{n_l} =\udx_0$, there exists $l \in \N$ such that $\udx_{n_l} \in T(U) \times X^\omega \cap \cO_\omega (T)$, i.e., $x_{n_l} =\tpi_1 (\udx_{n_l}) \in T(U)$. This contradicts $x_n \in X \smallsetminus T(U)$ for all $n\in \N$. Hence, $T(U)$ is open, and we conclude that $T$ is open.
\end{proof}

\begin{definition}[Strong transitivity]\label{strongly transitive correspondence}
    We say that a correspondence $T$ on a compact metric space $X$ is \defn{strongly transitive}\index{strongly transitive} if $\bigcup_{n=1}^{+\infty} T^{-n} (x)$ is dense in $X$ for every $x\in X$.
\end{definition}

\begin{remark}
    We call this property \defn{strongly transitive} because if $T=\cC_f$ for some continuous map $f\: X\to X$, then this property is slightly stronger than topological transitivity (cf.~Remark~\ref{r:C_f_properties}~(v)).
\end{remark}

\begin{prop}\label{lift of transitive}
	Let $T$ be a correspondence on a compact metric space $X$ and $\sigma \: \cO_\omega (T) \to \cO_\omega (T)$ be the shift map. If $T$ is open and strongly transitive, then $\sigma$ is topologically transitive.
\end{prop}

\begin{proof}
	To prove that $\sigma$ is transitive, we choose two arbitrary non-empty open subsets $\udU_1 ,\, \udU_2$ of $\cO_\omega (T)$ and aim to show that there exists $n\in \N$ such that $\sigma^n (\udU_1) \cap \udU_2 \neq \emptyset$. Without loss of generality, we assume $\udU_1 =V_1 \times \cdots \times V_m \times X^\omega \cap \cO_\omega (T)$, where $m\in \N$ and $V_1 ,\, \dots ,\, V_m$ are open subsets of $X$. We define $W_1 ,\, \dots ,\, W_m$ recursively as:
	\begin{equation}\label{W:sigma^m-1}
		W_1 \= V_1 ,\, W_{k+1} \= T(W_k) \cap V_{k+1} \text{ for all } k\in \oneto{m-1}.
	\end{equation}
	Since $V_1 ,\, \dots ,\, V_m$ are open in $X$ and since $T$ is open, we can get that $W_k$ is open in $X$ for all $k \in \oneto{m}$ by induction on $k$. Moreover, we can prove
	\begin{equation*}
		\sigma^{j-1} (V_1 \times \cdots \times V_m \times X^\omega \cap \cO_\omega (T)) =W_j \times V_{j+1} \times \cdots \times V_m \times X^\omega \cap \cO_\omega (T)
	\end{equation*}
	for all $j\in \oneto{m}$ by induction on $j$. In particular, $\sigma^{m-1} (\udU_1) =W_m \times X^\omega \cap \cO_\omega (T)$. The open subset $W_m$ of $X$ is non-empty because $\sigma^{m-1} (\udU_1)\neq \emptyset$ due to $\udU_1 \neq \emptyset$.
	
	As $\udU_2\neq \emptyset$, we can choose an orbit $(x_1,  x_2,  \dots) \in \udU_2$. The strong transitivity of $T$ ensures that $\bigcup_{n=1}^{+\infty} T^{-n} (x_1)$ is dense in $X$, so there exists $n\in \N$ such that $T^{-n} (x_1) \cap W_m \neq \emptyset$. This allows us to choose $\vect{y}{0}{n}  \in \cO_n (T)$ with $y_0 \in W_m$ and $y_n =x_1$. Consider the orbit $\udx_0 \=(y_0 ,  \dots ,  y_{n-1} ,  x_1 ,  x_2 ,  \dots) \in \cO_\omega (T)$. Since $y_0 \in W_m$, we have $\udx_0 \in W_m \times X^\omega \cap \cO_\omega (T) =\sigma^{m-1} (\udU_1)$.
	Moreover, $\sigma^n (\udx_0) =\vect{x}{1}{\infty} \in \udU_2$. Hence, $\sigma^{m+n-1} (\udU_1) \cap \udU_2 \neq \emptyset$ and we conclude that $\sigma$ is topologically transitive.
\end{proof}

\begin{definition}[Topological exactness]\label{topologically exact}
    Let $T$ be a correspondence on a compact metric space $X$. We say that $T$ is \defn{topologically exact}\index{topologically exact}\footnote{{\ifFirstInitial C.~\fi}Siqueira and {\ifFirstInitial D.~\fi}Smania \cite[Section~4.4]{SS17} called this property locally eventually onto for $\boldsymbol{f}_c$ on what they called ``hyperbolic repellers''.} if for every non-empty open subset $U\subseteq X$, there exists $N\in \N$ such that $T^N (U)=X$.
\end{definition}

\begin{prop}\label{lift of topologically exact}
	Let $T$ be a correspondence on a compact metric space $X$ and $\sigma \: \cO_\omega (T) \to \cO_\omega (T)$ be the shift map. If $T$ is open and topologically exact, then $\sigma$ is topologically exact.
\end{prop}

\begin{proof}
	To prove that $\sigma$ is topologically exact, we fix an arbitrary open subset $\udU\neq \emptyset$ of $\cO_\omega (T)$ and aim to show that there exists $n\in \N$ such that $\sigma^n (\udU) =\cO_\omega (T)$. Without loss of generality, we assume $\udU =V_1 \times \cdots \times V_m \times X^\omega \cap \cO_\omega (T)$, where $m\in \N$ and $V_1 ,\, \dots ,\, V_m$ are open subsets of $X$. %
	We define $W_1 ,\, \dots ,\, W_m$ recursively as (\ref{W:sigma^m-1}). Then we have proved in the proof of Proposition~\ref{lift of transitive} that $\sigma^{m-1} (\udU) =W_m \times X^\omega \cap \cO_\omega (T)$ and that $W_m$ is a non-empty open subset of $X$ since $T$ is open.
	
	By the topologically exactness of $T$, there exists $N \in \N$ such that $T^N (W_m) =X$, and thus 
	\begin{equation*}
		\begin{aligned}
			\sigma^{m+N-1} (\udU) 
			=\sigma^N (W_m \times X^\omega \cap \cO_\omega (T)) 
			=T^N (W_m) \times X^\omega \cap \cO_\omega (T) 
			=X\times X^\omega \cap \cO_\omega (T) 
			=\cO_\omega (T).
		\end{aligned}
	\end{equation*}
	Therefore, $\sigma$ is topologically exact.
\end{proof}

\begin{prop}\label{lift of Holder}
    Let $X$ be a compact metric space and $T$ be a correspondence on $(X,  d)$. If $\phi \: \cO_2 (T) \to \R$ is $\alpha$-H\"{o}lder continuous with respect to the metric $d_2$ on $\cO_2 (T)$, then the function $\tphi \: \cO_\omega (T) \to \R$ is $\alpha$-H\"{o}lder continuous with respect to the metric $d_\omega$ on $\cO_\omega (T)$.
\end{prop}

\begin{proof}
    Suppose that $\phi$ is $\alpha$-H\"{o}lder continuous with respect to the metric $d_2$ on $\cO_2 (T)$ and that a constant $C>0$ satisfy 
    $\abs{\phi (x_1 ,  x_2)- \phi (y_1 ,  y_2)} \leq C \cdot d_2 ((x_1 ,  x_2) ,  (y_1 ,  y_2))^\alpha$
    for all $(x_1 ,  x_2) ,\, (y_1 ,  y_2)\in \cO_2 (T)$. Then for arbitrary $\udx =\vect{x}{1}{\infty}$ and $\udy =\vect{y}{1}{\infty}$ in $\cO_\omega (T)$, we have 
    \begin{equation*}
    	\Absbig{  \tphi (\udx) -\tphi (\udy) } = \abs{\phi (x_1 ,  x_2) -\phi (y_1 ,  y_2)}\leq C \cdot \max\{ d(x_1 ,  y_1) ,  d(x_2 ,  y_2) \}^\alpha.
    \end{equation*}
    Since 
    $d_\omega (\udx ,  \udy) 
    		\geq \frac{d(x_1 ,  y_1)}{2(1+ d(x_1 ,  y_1))} +\frac{d(x_2 ,  y_2)}{4(1+ d(x_2 ,  y_2))} 
    		\geq \frac{\max \{ d(x_1 ,  y_1) ,  d(x_2 ,  y_2) \}}{4(1+ \diam X)}$,
    we have 
    \begin{equation*}
    	\Absbig{ \tphi (\udx) -\tphi (\udy) } \leq C \cdot (4(1+ \diam X))^\alpha \cdot d_\omega (\udx ,  \udy)^\alpha.
    \end{equation*}
    Therefore, $\tphi$ is $\alpha$-H\"{o}lder continuous with respect to the metric $d_\omega$ on $\cO_\omega (T)$.
\end{proof}

Let $T$ be a correspondence on a compact metric space $X$, recall $\O_{-n} (x)= \cO_{n+1} (T) \cap X^n \times \{ x\} =\{ \vect{y}{0}{n}  \in \cO_{n+1} (T)  :  y_n =x \}$ for all $n\in \N$ and $x\in X$ from Theorem~\ref{phwfq9}. If $T$ is forward expansive, then the set $\O_{-n} (x)$ is finite for all $n\in \N$ and $x\in X$, ensured by the fact shown in Remark~\ref{preimage of forward expansive correspondence is finite} that $T^{-1} (y)=\{ z\in X  :  y\in T(z) \}$ is a finite set for all $y\in X$.

The proofs of Theorem~\ref{phwfq9}~(2) and the uniqueness of the equilibrium state are similar to the proofs of Theorem~\ref{equilibrium state 1}~(i),~(ii),~(iii), and the uniqueness of the equilibrium state and the proofs of Theorem~\ref{phwfq9}~(3)(a) and~(3)(b) are similar, so now we sketch the proofs of Theorem~\ref{phwfq9}~(2),~(3)(b), and the uniqueness of the equilibrium state and give a detailed proof of Theorem~\ref{phwfq9}~(3)(a).

First, to prove Theorem~\ref{phwfq9}~(2), we should note that if an open, strongly transitive, distance-expanding correspondence $T$ on $X$ is given, then $\sigma \: \cO_\omega (T) \to \cO_\omega (T)$ is open (by Proposition~\ref{lift of open}), topologically transitive (by Proposition~\ref{lift of transitive}), and distance-expanding (by Proposition~\ref{lift of distance-expanding}). Also, by Proposition~\ref{lift of Holder}, the lifted potential function $\tphi \: \cO_\omega (T) \to \R$ is $\alpha$-H\"{o}lder continuous if $\phi \: X\to \R$ is $\alpha$-H\"{o}lder continuous. Thereby, under the setting of Theorem~\ref{phwfq9}, we can apply the following version of the Ruelle--Perron--Frobenius theorem for $\sigma$ and potential  $\tphi$.

\begin{prop}\label{equilibrium state for single-valued maps2}
    Let $Y$ be a compact metric space, $f\: Y\to Y$ be an open, topologically transitive, distance-expanding continuous map, and $\psi\: Y\to \R$ be an $\alpha$-H\"{o}lder continuous function with respect to the metric on $Y$, where $\alpha \in (0,  1)$. Then the following statements are true:
    \begin{enumerate}
        \smallskip
        \item[(i)] There is a unique eigenvector $\nu$ (up to a multiplicative constant) of $\cL_\psi^*$ acting on finite Borel measures on $Y$, i.e.,
        $\cL_\psi^* (\nu) =\lambda \nu$. 
        Moreover, $\lambda =e^{P(f,  \psi)}$ and $\nu$ is a Gibbs measure for $\psi$.
        \smallskip
        \item[(ii)] There is a unique positive $\alpha$-H\"{o}lder eigenfunction $\Phi$ (up to a multiplicative constant) of $\cL_\psi$, i.e.,
        $\cL_\psi (\Phi)= \lambda \Phi>0$.
        Moreover, $\lambda= e^{P(f,  \psi)}$ and $\Phi \nu$ is the unique equilibrium state for $\psi$.
        \smallskip
        \item[(iii)] The sequence $e^{-n P(f,  \psi)} \cdot \cL_\psi^n (\mathbbold{1}_Y)$ converges uniformly to $\Phi$ as $n\to +\infty$.
    \end{enumerate}
    In addition, the backward orbits under $f$ are equidistributed with respect to the measure $\Phi \nu$. More precisely, if we write $W(z,n) \= \exp \bigl( \sum_{i=0}^{n-1}   \psi (f^i (z)) \bigr)$, the following statements are true for all $y\in Y$:
    \begin{enumerate}
        \smallskip
        \item[(a)] $\frac{1}{\sum_{z \in f^{-n} (y)}  W(z,n) } \sum_{z \in f^{-n} (y)}   \frac{\sum_{j=0}^n   \delta_{f^j (z) } W(z,n)}{n+1}   \in \PPP(Y)$
        converges to $\Phi \nu$ in the weak* topology as $n \to +\infty$.
        \smallskip
        \item[(b)] If, moreover, $f$ is topologically exact, then  
        $\frac{1}{\sum_{z \in f^{-n} (y)}   W(z,n) } \sum_{z \in f^{-n} (y)} \delta_z  W(z,n)  \in \PPP(Y)$
        converges to $\nu$ in the weak* topology as $n \to +\infty$.
    \end{enumerate}
\end{prop}

This proposition is summarized from \cite[Chapter~5]{PU10}. In detail, statement~(i) comes from \cite[Theorem~5.2.8, Propositions~5.2.11, and~5.1.1]{PU10}, statement~(ii) comes from \cite[Propositions~5.1.5,~5.3.1,~5.2.10, Theorems~5.3.2, and~5.6.2]{PU10}, statement~(iii) comes from \cite[Section~5.4, (5.4.2)]{PU10}, statement~(a) comes from \cite[Remark~4.4.4]{PU10}, and statement~(b) comes from \cite[Section~5.4, (5.4.4)]{PU10}.

By applying Proposition~\ref{equilibrium state for single-valued maps2}~(i) for $\sigma$ and $\tphi$, we can get the unique $\nu \in \PPP (\cO_\omega (T))$ with $\cL_{\tphi}^* (\nu) = \exp \bigl( P\bigl( \sigma ,  \tphi \bigr) \bigr) \cdot \nu$. The proof of Theorem~\ref{equilibrium state 1}~(i) indicates that $\nu$ is of the form $m_\phi \cQ^\omega |_T$, where $m_\phi \in \PPP(X)$ and $\cQ \in \tpk( X ; T )$. Consequently, part~(i) in Theorem~\ref{phwfq9}~(2) follows.

By applying Proposition~\ref{equilibrium state for single-valued maps2}~(ii) for $\sigma$ and $\tphi$, we can get the unique $\alpha$-H\"{o}lder function $\Phi \: \cO_\omega (T) \to \R$ with $\cL_{\tphi} (\Phi) = \exp \bigl( P\bigl( \sigma ,  \tphi \bigr) \bigr) \cdot \Phi$. Applying Proposition~\ref{equilibrium state for single-valued maps2}~(iii), we can see $\Phi =\tu_\phi$ for some function $u_\phi \in L^1 (m_\phi)$ following the proof of Theorem~\ref{equilibrium state 1}~(ii). In addition, suppose that $T$ is continuous in the sense of Definition~\ref{continuity}, we aim to prove that $u_\phi$ is continuous. Fix an arbitrary $\epsilon >0$. By Proposition~\ref{equilibrium state for single-valued maps2}~(ii), $\tu_\phi$ is H\"{o}lder continuous, so we can choose $\delta >0$ such that $\abs{\tu_\phi (\udx) -\tu_\phi (\udy)} <\epsilon$ holds for all $\udx ,\, \udy \in \cO_\omega (T)$ with $d_\omega (\udx , \udy) <2\delta$. Choose $n\in \N$ with $2^{-n} <\delta$. Set $\delta_n \= \delta$. By Definition~\ref{continuity}, the compactness of $X$ implies that $T\: X \to \cF (X)$ is uniformly continuous, with $\cF(X)$ equipped with the Hausdorff distance $d_H$. This allows us to choose $\delta_n > \delta_{n-1} > \delta_{n-2} > \cdots > \delta_1>0$ one by one such that $d_H (T(x) ,T(y)) <\delta_{k+1}$ for all $k\in \oneto{n-1}$ and $x ,\, y\in X$ with $d(x ,y)<\delta_k$.

Fix arbitrary $x_1 ,\, y_1 \in X$ with $d (x_1 ,y_1)< \delta_1$. By induction on $k$, we can choose $x_2 \in T(x_1) ,\, y_2 \in T(y_1) ,\, \dots ,\, x_n \in T(x_{n-1}) ,\, y_n \in T(y_{n-1})$ such that $d(x_k ,y_k) <\delta_k <\delta_n =\delta$ for all $k\in \oneto{n}$. Furthermore, we choose $\vect{x}{n+1}{\infty} ,\, \vect{y}{n+1}{\infty} \in X^\omega$ such that $\vect{x}{1}{\infty} \in \cO_\omega (T)$ and $\vect{y}{1}{\infty} \in \cO_\omega (T)$. We have
\begin{equation*}
	d_\omega (\vect{x}{1}{\infty} ,\vect{y}{1}{\infty}) =\sum_{k=1}^{+\infty} \frac{1}{2^k} \frac{d(x_k ,y_k)}{1+ d(x_k ,y_k)} \leq \sum_{k=1}^n \frac{1}{2^k} \delta +\sum_{k=n+1}^{+\infty} \frac{1}{2^k} <\delta +2^{-n} <2\delta.
\end{equation*}
This implies $\abs{u_\phi (x_1) -u_\phi (y_1)} =\abs{\tu_\phi (\vect{x}{1}{\infty}) -\tu_\phi (\vect{y}{1}{\infty})} <\epsilon$. Since $\epsilon$ is chosen arbitrarily, we conclude that $u_\phi$ is continuous. Part~(ii) in  Theorem~\ref{phwfq9}~(2) follows.

We have proved $(u_\phi m_\phi)\cQ^\omega|_T = \tu_\phi (m_\phi\cQ^\omega|_T)$ in the proof of Theorem~\ref{equilibrium state 1}~(iii). This equality and Proposition~\ref{equilibrium state for single-valued maps2}~(ii) imply part~(iii) in Theorem~\ref{phwfq9}~(2).

In the proof of the uniqueness of the equilibrium state in Theorem~\ref{equilibrium state 1}, we have shown that if the equilibrium state for the shift map $\sigma$ and potential $\tphi$ is unique, then the equilibrium state for the correspondence $T$ and potential $\phi$ is unique in the sense of Theorem~\ref{equilibrium state 1}. The uniqueness of the equilibrium state in the setting of Theorem~\ref{phwfq9} also follows by the uniqueness of the equilibrium state for the shift map $\sigma$ and potential $\tphi$ (see Proposition~\ref{equilibrium state for single-valued maps2}~(ii)) in the same way as the proof of the uniqueness of the equilibrium state in Theorem~\ref{equilibrium state 1}.

Now we give a detailed proof of Theorem~\ref{phwfq9}~(3)(a).

\begin{proof}[Proof of Theorem~\ref{phwfq9}~(3)(a)]
    We have pointed out that $\sigma \: \cO_\omega (T) \to \cO_\omega (T)$ is an open, topologically transitive, distance-expanding continuous map, and that $\tphi \: \cO_\omega (T) \to \R$ is a H\"{o}lder continuous function. This allows us to apply Proposition~\ref{equilibrium state for single-valued maps2}~(a) to the shift map $\sigma$ and potential $\tphi$:

    For each $\udx \in \cO_\omega (T)$, 
    $\frac{1}{\sum_{\udz \in \sigma^{-n} (\udx)}   W(\udz, n) } \sum_{\udz \in \sigma^{-n} (\udx)}   \frac{\sum_{j=0}^n   \delta_{\sigma^j (\udz)}  W(\udz, n) }{n+1} \in \PPP (\cO_\omega (T))$
    converges to $\mu_\phi \cQ^\omega |_T$ in the weak* topology as $n \to +\infty$. Here $W(\udz, n) \= \exp \bigl(  \sum_{i=0}^{n-1}   \tphi (\sigma^i (\udz)) \bigr)$.

    If we consider the projection of the sequence onto the first coordinate, i.e., we consider each item composing $\tpi_1^{-1}$, then we get Theorem~\ref{phwfq9}~(3)(a).
\end{proof}

Note that by Proposition~\ref{lift of topologically exact}, we can apply Proposition~\ref{equilibrium state for single-valued maps2}~(b) to the shift map $\sigma$ and potential $\tphi$ under the assumption that $T$ is topologically exact. Then we consider the projection of the sequence of measures in Proposition~\ref{equilibrium state for single-valued maps2}~(b), and then we get Theorem~\ref{phwfq9}~(3)(b) in the same way as the proof of Theorem~\ref{phwfq9}~(3)(a).

\appendix

\section{Transition probability kernels}

In this appendix, we collect some basic properties of transition probability kernels which are standard to experts. We refer the reader to e.g., \cite[Chapter~3]{Ka21}, \cite[Section~6.1]{Le16}, and~\cite[Chapter~3]{MT12} for more background.

Throughout this appendix, let $(X,  \SAA (X))$ and $(Y,  \SAA (Y))$ be measurable spaces.

Lemmas~\ref{continuity of Q} and~\ref{l_muQ^[1]circhpi_1^-1=mu}  follow from the definition and simple calculations. %

\begin{lemma}   \label{continuity of Q}
	For $\cQ \in \tpk( Y, X )$ and $f \in \BBB(X,\R)$, if a sequence of uniformly bounded $f_n\in \BBB(X,\R)$, $n\in \N$, converges pointwise to $f$ as $n\to +\infty$, then $\cQ f_n$ converges pointwise to $\cQ f$ as $n\to +\infty$. Moreover, $\cQ f$ is measurable and $ \norm{\cQ f}_\infty \leq  \norm{f}_\infty$.
\end{lemma}

\begin{lemma}\label{l_muQ^[1]circhpi_1^-1=mu}
	For $\cQ \in \tpk( X )$ and $\mu \in \PPP(X)$, we have
	$\bigl(\mu \cQ^{\zeroto{1}}\bigr) \circ \tpi_1^{-1} =\mu$.
\end{lemma}

Lemmas~\ref{3w94g8r0q7gwd0awedcbx} and~\ref{l_Qun} are straightforward to check and standard (see  e.g., \cite[Lemma~3.3~(iii),(v)]{Ka21}).

\begin{lemma}\label{3w94g8r0q7gwd0awedcbx}
	For $\mu \in \PPP(Y)$, $\cQ \in \tpk( Y,X )$, and $f \in \BBB( X, \R)$, we have
	$\int_Y \! \cQ f  \diff\mu =\int_X \! f  \diff (\mu\cQ )$.
\end{lemma}

\begin{lemma}   \label{l_Qun}
	For $\cQ \in \tpk( X )$, $n\in \N$, and $\mu \in \cP (X)$, if $B \in \SAA \bigl( X^{n+1} \bigr)$ and $\vect{A}{0}{n}  \in (\SAA (X))^{n+1}$, then
	\begin{align}
		\bigl( \mu \cQ^{[n]} \bigr)(B) & =\int_{X^n} \! \cQ \bigl(x_{n+1} , \pi_{n+1} \bigl( \vect{x}{2}{n+1}  ;  B \bigr) \bigr)   \diff \bigl(\mu \cQ^{\zeroto{n-1}} \bigr) \bigl( \vect{x}{2}{n+1} \bigr), \label{muQ^[n](B)}  \\
		\bigl( \mu \cQ^{[n]} \bigr)(A_0 \times \cdots \times A_n) &=\int_{A_0 \times \cdots \times A_{n-1}} \! \cQ (x_n , A_n)   \diff \bigl(\mu \cQ^{\zeroto{n-1}} \bigr) ( \vect{x}{1}{n} ).            \label{muQ^{n}(A_0**A_n)}
	\end{align}
\end{lemma}

The following lemma is intuitively clear and straightforward to check. 

\begin{lemma}\label{l_QQnw}
	If $\cQ \in \tpk( X )$, $n\in \N$, $x\in X$, $A\in \SAA (X^\omega)$, $B \in \SAA (X^n)$, and $C \in \SAA (X)$, then
	\begin{align}
		\bigl(\cQ \cQ^{\zeroto{n-1}}\bigr) (x,  B) &= \cQ^{\zeroto{n}} (x,  X\times B),   \label{QQ^[n-1]=Q^[n](X*.)} \\
		\cQ (x , C) &=\cQ^{[1]} (x, X \times C),   \label{Q=Q^[1](X*.)}  \\
		(\cQ \cQ^\omega )(x,  A) &= \cQ^\omega (x,  X\times A).   \label{QQ^N(x,A)=Q^N(x,X*A)}
	\end{align}
\end{lemma}

We have the following corollary from Lemma~\ref{l_QQnw} and Definition~\ref{transition probability kernel act on measures}.

\begin{cor}  \label{c:muQ}
	If $\cQ \in \tpk( X )$, $n\in \N$, $B \in \SAA (X^n)$, and $\mu \in \cP (X)$, then
    \begin{equation}\label{muQQ^[n-1](B)=muQ^[n](X*B)}
        \bigl( \mu \cQ \cQ^{\zeroto{n-1}} \bigr) (B)= \bigl( \mu \cQ^{\zeroto{n}} \bigr) (X\times B).
    \end{equation}
    Moreover, if $A \in \SAA (X^\omega)$, then $(\mu \cQ \cQ^\omega ) (A)= (\mu \cQ^\omega) (X\times A)$. Additionally, if $\mu$ is $\cQ$-invariant, then
    \begin{equation}\label{muQ^omega(B)=muQ^omega(X*B)}
        (\mu \cQ^\omega) (A) = (\mu \cQ^\omega) (X \times A).
    \end{equation}
\end{cor}

If we take $n=1$ in (\ref{muQQ^[n-1](B)=muQ^[n](X*B)}), we get

\begin{equation}\label{muQ^[1]circhpi_2^-1=muQ}
    \bigl(\mu \cQ^{\zeroto{1}}\bigr) \circ \tpi_2 =\mu \cQ.
\end{equation}

\begin{lemma}\label{203q9j}
    Let $\cQ \in \tpk( X )$, $n\in \N_0$, and $f\: X^{n+2} \to \R$ be a measurable function. Then $\cQ^{\zeroto{n}}_y (\{ y \} \times X^n)=1$ and, for each $y\in X$,
    \begin{equation*}
    		    \int_{X^{n+1}} \! f(x_0 ,  x_0 ,  x_1 ,  \dots ,  x_n)   \diff \cQ^{\zeroto{n}}_y ( \vect{x}{0}{n}  ) 
    		    =\int_{X^{n+1}} \! f(y ,  x_0 ,  x_1 ,  \dots ,  x_n)   \diff \cQ^{\zeroto{n}}_y ( \vect{x}{0}{n} ) .
    \end{equation*}
\end{lemma}

\begin{proof}
    By Lemma~\ref{continuity of Q}, for each $y\in X$, 
    $\cQ^{\zeroto{n}}_y (\{ y \} \times X^n) =\cQ^{\zeroto{0}} (y ,  \{ y \}) =\wh{\operatorname{id}_X} (y ,  \{ y \})=1$.
    Hence,  
    \begin{align*}
        & \int_{X^{n+1}} \! f(x_0 ,  x_0 ,  \dots ,  x_n)   \diff \cQ^{\zeroto{n}}_y ( \vect{x}{0}{n} )
         =\int_{\{ y \} \times X^n} \! f(x_0 ,  x_0 ,  \dots ,  x_n)   \diff \cQ^{\zeroto{n}}_y ( \vect{x}{0}{n} )    \\
        &\qquad =\int_{\{ y \} \times X^n} \! f(y ,  x_0 ,  \dots ,  x_n)   \diff \cQ^{\zeroto{n}}_y ( \vect{x}{0}{n} ) 
         =\int_{X^{n+1}} \! f(y ,  x_0 ,  \dots ,  x_n)   \diff \cQ^{\zeroto{n}}_y ( \vect{x}{0}{n} ) . \qedhere
    \end{align*}
\end{proof}

\begin{lemma} \label{l:int_O_+infty(T)tphidmuQ^N|_T=intphidmu}
	For $\cQ \in \tpk( X )$, $\phi \in B \bigl(X^2 , \R \bigr)$, and $\mu \in \cP (X)$, we have
	\begin{equation}\label{int_O_+infty(T)tphidmuQ^N|_T=intphidmu}
		\begin{aligned}
			\int_{X^\omega} \! \phi (x_1 ,  x_2)   \diff (\mu \cQ^\omega) (\vect{x}{1}{\infty}) 
			=\int_{X^2} \! \phi   \diff \bigl(\mu \cQ^{\zeroto{1}}\bigr)
			= \int_{X} \! \int_{X} \! \phi (x_1 ,  x_2)   \diff \cQ_{x_1} (x_2)   \diff \mu (x_1).			
		\end{aligned}
	\end{equation}
\end{lemma}

\begin{proof}
	By taking $n=1$ in (\ref{muQ^N(A*X^infty)=muQ^[n](A)}), we get $\int_{X^\omega} \! \phi (x_1 ,  x_2)   \diff (\mu \cQ^\omega) ( \vect{x}{1}{\infty} ) 
	=\int_{X^2} \! \phi   \diff \bigl(\mu \cQ^{\zeroto{1}}\bigr)$.
	
	Moreover, by Lemma~\ref{3w94g8r0q7gwd0awedcbx} and Definition~\ref{transition probability kernel act on functions}, we have
	\begin{equation*}\label{3g78oybc348ec}
		\int_{X^2} \! \phi   \diff \bigl(\mu \cQ^{\zeroto{1}}\bigr) =\int_X \! \cQ^{[1]} \phi \diff \mu =\int_X \! \int_{X^2} \! \phi (x_1 ,x_2) \diff \cQ^{[1]}_y (x_1 ,x_2) \diff \mu (y)  .
	\end{equation*}
	By Lemma~\ref{203q9j} and~(\ref{Q=Q^[1](X*.)}) in Lemma~\ref{l_QQnw},
	$\int_{X^2} \! \phi (x_1 ,x_2) \diff \cQ^{[1]}_y (x_1 ,x_2) =\int_{X^2} \! \phi (y ,x_2) \diff \cQ^{[1]}_y (x_1 ,x_2) =\int_X \! \phi (y, x_2) \diff \cQ_y (x_2)$.
	Therefore,  
	$\int_{X^2} \! \phi   \diff \bigl(\mu \cQ^{\zeroto{1}}\bigr)
			= \int_{X} \! \int_{X} \! \phi (x_1 ,  x_2)   \diff \cQ_{x_1} (x_2)   \diff \mu (x_1)$.
\end{proof}

\begin{lemma}\label{A9}
	For $\cQ, \, \cR \in \tpk( X )$, and $\mu \in \cP (X)$, if $\mu \cQ^{[1]} =\bigl(\mu \cR^{[1]}\bigr) \circ \gamma_2^{-1}$, then $\mu \in \MMM (X ,\cQ) \cap \MMM (X ,\cR)$ and $\mu \cQ^{[n]} =\bigl(\mu \cR^{[n]}\bigr) \circ \gamma_{n+1}^{-1}$ for all $n\in \N$.
\end{lemma}

\begin{proof}
	Since $\mu \cQ^{[1]} =\bigl(\mu \cR^{[1]}\bigr) \circ \gamma_2^{-1}$, we have $\bigl(\mu \cQ^{[1]}\bigr) \circ \tpi_1 =\bigl(\mu \cR^{[1]}\bigr) \circ \tpi_2$ and $\bigl(\mu \cQ^{[1]}\bigr) \circ \tpi_2 =\bigl(\mu \cR^{[1]}\bigr) \circ \tpi_1$. Thus, by Lemma~\ref{l_muQ^[1]circhpi_1^-1=mu} and  (\ref{muQ^[1]circhpi_2^-1=muQ}), we have $\mu =\mu \cR$ and $\mu \cQ =\mu$, i.e., $\mu \in \MMM (X ,\cQ) \cap \MMM (X ,\cR)$.
	
	Now we prove $\mu \cQ^{[n]} =\bigl(\mu \cR^{[n]}\bigr) \circ \gamma_{n+1}^{-1}$ by induction on $n\in \N$. The case $n=1$ holds by hypothesis.
	
	Suppose $\mu \cQ^{[n-1]} =\bigl(\mu \cR^{[n-1]}\bigr) \circ \gamma_n^{-1}$ holds for some $n\geq 2$. To show $\mu \cQ^{[n]} =\bigl(\mu \cR^{[n]}\bigr) \circ \gamma_{n+1}^{-1}$, it suffices to prove $\bigl(\mu \cQ^{[n]} \bigr) (A_0 \times \cdots \times A_n)=\bigl(\mu \cR^{[n]}\bigr) (A_n \times \cdots \times A_0)$ for all $\vect{A}{0}{n}  \in (\SAA (X))^{n+1}$.
	
	Fix arbitrary $\vect{A}{0}{n}  \in (\SAA (X))^{n+1}$. Write $A_i^n \= A_i \times \cdots \times A_n$, $A_n^i \= A_n \times \cdots \times A_i$, $\udx_i^n \= (x_i , \dots ,x_n)$, and $\udx_n^i \= (x_n ,\dots ,x_i)$ for $i=0$ and $i=1$.
	\begin{alignat*}{3}
		\bigl(\mu \cR^{[n]}\bigr) (A_n^0) 
		& =\int_{A_n^1} \! \cR (x_1 ,A_0) \diff \bigl( \mu \cR^{[n-1]}\bigr) (\udx_n^1) & \text{(by (\ref{muQ^{n}(A_0**A_n)}))}\\
		& = \int_{A_1^n} \! \cR (x_1 ,A_0) \diff \bigl( \mu \cQ^{[n-1]}\bigr) (\udx_1^n) & \text{(by } \mu \cQ^{[n-1]} =\bigl(\mu \cR^{[n-1]}\bigr) \circ \gamma_{n-1}^{-1} \text{)}\\
		& = \int_X \!\ \int_{A_1^n} \! \cR_{x_1} (A_0) \diff \cQ^{[n-1]}_y (\udx_1^n) \diff \mu (y) & \text{(by Lemma~\ref{3w94g8r0q7gwd0awedcbx} and Definition~\ref{transition probability kernel act on functions})}\\
		& = \int_X \! \cR_y (A_0) \cQ^{[n-1]}_y (A_1^n) \diff \mu (y) & \text{(by Lemma~\ref{203q9j})}\\
		& = \int_{X\times A_0} \! \cQ^{[n-1]}_y (A_1^n) \diff \bigl( \mu \cR^{[1]} \bigr) (y ,x_0) & \text{(by (\ref{int_O_+infty(T)tphidmuQ^N|_T=intphidmu}))}\\
		& = \int_{A_0 \times X} \! \cQ^{[n-1]}_y (A_1^n) \diff \bigl( \mu \cQ^{[1]} \bigr) (x ,y) & \text{(by } \mu \cQ^{[1]} =\bigl(\mu \cR^{[1]}\bigr) \circ \gamma_2^{-1} \text{)}\\
		& = \int_{A_0} \! \int_X \! \cQ^{[n-1]}_y (A_1^n) \diff \cQ_x (y) \diff \mu (x) & \text{(by (\ref{int_O_+infty(T)tphidmuQ^N|_T=intphidmu}))}\\
		& = \int_{A_0} \! \bigl(\cQ \cQ^{[n-1]}\bigr) (x, A_1^n) \diff \mu (x) & \text{(by Definitions~\ref{transition probability kernel act on measures} and~\ref{compose of transition probability kernels})}\\
		& = \int_{A_0} \! \cQ^{[n]} (x, X\times A_1^n) \diff \mu (x) & \text{(by (\ref{QQ^[n-1]=Q^[n](X*.)}))}\\
		& = \int_X \! \cQ^{[n]} (x, A_0 \times A_1^n) \diff \mu (x) & \text{(by } \cQ^{[n]} (x , \{x\} \times X^n) =1 \text{ in Lemma~\ref{203q9j})}\\
		& = \bigl(\mu \cQ^{[n]}\bigr) (A_0^n) & \text{(by Definition~\ref{transition probability kernel act on measures})}.
	\end{alignat*}

	Hence, we conclude $\mu \cQ^{[n]} =\bigl(\mu \cR^{[n]}\bigr) \circ \gamma_{n+1}^{-1}$, and therefore, Lemma~\ref{A9} follows.
\end{proof}

In Sections~\ref{sct_Variational_principle_for_positively_RW-expansive_correspondences} and~\ref{sct_Thermodynamic_formalism_of_expansive_correspondences}, the following question is central: for a probability measure $\nu$ on $X^2$, how to find $\mu \in \PPP(X)$ and $\cQ\in \tpk( X )$ such that $\nu =\mu \cQ^{\zeroto{1}}$? The following proposition is standard from the \emph{disintegration} theory (cf.~\cite[Theorem~3.4]{Ka21}). 

\begin{prop}\label{decompose a two-dimentional measure}

    Let $X_1$ and $X_2$ be compact metric spaces, $M\neq \emptyset$ be a closed subset of $X_1 \times X_2$, $\nu \in \PPP (X_1 \times X_2)$ with $\nu (M)=1$, and $\kappa \: X_1 \times X_2 \to X_1$ be given by $\kappa (x_1 ,  x_2)=x_1$ for $x_1 \in X_1$ and $x_2 \in X_2$. Then there exists $\mu \in \PPP(X_1)$ and $\cQ \in \tpk( X_1, X_2 )$ with the following properties:
    \begin{itemize}
        \smallskip
        \item[(a)] $\cQ (x_1 ,  \{ x_2 \in X_2  :  (x_1 ,  x_2) \in M \}) =1$ for each $x_1 \in \kappa (M)$.
        \smallskip
        \item[(b)] $\nu (C)= \int_{X_1} \! \cQ (x_1 ,  \{ x_2 \in X_2  :  (x_1 ,  x_2) \in C \})   \diff \mu (x_1)$ for each $C\in \SBB (X_1 \times X_2)$.                
    \end{itemize}
    Moreover,  $\mu$ must be $\nu \circ \kappa^{-1}$, and $\cQ$ is unique in the sense that if both $\mu,\, \cQ$ and $\mu,\, \cQ '$ satisfy properties~(a) and~(b), then $\cQ (x_1,  B)= \cQ '(x_1,  B)$ for $\mu$-almost every $x_1 \in X_1$ and all $B \in \SBB (X_2)$.
\end{prop}

\begin{rem}\label{two equivalent conditions for the conditional transition probability kernel}
    We list three properties equivalent to property~(b) in Proposition~\ref{decompose a two-dimentional measure} for the Borel probability measure $\mu$ on $X_1$ and the transition probability kernel $\cQ$ from $X_1$ to $X_2$:
    \begin{itemize}
        \smallskip
        \item[(b1)] For each $A\in \SBB (X_1)$ and each $B\in \SBB (X_2)$, the following equality holds:
        \begin{equation}\label{nu(AtimesB)=int_AQ(x_1,B)dmu(x_1)}
            \nu (A\times B)= \int_A \! \cQ (x_1 ,  B)   \diff \mu (x_1).
        \end{equation}
        \item[(b2)] There exist some $\pi$-systems $\fA_1 \subseteq \SBB (X_1)$ and $\fA_2 \subseteq \SBB (X_2)$ with the following property:
        \begin{enumerate}
            \smallskip
            \item[(i)] The $\sigma$-algebra generated by $\fA_i$ is $\SBB (X_i)$ for each $i\in \{1, \, 2\}$.
            \smallskip
            \item[(ii)] For each $A\in \fA_1$ and each $B\in \fA_2$, the equality (\ref{nu(AtimesB)=int_AQ(x_1,B)dmu(x_1)}) holds.
        \end{enumerate}
        \smallskip
        \item[(b3)] For each lower bounded Borel measurable function $f\: X_1 \times X_2 \to \R \cup \{+\infty \}$, we have
        \begin{equation}\label{int_X_1*X_2f(x_1,x_2)dnu(x_1,x_2)=int_X_1(int_X_2f(x_1,x_2)dQ_x_1(x_2))dmu(x_1)}
        \int_{X_1 \times X_2} \! f (x_1 ,  x_2)   \diff \nu (x_1 ,  x_2) = \int_{X_1} \! \biggl( \int_{X_2} \! f (x_1 ,  x_2)   \diff \cQ_{x_1} (x_2) \biggr)   \diff \mu (x_1).
        \end{equation}
    \end{itemize}

    The equivalence of properties~(b), (b1), and~(b2) can be verified by Dynkin's $\pi$-$\lambda$ theorem. Clearly (b3) implies (b). We explain why (b) implies (b3):

    Suppose (b) holds for $\mu$ and $\cQ$. Property~(b) implies that (\ref{int_X_1*X_2f(x_1,x_2)dnu(x_1,x_2)=int_X_1(int_X_2f(x_1,x_2)dQ_x_1(x_2))dmu(x_1)}) holds when $f$ is a characteristic function of an arbitrary Borel subset of $X_1 \times X_2$, and thus by Lemma~\ref{continuity of Q},  (\ref{int_X_1*X_2f(x_1,x_2)dnu(x_1,x_2)=int_X_1(int_X_2f(x_1,x_2)dQ_x_1(x_2))dmu(x_1)}) holds when $f$ is an arbitrary simple function on $X_1 \times X_2$. Because each lower bounded Borel measurable function on $X_1 \times X_2$ can be pointwise approached by an increasing sequence of bounded simple functions,  (\ref{int_X_1*X_2f(x_1,x_2)dnu(x_1,x_2)=int_X_1(int_X_2f(x_1,x_2)dQ_x_1(x_2))dmu(x_1)}) holds for all lower bounded Borel measurable functions $f \: X_1 \times X_2 \to \R \cup \{+\infty\}$.
\end{rem}

Let $X_1$ and $X_2$ be compact metric spaces, $M \neq \emptyset$ be a closed subset of $X_1 \times X_2$, and $\nu \in \PPP(X_1 \times X_2)$ be supported on $M$. Denote by $\kappa_1 \: X_1 \times X_2 \to X_1$ and $\kappa_2 \: X_1 \times X_2 \to X_2$ the projection maps given by $\kappa_1 (x_1 ,  x_2)=x_1$ and $\kappa_2 (x_1 ,  x_2)=x_2$, respectively, for all $x_1 \in X_1$ and $x_2 \in X_2$. Proposition~\ref{decompose a two-dimentional measure} ensures the notions defined in the following two definitions always exist.

\begin{definition}
    If $\cQ\in \tpk( X_1, X_2 )$ and the Borel probability measure $\mu =\nu \circ \kappa_1^{-1}$ on $X_1$ satisfy the two properties~(a) and~(b) in Proposition~\ref{decompose a two-dimentional measure}, then $\cQ$ is called a \defn{forward conditional transition probability kernel of $\nu$ from $X_1$ to $X_2$ supported on $M$}\index{forward conditional transition probability kernel}.
\end{definition}

\begin{definition}\label{backward conditional transition probability kernel}
    A transition probability kernel $\cQ \in \tpk( X_2, X_1 )$ is called a \defn{backward conditional transition probability kernel of $\nu$ from $X_2$ to $X_1$ supported on $M$}\index{backward conditional transition probability kernel} if it satisfies the following properties:
    \begin{itemize}
        \smallskip
        \item[(a)] $\cQ (x_2 ,  \{ x_1 \in X_1  :  (x_1 ,  x_2) \in M \}) =1$ for each $x_2 \in \kappa_2 (M)$.
        \smallskip
        \item[(b)] $\nu (C)= \int_{X_2} \! \cQ (x_2 ,  \{ x_1 \in X_2  :  (x_1 ,  x_2) \in C \})   \diff \bigl(\nu \circ \kappa_2^{-1}\bigr) (x_2)$ for each $C\in \SBB (X_1 \times X_2)$.
    \end{itemize}
\end{definition}

\begin{rem}\label{nu=muQ^[1]}
	If $X_1 =X_2$, then by (\ref{int_O_+infty(T)tphidmuQ^N|_T=intphidmu}) (in the case where $\phi$ is a characteristic function of a measurable subset of $X^2$), property~(b) in Proposition~\ref{decompose a two-dimentional measure} is equivalent to $\nu = \mu \cQ^{[1]}$. Similarly, (b) in Definition~\ref{backward conditional transition probability kernel} is equivalent to $\nu \circ \gamma_2^{-1} =\bigl(\nu \circ \kappa_2^{-1} \bigr) \cQ^{[1]}$, where $\gamma_2 (x ,y) =(y, x)$ for all $(x ,y) \in X^2$.
\end{rem}

\printindex


\begin{thebibliography}{1024}
%

\bibitem[AF90]{AF90}
\textsc{Aubin,~J.-P.} and \textsc{Frankowska,~H.},
\textit{Set-valued analysis},
Birkh\"auser, Boston, %
1990.%

\bibitem[AFL91]{AFL91}
\textsc{Aubin,~J.-P.}, \textsc{Frankowska,~H.}, and \textsc{Lasota,~A.},
Poincar\'e's recurrence theorem for set-valued dynamical systems.
\textit{Ann.\ Polon.\ Math.} 54 (1991), 85--91.%

\bibitem[BM10]{BM10} 
\textsc{Bonk,~M.} and \textsc{Meyer,~D.},
Expanding Thurston maps. Preprint, (arXiv:1009.3647v1), 2010.

\bibitem[BM17]{BM17} 
\textsc{Bonk,~M.} and \textsc{Meyer,~D.},
\textit{Expanding Thurston maps}, volume 225 of \textit{Math.\ Surveys Monogr.}, Amer.\ Math.\ Soc., Providence, RI, 2017.%


%
%
%
%

\bibitem[Bow75]{Bow75}
\textsc{Bowen,~R.},
\textit{Equilibrium states and the ergodic theory of Anosov diffeomorphisms},
volume~470 of \textit{Lecture Notes in Math.}, Springer, Berlin, 1975.%


\bibitem[Bu00]{Bu00}
\textsc{Bullett,~S.},
A combination theorem for covering correspondences and an application to mating polynomial maps with Kleinian groups.
\textit{Conform.\ Geom.\ Dyn.} 4 (2000), 75--96.%


\bibitem[BF05]{BF05}
\textsc{Bullett,~S.} and \textsc{Freiberger,~M.},
Holomorphic correspondences mating Chebyshev-like maps with Hecke groups.
\textit{Ergodic Theory Dynam.\ Systems} 25 (2005), 1057--1090.%

\bibitem[BH07]{BH07}
\textsc{Bullett,~S.} and \textsc{Ha\"{i}ssinsky,~P.},
Pinching holomorphic correspondences.
\textit{Conform.\ Geom.\ Dyn.} 11 (2007), 65--89.

\bibitem[BL20]{BL20}
\textsc{Bullett,~S.} and \textsc{Lomonaco,~L.},
Mating quadratic maps with the modular group II.
\textit{Invent.\ Math.} 220 (2020), 185--210.%

\bibitem[BL22]{BL22}
\textsc{Bullett,~S.} and \textsc{Lomonaco,~L.},
Dynamics of modular matings.
\textit{Adv.\ Math.} 410 (2022), 108758.

\bibitem[BL24]{BL24}
\textsc{Bullett,~S.} and \textsc{Lomonaco,~L.},
Mating quadratic maps with the modular group III: The modular Mandelbrot set.
\textit{Adv.\ Math.} 458 (2024), 109956.

\bibitem[BP94]{BP94}
\textsc{Bullett,~S.} and \textsc{Penrose,~C.},
Mating quadratic maps with the modular group.
\textit{Invent.\ Math.} 115 (1994), 483--511.%

%
%
%
%

\bibitem[BP01]{BP01}
\textsc{Bullett,~S.} and \textsc{Penrose,~C.},
Regular and limit sets for holomorphic correspondences.
\textit{Fund.\ Math.} 167 (2001), 111--171.

\bibitem[CPMP08]{CPMP08}
\textsc{Chinchuluun,~A.}, \textsc{Pardalos,~P.M.}, \textsc{Migdalas,~A.}, and \textsc{Pitsoulis,~L.},
\textit{Pareto optimality, game theory and equilibria},
Springer, New York, %
2008.%


\bibitem[CP16]{CP16}
\textsc{Cordeiro,~W.} and \textsc{Pac\'ifico,~M.J.},
Continuum-wise expansiveness and specification for set-valued functions and topological entropy.
\textit{Proc.\ Amer.\ Math.\ Soc.} 144 (2016), 4261--4271.%


\bibitem[DKW20]{DKW20}
\textsc{Dinh,~T.C.}, \textsc{Kaufmann,~L.}, and \textsc{Wu,~Hao},
Dynamics of holomorphic correspondences on Riemann surfaces.
\textit{Internat.\ J.\ Math.} 31 (2020), 2050036.

\bibitem[Do68]{Do68}
\textsc{Dobruschin,~R.L.},
The description of a random field by means of conditional probabilities and conditions for its regularity.
\textit{Theory Probab.\ Appl.} 13 (1968), 197--224.

\bibitem[Fa29]{Fa29}
\textsc{Fatou,~P.},
Notice sur les travaux scientifiques de MP Fatou, 1929.

%
%
%
%

\bibitem[Fo88]{Fo88}
\textsc{Forbus,~K.D.},
Chapter~7 -- \textit{Qualitative Physics: Past, Present, and Future},
Exploring Artificial Intelligence Survey Talks from the National Conferences on Artificial Intelligence 1988, Pages 239--296.


%
%
%
%


\bibitem[Ha02]{Ha02}
\textsc{Hadamard,~J.},
Sur les probl\`{e}mes aux d\'{e}riv\'{e}es partielles et leur signification physique.
\textit{Princeton University Bulletin} 13 (1902), 49--52.

\bibitem[HP09]{HP09}
\textsc{Ha\"{\i}ssinsky,~P.} and \textsc{Pilgrim,~K.M.},
Coarse expanding conformal dynamics.
\textit{Ast\'{e}risque} 325 (2009).

\bibitem[IM06]{IM06}
\textsc{Ingram,~W.T.} and \textsc{Mahavier,~W.S.},
Inverse limits of upper semi-continuous set valued functions (English summary).
\textit{Houston J.\ Math.} 32 (2006), 119--130.

\bibitem[Ka21]{Ka21}
\textsc{Kallenberg,~A.},
\textit{Foundations of modern probability}, 3rd ed.,
volume~99 of \textit{Probab.\ Theory Stoch.\ Model.},
Springer, Cham, 2021.

\bibitem[KT17]{KT17}
\textsc{Kelly,~J.P.} and \textsc{Tennant,~T.},
Topological entropy of set-valued functions.
\textit{Houston J.\ Math.} 43 (2017), 263--282.


%
%
%
%

\bibitem[Ku58]{Ku58}
\textsc{Kuratowski,~K.},
\textit{Topologie, volume~I},
PWN---Polish Scientific Publishers, Warsaw, 1958.


\bibitem[Le16]{Le16}
\textsc{Le Gall,~J.-F.},
\textit{Brownian motion, martingales, and stochastic calculus},
Springer ,%
Cham, 2016.%

\bibitem[LLMM21]{LLMM21}
\textsc{Lee,~S.}, \textsc{Lyubich,~M.Yu.}, \textsc{Makarov,~N.G.}, and \textsc{Mukherjee,~S.},
Schwarz reflections and anti-holomorphic correspondences.
\textit{Adv.\ Math.} 385 (2021), 107766.


\bibitem[LMM24]{LMM24}
\textsc{Lyubich,~M.Yu.}, \textsc{Mazor,~J.}, and \textsc{Mukherjee,~S.},
Antiholomorphic correspondences and mating I: Realization theorems.
\textit{Comm.\ Amer.\ Math.\ Soc.} 4 (2024), 495--547.

\bibitem[LM97]{LM97}
\textsc{Lyubich,~M.Yu.} and \textsc{Minsky,~Y.},
Laminations in holomorphic dynamics.
\textit{J.\ Differential Geom.} 47 (1997), 17--94.

\bibitem[Ma23a]{Ma23a}
\textsc{Matus~de~la~Parra,~V.},
Equidistribution for matings of quadratic maps with the modular group.
\textit{Ergodic Theory Dynam.\ Systems} (2023), 1--29.

\bibitem[Ma23b]{Ma23b}
\textsc{Matus~de~la~Parra,~V.},
Entropy of compositions of covering correspondences.
Preprint, (arXiv:2310.14330), 2023.


\bibitem[Mc95]{Mc95}
\textsc{McMullen,~C.T.},
The classification of conformal dynamical systems.
In \textit{Current developments in mathematics}, S.T.~Yau et al.~(ed.), 
Cambridge, MA 1995, pp.~323--360.

\bibitem[Mc96]{Mc96}
\textsc{McMullen,~C.T.},
\textit{Renormalization and 3-manifolds which fiber over the circle},
Princeton Univ.\  Press, 1996.

\bibitem[MT12]{MT12}
\textsc{Meyn,~S.} and \textsc{Tweedie,~R.L.},
\textit{Markov chains and stochastic stability}, 2nd ed.,
Springer, London, %
2012.%

\bibitem[Mic56a]{Mic561}
\textsc{Michael,~E.},
Continuous selections I.
\textit{Ann.\ of Math.\ (2)} 63 (1956), 361--382.%

\bibitem[Mic56b]{Mic562}
\textsc{Michael,~E.},
Continuous selections II.
\textit{Ann.\ of Math.\ (2)} 64 (1956), 562--580.%

\bibitem[Mic57]{Mic57}
\textsc{Michael,~E.},
Continuous selections III.
\textit{Ann.\ of Math.\ (2)} 65 (1957), 375--390.%

\bibitem[Mil95]{Mil95}
\textsc{Miller,~W.M.},
Frobenius-Perron operators and approximation of invariant measures for set-valued dynamical systems.
\textit{Set-Valued Var.\ Anal.} 3 (1995), 181--194.


\bibitem[MA99]{MA99}
\textsc{Miller,~W.} and \textsc{Akin,~E.},
Invariant measures for set-valued dynamical systems.
\textit{Trans.\ Amer.\ Math.\ Soc.} 351 (1999), 1203--1225.%


\bibitem[MM23]{MM23}
\textsc{Mj,~M.} and \textsc{Mukherjee,~S.},
Combining rational maps and Kleinian groups via orbit equivalence.
\textit{Proc.\ Lond.\ Math.\ Soc.\ (3)} 126 (2023), 1740--1809.

\bibitem[PV17]{PV17}
\textsc{Pac\'ifico,~M.J.} and \textsc{Vieitez,~J.L.},
Expansiveness, Lyapunov exponents and entropy for set valued maps.
Preprint, (arXiv:1709.05739), 2017.%


\bibitem[Pa64]{Pa64}
\textsc{Parry,~W.},
Intrinsic Markov chains.
\textit{Trans.\ Amer.\ Math.\ Soc.} 112 (1964), 55--66.%

\bibitem[Pe93]{Pe93}
\textsc{Petrosyan,~L.A.},
\textit{Differential games of pursuit},
World Scientific, Singapore, 1993.

%
%
%
%

\bibitem[Po21]{Po21}
\textsc{Pommaret,~J.-F.},
Differential correspondences and control theory.
\textit{Advances in Pure Mathematics} 11 (2021), 835--882.

\bibitem[PU10]{PU10}
\textsc{Przytycki,~F.} and \textsc{Urba\'{n}ski,~M.},
\textit{Conformal fractals: ergodic theory methods},
Cambridge Univ.\ Press, Cambridge, 2010.%


\bibitem[RT18]{RT18}
\textsc{Raines,~E.} and \textsc{Tennant,~T.},
The specification property on a set-valued map and its inverse limit.
\textit{Houston J.\ Math.} 44 (2018), 665--677.


%
%
%
%

\bibitem[Ru78]{Ru78}
\textsc{Ruelle,~D.},
\textit{Thermodynamic formalism},
Addison-Wesley, Reading, MA, 1978.

\bibitem[Ru92]{Ru92}
\textsc{Ruelle,~D.},
Thermodynamic formalism of maps satisfying positive expansiveness and specification.
\textit{Nonlinearity} 5 (1992), 1223--1236.%

%
%
%
%

\bibitem[Sin72]{Sin72}
\textsc{Sinai,~Ya.G.},
Gibbs measures in ergodic theory.
\textit{Russian Math.\ Surveys} 27 (1972), 21--69.

\bibitem[Siq15]{Siq15}
\textsc{Siqueira,~C.},
Dynamics of holomorphic correspondences. Ph.D.\ thesis, University of S\~ao Paulo, 2015.%

\bibitem[Siq22]{Siq22}
\textsc{Siqueira,~C.},
Dynamics of hyperbolic correspondences.
\textit{Ergodic Theory Dynam.\ Systems} 42 (2022), 2661--2692.%

\bibitem[Siq23]{Siq23}
\textsc{Siqueira,~C.},
Hausdorff dimension of Julia sets of unicritical correspondences.
\textit{Proc.\ Amer.\ Math.\ Soc.} 151 (2023), 633--645.%

\bibitem[SS17]{SS17}
\textsc{Siqueira,~C.} and \textsc{Smania,~D.},
Holomorphic motions for unicritical correspondences.
\textit{Nonlinearity} 30 (2017), 3104--3125.%

\bibitem[Su85]{Su85}
\textsc{Sullivan,~D.P.},
Quasiconformal homeomorphisms and dynamics I: Solution of the Fatou-Julia problem on wandering domains.
\textit{Ann.\ of Math.\ (2)} 122 (1985), 401--418.%

\bibitem[VS22]{VS22}
\textsc{Vivas,~K.J.} and \textsc{Sirvent,~V.F.},
Metric entropy for set-valued maps.
\textit{Discrete Contin.\ Dyn.\ Syst.\ Ser.\ B} 27 (2022), 6589--6604.

\bibitem[Wa76]{Wa76}
\textsc{Walters,~P.},
A variational principle for the pressure of continuous transformations.
\textit{Amer.\ J.\ Math.} 17 (1976), 937--971.%

\bibitem[Wa82]{Wa82}
\textsc{Walters,~P.},
\textit{An introduction to ergodic theory},
Springer, New York, 1982.%

\bibitem[Wi70]{Wi70}
\textsc{Williams,~R.K.},
Some results on expansive mappings. 
\textit{Proc.\ Amer.\ Math.\ Soc.} 26 (1970), 655--663.%

\bibitem[Wu20]{Wu20}
\textsc{Wu,~Hao},
Dynamics of holomorphic maps and correspondences. Ph.D.\ thesis, National University of Singapore, 2020.%

%

%




\end{thebibliography}
\end{document}